\documentclass{article}
\usepackage{amssymb, amsthm, amsmath}
\usepackage{bm}
\usepackage[dvipsnames]{xcolor}
\definecolor{MutedIndigo}{HTML}{B8A9C9}
\definecolor{SageGreen}{HTML}{7E9F8A}
\definecolor{DustyRose}{HTML}{D8A1A9}
\definecolor{DustyRed}{HTML}{B8445A}
\definecolor{MutedLavender}{HTML}{847FC6}
\definecolor{SlateBlue}{HTML}{5F8FC9}
\definecolor{EggShell}{RGB}{244, 241, 222}
\definecolor{Graphite}{HTML}{7A6F82}
\definecolor{DarkGraphite}{HTML}{4A4E57}
\definecolor{espresso}{HTML}{3A2A24}
\definecolor{DirtyYellow}{HTML}{C6B44A}
\definecolor{OliveDrab}{HTML}{6B8E23}
\usepackage{tikz}
\usetikzlibrary{decorations.markings}
\usetikzlibrary{shapes.geometric,positioning}
\usetikzlibrary{arrows,matrix, fit, backgrounds}
\usetikzlibrary{decorations.pathmorphing, patterns,shapes, trees}
 \usetikzlibrary{calc}
 \tikzset{snake it/.style={decorate, decoration=snake}}
 \usetikzlibrary{decorations.pathreplacing}

\usepackage{hyperref}

\usepackage{todonotes}

\usepackage{stmaryrd}
\usepackage{pifont}
\theoremstyle{plain}
\newtheorem{thm}{Theorem}[section]
\newtheorem{theorem}[thm]{Theorem}
\newtheorem{lemma}[thm]{Lemma}

\newtheorem{proposition}[thm]{Proposition}

\newtheorem{prop}[thm]{Proposition}

\newtheorem*{claim*}{Claim}
\newtheorem*{thm*}{Theorem}

\newcommand{\incs}{\subsetneq}
 \usepackage{float }
 \usepackage{caption}
\theoremstyle{definition}
\newtheorem{definition}[thm]{Definition}
\newtheorem{convention}[thm]{Convention}

\newtheorem{remark}[thm]{Remark}

\newtheorem{example}[thm]{\sc Example}

\newcommand{\inc}{\subseteq}
\newcommand{\union}{\cup}	
\newcommand{\Union}{\bigcup}	
\newcommand{\set}[1]{\{#1\}}
\newcommand{\setc}[2]{\set{#1 \mid #2}}
\newcommand{\inter}{\cap}
\newcommand{\hyper}[1]{{\bf #1}}
\newcommand{\restrconstr}[2]{{#1}_{\lceil_{#2}}}
\newcommand{\Invt}{\widetilde{\mathrm{Inv}}}
\newcommand{\Inv}{{\mathrm{Inv}}}
 
\newcommand{\clandec}{\underline{\leadsto}}

\newcommand{\nin}{\not\in}
\newcommand{\restrH}[2]{\hyper{#1}\backslash #2}

 \newcommand{\squash}{<\!\!\!\!\lessdot\,\,}
\usepackage{centernot}
\usepackage[normalem]{ulem}
\newcommand{\subsetdot}{{\subset\!\!\!\!\cdot\,}}

\def\mathraise#1#2#3#4{
\mathchoice
{\raisebox{#2}{$\displaystyle{#1}$}}
{\raisebox{#2}{$#1$}}
{\raisebox{#3}{$\scriptstyle{#1}$}}
{\raisebox{#4}{$\scriptscriptstyle{#1}$}}
}

\def\mathraiseord#1#2#3#4{\mathord{\mathraise#1{#2}{#3}{#4}}}
\def\valup{\mathraiseord{\uparrow}{0.29ex}{0.21ex}{0.16ex}}

\newcommand\varlist{\makebox[1em][c]{.\hfil.\hfil.}}

\usepackage[style=numeric, backend=biber, maxcitenames=5, maxbibnames=5]{biblatex}
\title{Generalised flip order on the faces of nestohedra}

 \usepackage{authblk}

\author[$^\dagger$]{ Pierre-Louis Curien}
\affil[$^\dagger$]{{\footnotesize Université Paris Cité, CNRS, Inria, IRIF, Picube project-team, Paris, France}}

\author[$^\ast$]{Bérénice Delcroix-Oger}
\affil[$^\ast$]{\footnotesize Université de Montpellier, CNRS, IMAG, Montpellier, France}

\author[$^\ddagger$]{Jovana Obradovi\' c} 
\affil[$^\ddagger$]{\footnotesize Mathematical Institute of the Serbian Academy of Sciences and Arts, Belgrade, Serbia}

\begin{document} 	\maketitle

\begin{abstract}
\noindent 
Tamari lattices and weak Bruhat order are classical posets on the vertices of associahedra and permutohedra. These posets are particular instances of the \emph{flip order} on the vertices of graph associahedra. We present here a   generalisation of the flip order  to all the faces of   nestohedra, encoded in terms of hypergraphs. For right-filled hypergraphs, we characterise our order  in terms of generalised inversions. As an application, we  describe shuffle products on faces of suitable nestohedra as sums over intervals in this order.   
\end{abstract}

\section{Introduction} \label{intro-section}
\noindent In this article we investigate a binary relation on the set of all faces of a fixed nestohedron (a.k.a. hypergraph polytope) that was previously introduced by Curien and Laplante-Anfossi in \cite{CLA2}. Our first result is that this relation generates a partial order on the faces of a nestohedron. We show that, when restricted to faces of dimension 0 (i.e., to vertices), this order coincides with the flip order found in the work of  Barnard and McConville \cite{BM}. For this reason, we refer to the order studied here as the {\em generalised flip order}. The flip order in turn generalises both the classical Tamari order on planar binary trees and the weak Bruhat order on permutations.

We note that the flip order of \cite{BM} was defined for graph associahedra, but its definition and acyclicity proof extend straightforwardly to arbitrary nestohedra. The genuinely new feature of the present work is the extension of this order from vertices to {\em all} faces. 

The generalised flip order is thus  a generalisation of the above classical orders on planar binary trees and on permutations in two
 directions: to other polytopes on one hand, and to all the faces of these polytopes (rather than on their vertices only) on the other hand. An extension of the second kind had been proposed already in \cite{KLNPS} for permutohedra and in \cite{PalaciosRonco} for both associahedra and permutohedra, whose faces are described by planar trees and by surjections, respectively. Our order is defined for any hypergraph polytope and instantiates precisely to those early appearances (see Remark \ref{extends-PR-K}).

\smallskip

Our motivation is partly algebraic. In earlier work \cite{PLBJ1}, we introduced $q$-tridendriform structures on the faces of certain families of 
hypergraph polytopes. The operations defining these structures assemble into shuffle products that generalise the classical shuffle products on surjections and planar rooted trees. In the present paper we show that these shuffle products can be characterised as sums over suitable intervals in the generalised flip order. This generalises results \cite{PalaciosRonco} of Palacios and Ronco for permutohedra and associahedra. 

\smallskip
But we also believe that the generalised flip order stands out on its own and deserves further study. In particular, we managed, under a suitable right-filledness condition, to characterise it in terms of generalised inversions, extending in a highly non-trivial way both the condition and the characterisation of the flip order given in \cite{BM}.

\medskip
 
In order to remain as self-contained as possible, we shall recall the necessary definitions on nestohedra, as well as the main definitions and results of \cite{PLBJ1} concerning the shuffle products on faces of nestohedra, in  \S \ref{nestohedra-reminders-section} and \S \ref{comb-subsection}.
 In \S \ref{GFO-section}, we recall the generalised flip relation from \cite{CLA2} (which was given no name there) and show that it generates a partial order (Proposition \ref{lessdot-order}) whose restriction to the 0-dimensional faces coincides with the flip order as defined by Barnard and McConville (Proposition \ref{GFO-FO}). In \S \ref{inversion-subsection}, we also provide a characterisation of the generalised flip order in terms of generalised inversions. In \S \ref{shuffle-order-section}, we state and prove our characterisation of the operations of \S  \ref{comb-subsection} in terms of intervals. 
The final \S \ref{comparison-section} is a discussion section. In particular, we offer
 some comparisons with other orders that have been introduced on all faces of nestohedra.

\medskip

\noindent {\bf Notation}.
 Throughout the paper, for a finite set $X$ endowed with a total order and subsets $Y,Z\inc X$ (which are total orders themselves for the induced ordering), we shall write $Y <Z$ (or $Z>Y$) to denote that $\max(Y)<\min(Z)$. We also write $|X|$ for the cardinality of $X$.  
\paragraph{Acknowledgements.}{The authors would like to express their gratitude to Mar\'{\i}a Ronco and Vincent Pilaud for insightful discussions regarding various poset structures on faces of various classes of polytopes.}

\section{Reminders on nestohedra} \label{nestohedra-reminders-section}

We recall from \cite{COI} and \cite{DP} the combinatorial description of nestohedra in terms of hypergraphs and their associated tree-like structures, called constructs.

\subsection{Hypergraphs} \label{hypergraphs-subsection}
A hypergraph consists of a set $H$ of {\em vertices} and a subset 
$\hyper{H}\inc {\mathcal{P}}(H)\backslash\emptyset$, whose elements are called {\em hyperedges}, such that $\Union \hyper{H}=H$. In addition, we assume that each hypergraph is {\em atomic}, meaning that 
 $\set{x}\in \hyper{H}$, for all $x\in H$. By the atomicity assumption, the vertices of a hypergraph can be seen as its hyperedges of 
 cardinality $1$. This justifies the convention, which we take from now on, to denote with boldface letter ${\bf H}$ both an entire hypergraph and its set of hyperedges, while $H$ refers to the set of vertices (or singleton hyperedges) of ${\bf H}$.

\smallskip

 A hyperedge of cardinality 2 (of a hypergraph ${\bf H}$) is called an {\em edge} (of ${\bf H}$). Every ordinary graph $(V,E)$ can be viewed as the atomic hypergraph
$\setc{\set{v}}{v\in V} \union \setc{e}{e\in E}$ (with no hyperedges of cardinality $\geq 3$).

\begin{example}\label{hypergraphex1} The hypergraph 
 $${\bf H}^\maltese=\{\{1\},\{2\},\{3\},\{4\},\{3,4\},\{2,3,4\},\{1,2,3\}\}$$ can be represented pictorially as follows, where the hyperedge $\{1,2,3\}$ (resp. $\{2,3,4\}$) is represented by the enclosed region around the vertices $1$, $2$ and $3$ (resp. $2$, $3$ and $4$):
\vspace{-.2cm}	\begin{center}
		\resizebox{1.75cm}{!}{\begin{tikzpicture}
			\draw [rounded corners=5mm,draw=black,fill=DustyRose,opacity=0.3] (-0.35,1.35)--(1.6,1.2)--(-0.2,-0.6)--cycle;
			\draw [rounded corners=5mm,draw=black,fill=OliveDrab,opacity=0.3] (1.35,-0.35)--(1.2,1.6)--(-0.6,-0.2)--cycle;
			\draw [rounded corners=5mm,draw=black] (-0.35,1.35)--(1.6,1.2)--(-0.2,-0.6)--cycle;
			\draw [rounded corners=5mm,draw=black] (1.35,-0.35)--(1.2,1.6)--(-0.6,-0.2)--cycle;
			\node (x) [circle,draw=none,inner sep=0mm,minimum size=1.3mm] at (0,1) {\small $1$};
			\node (u) [circle,draw=none,inner sep=0mm,minimum size=1.3mm] at (1,1) {\small $2$};
			\node (y) [circle,draw=none,inner sep=0mm,minimum size=1.3mm] at (0,0) {\small $3$};
			\node (z) [circle,draw=none, inner sep=0mm,minimum size=1.3mm] at (1,0) {\small $4$};
			\draw (y)--(z);
		\end{tikzpicture}}
	\end{center}  
\end{example}

For a hypergraph $\hyper{H}$ and a subset $X\inc H$, we define the {\em restriction of ${\bf H}$ to $X$} as the hypergraph
$\hyper{H}_X:=\setc{Z}{Z\in \hyper{H}\;\mbox{and}\; Z\inc X}$. We introduce the shorthand $\restrH{H}{X}:=\hyper{H}_{H\backslash X}$ for the restriction of ${\bf H}$ to $H\backslash X$.
We say that $\hyper{H}$ is {\em connected} if there is no non-trivial partition $H=X_1\union X_2$ such that $\hyper{H}=\hyper{H}_{X_1}\union \hyper{H}_{X_2}$, and that $X\inc H$ is connected in, or is a {\em tube} of $\hyper{H}$, if $\hyper{H}_X$ is connected, and then we often write $\hyper{X}=\hyper{H}_X$.
 
\smallskip

Every finite hypergraph admits a unique partition
$H=X_1\union\cdots\union X_m$ such that each $\hyper{H}_{X_i}$ is connected and $\hyper{H}=\Union(\hyper{H}_{X_i})$. The $\hyper{H}_{X_i}$'s (or simply the $X_i$'s) are called the {\em connected components} of $\hyper{H}$. The notation
$\hyper{H},X \leadsto H_1,\ldots, H_n$
 will mean that $\hyper{H}_1,\ldots,\hyper{H}_n$ are the
 connected components of $\restrH{H}{X}$.

\subsection{Constructs} \label{constructs-subsection}
As Do\v sen and Petri\'c show in \cite{DP}, the non-trivial connected subsets of a finite connected hypergraph $\hyper{H}$ can be interpreted as instructions for truncating a simplex whose vertex count equals $|H|$. The resulting polytope is referred to as the hypergraph polytope (or nestohedron) linked to $\hyper{H}$; whenever $\hyper{H}$ reduces to an ordinary graph, this construction recovers the extensively investigated notion of graph associahedra. The faces of the hypergraph polytope associated to ${\bf H}$ can be nicely described as
non-planar trees, called {\em constructs}, whose nodes are decorated by non-empty subsets of $H$ and whose recursive definition is given next, using the syntax introduced in \cite{COI}. 

\smallskip

Pick a non-empty subset $\emptyset\neq Y \subsetneq H$. Suppose that $\hyper{H},Y \leadsto H_1,\ldots, H_n$. If $T_1, \dots, T_n$ are constructs of the connected components $H_1, \dots, H_n$ of $\hyper{H}\backslash Y$, respectively, then the tree obtained by grafting $T_1,\ldots,T_n$ on the root node decorated by $Y$, denoted by $Y(T_1,\ldots,T_n)$ (or by $Y\setc{T_i}{1\leq i\leq n}$), and drawn as 
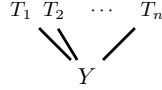
\begin{center}
 \resizebox{!}{1.4cm}{\begin{tikzpicture}
\node (Y) at (0,-1) { $Y$};
 \node (T1) [circle,draw=none,minimum size=0.75cm,inner sep=0mm] at (-1,0) {\small $T_{1}$};
 \node (T2) [rectangle,draw=none,minimum size=0.75cm,inner sep=-0.5mm] at (-0.5,0) {\small $T_{2}$};
 \node (d2) [circle,draw=none,minimum size=0.75cm,inner sep=0mm] at (0.25,0) {\footnotesize $\cdots$};
 \node (Tn) [circle,draw=none,minimum size=0.75cm,inner sep=0mm] at (1,0) {\small $T_{n}$};
\draw[very thick] (Y)--(T1);
\draw[very thick] (Y)--(-0.44,-0.28);
\draw[very thick] (Y)--(Tn);
 \end{tikzpicture}}
\end{center} is a construct of $\hyper{H}$. 
 The base case is when $Y=H$ (and hence $n=0$): then the one-node tree $H()$, written simply $H$, and 
drawn as
\begin{center}
 \resizebox{!}{0.5cm}{\begin{tikzpicture}
\node (Y) [circle,minimum size=0.17cm,inner sep=0.3mm] at (0,-0.9) {${H}$};
 \end{tikzpicture}}
\end{center}
 is a construct of $\hyper{H}$.

\smallskip

We use the notation $T:\hyper{H}$ to indicate that $T$ is a construct of $\hyper{H}$. If $T=Y(T_1,\ldots,T_n)$ (including the possibility that $Y=H$), we write $Y=\mbox{root}(T)$.

  Subconstructs of $T:\hyper{H}$ are defined as follows: $T$ is a subconstruct of itself, and if $T=Y(T_1,\ldots,T_n)$, with $\hyper{H},Y\leadsto H_1,\ldots, H_n$, then a proper subconstruct of $T$ is a subconstruct of one of the $T_i$. In particular, $T_1,\ldots,T_n$ are called the immediate subconstructs of $T$.

\begin{convention}\label{con1}
To shorten the notation, we shall represent singleton vertices of constructs without the braces. For example, the constructs $\{x\}(\{u,v\},\{y\})$ and $\{x\}(\{u\}(\{v\}),\{y\})$ of the hypergraph $${\bf H}=\{\{x\},\{y\},\{u\},\{v\},\{x,y\},\{x,v\},\{u,v\}\}$$ will be denoted by $x(\{u,v\},y)$ and $x(u(v),y)$ and drawn as 
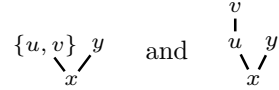
\begin{center}
\raisebox{-1.2em}{\resizebox{1.3cm}{!}{\begin{tikzpicture}
\node (b) [circle,minimum size=0.1cm,inner sep=0.2mm] at (4,-0.2) {\scriptsize $x$};
\node (c) [rectangle,minimum size=0.1cm,inner sep=0.1mm] at (3.7,0.2) {\scriptsize $\{u,v\}$};
 \node (d) [circle,minimum size=0.1cm,inner sep=0.2mm] at (4.3,0.2) {\scriptsize $y$};
 \draw[thick] (c)--(b)--(d);\end{tikzpicture}}} \quad and \quad\raisebox{-1.2em}{\resizebox{0.8cm}{!}{\begin{tikzpicture}
\node (b) [circle,minimum size=0.1cm,inner sep=0.2mm] at (4,-0.2) {\scriptsize $x$};
\node (c) [circle,minimum size=0.1cm,inner sep=0.2mm] at (4.2,0.2) {\scriptsize $y$};
 \node (d) [circle,minimum size=0.1cm,inner sep=0.2mm] at (3.8,0.2) {\scriptsize $u$};
 \node (a) [circle,minimum size=0.1cm,inner sep=0.2mm] at (3.8,0.6) {\scriptsize $v$};
 \draw[thick] (c)--(b)--(d)--(a);\end{tikzpicture}}}
\end{center}
respectively. Also, we shall identify 
 vertices of constructs with the sets decorating them.
\end{convention}
\begin{remark} \label{constructs-tubings}
Constructs (also known as $B$-trees \cite{P09}) provide an alternative description of the {\em tubings} or {\em nested sets}, which are traditionally used in the study of graph associahedra \cite{CD-CCGA} and nestohedra \cite{P09}, respectively. Indeed, let us denote, for every node $Y$ of a construct $T:{\bf H}$, by $\valup_T(Y)$ the union of the labels of the descendants of $Y$ in $T$ (all the way to the leaves), including $Y$.  We then associate with $T$ the following nested set:
$$\psi(T)= \setc{\valup_T(Y)}{Y\;\mbox{is a (label of a) node of}\;T}.$$  
\begin{example}
The tubings of the graph 
\begin{center}
{\resizebox{2.8cm}{!}{\begin{tikzpicture}
\node (u) {$u$};
\node[right=10pt of u] (v) {$v$};
\node[right=10pt of v] (x) {$x$};
\node[right=10pt of x] (y) {$y$};
\draw (u)--(v)--(x)--(y);
\end{tikzpicture}}}
\end{center}
associated to the constructs $x(\{u,v\},y)$ and $x(u(v),y)$ from Convention \ref{con1} are
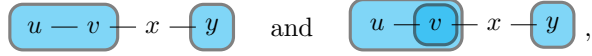
\begin{center}
{{\resizebox{2.75cm}{!}{\begin{tikzpicture}
\node (u) {$u$};
\node[right=10pt of u] (v) {$v$};
\node[right=10pt of v] (x) {$x$};
\node[right=10pt of x] (y) {$y$};
\draw (u)--(v)--(x)--(y);
 \begin{scope}[on background layer]
\node[draw,fill=cyan,very thick,inner sep=3pt,rounded corners=5pt,rotate fit=0,fit=(u)(v), opacity=0.5] {};
\node[draw,fill=cyan,very thick,inner sep=2pt,rounded corners=5pt,rotate fit=0,fit=(y), opacity=0.5] {};
\end{scope}
\end{tikzpicture}}}} \quad \raisebox{0.2cm}{and} \quad {\resizebox{2.75cm}{!}{\resizebox{2.5cm}{!}{\begin{tikzpicture}
\node (u) {$u$};
\node[right=10pt of u] (v) {$v$};
\node[right=10pt of v] (x) {$x$};
\node[right=10pt of x] (y) {$y$};
\draw (u)--(v)--(x)--(y);
 \begin{scope}[on background layer]
\node[draw, fill=cyan,very thick,inner sep=4pt,rounded corners=5pt,rotate fit=0,fit=(u)(v), opacity=0.5] {};
\node[draw, fill=cyan,very thick,inner sep=2pt,rounded corners=5pt,rotate fit=0,fit=(y), opacity=0.5] {};
\node[draw,fill=cyan,very thick,inner sep=2pt,rounded corners=5pt,rotate fit=0,fit=(v), opacity=0.5] {};
\end{scope}
\end{tikzpicture}}}} \raisebox{0.2cm}{,}
\end{center}
respectively.
\end{example}
\end{remark}

From the geometric point of view, the dimension of the face represented by a construct $T$ is given by the sum $\sum_X (|X|-1)$ ranging over all the nodes $X$ of $T$. Consequently, the vertices of the polytope realising $\hyper{H}$ correspond precisely to the constructs in which every node is a singleton. These specific trees are referred to as {\em constructions}.

\smallskip

\begin{definition} \label{rqueSubfaces}

The set of constructs inherits a natural subface poset structure, which we denote by $\subseteq$, via edge contractions. More explicitly, if a construct $S$ contains an edge connecting a parent node $Y$ to a child node $X$, this edge can be contracted to merge these two nodes into a single node labelled $X \cup Y$. The resulting tree $T$ remains a well-defined construct, and we denote this covering relation by $S \subsetdot T$. The subface relation $\subseteq$ is then defined as the reflexive transitive closure of $\subsetdot$.\end{definition}

The relation $\subsetdot$ models exactly the immediate subface relation within the geometric realisation of the polytope. Following this geometric interpretation, if $S \subseteq T$ we say that the face $S$ is contained in the face $T$. We illustrate the edge-contraction relation in Figure \ref{immediate-subface}.

\begin{figure}[H]
\begin{center}
\resizebox{!}{2cm}{\begin{tikzpicture}[scale=0.6]
\node (Y) [circle,minimum size=0.17cm,inner sep=0.5mm] at (0,-0.9) {$Y$};
 \node (X) [circle,minimum size=0.17cm,inner sep=0.5mm] at (-1,0.35) {$X$};
 \node (T11) [circle,draw=none,minimum size=0.75cm,inner sep=-0.1mm] at (-1.75,1.75) {\small $T_{11}$};
 \node (d1) [circle,draw=none,minimum size=0.75cm,inner sep=0mm] at (-1.1,1.75) {\footnotesize $\cdots$};
 \node (T1m) [circle,draw=none,minimum size=0.75cm,inner sep=0mm] at (-0.35,1.75) {\small $T_{1m}$};
 \node (T2) [circle,draw=none,minimum size=0.75cm,inner sep=0mm] at (0.55,1.75) {\small $T_{2}$};
 \node (d2) [circle,draw=none,minimum size=0.75cm,inner sep=0mm] at (1.25,1.75) {\footnotesize $\cdots$};
 \node (Tn) [circle,draw=none,minimum size=0.75cm,inner sep=0mm] at (2.05,1.75) {\small $T_{n}$};
\draw[thick] (Y)--(X);
\draw[thick] (-1.75,1.4)--(X);
\draw[thick] (-0.35,1.4)--(X);
\draw[thick] (Y)--(0.5,1.4);
\draw[thick] (Y)--(2,1.4);
 \end{tikzpicture}
 }\,\raisebox{0.75cm}{$\subsetdot$}\,\resizebox{!}{2cm}{\begin{tikzpicture}[scale=0.7]
\node (Y) [rectangle,minimum size=0.17cm,inner sep=0.6mm] at (0,-0.9) {$X\cup Y$};
 \node (T11) [circle,draw=none,minimum size=0.75cm,inner sep=-0.1mm] at (-1.75,1.75) {\small $T_{11}$};
 \node (d1) [circle,draw=none,minimum size=0.75cm,inner sep=0mm] at (-1.1,1.75) {\footnotesize $\cdots$};
 \node (T1m) [circle,draw=none,minimum size=0.75cm,inner sep=0mm] at (-0.35,1.75) {\small $T_{1m}$};
 \node (T2) [circle,draw=none,minimum size=0.75cm,inner sep=0mm] at (0.55,1.75) {\small $T_{2}$};
 \node (d2) [circle,draw=none,minimum size=0.75cm,inner sep=0mm] at (1.25,1.75) {\footnotesize $\cdots$};
 \node (Tn) [circle,draw=none,minimum size=0.75cm,inner sep=0mm] at (2.05,1.75) {\small $T_{n}$};
\draw[thick] (-1.75,1.4)--(Y);
\draw[thick] (-0.35,1.4)--(Y);
\draw[thick] (Y)--(0.5,1.4);
\draw[thick] (Y)--(2,1.4);
 \end{tikzpicture}
 }
\end{center}
\caption{Covering relation in the subface poset of constructs, which is isomorphic to the poset of faces of the geometric realisation.}
\label{immediate-subface}
\end{figure}
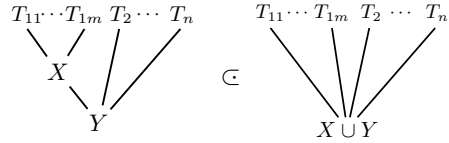

\subsection{Fundamental examples of hypergraph polytopes} \label{fundamental-examples-section}
 We conclude this introductory section by reviewing several important families of hypergraph polytopes. The correspondences between the classical descriptions of their face posets and the characterisations in terms of corresponding posets of constructs can be found in \cite[\S 2.4]{COI}.

\smallskip
The most classical example is obtained from the complete graph on a finite set $X=\set{x_1,\ldots,x_n}$: 
$$\hyper{P}^X= \set{\set{x_1},\ldots,\set{x_n}} \cup \setc{\set{x_i,x_j}}{1\leq i\neq j\leq n}.$$
The polytope associated to $\hyper{P}^X$ is the $(n-1)$-dimensional {\em permutohedron}. The constructs of $\hyper{P}^X$ are in bijective correspondence with the surjections $X \twoheadrightarrow \{1,\ldots,k\}$ for 
$1 \leq k \leq |X|$ (or, equivalently, with the ordered set partitions of $X$, also known as set compositions of $X$).

\smallskip
Another fundamental family of hypergraph polytopes arises from the linear graph on a totally ordered set $X=\{x_1<\cdots <x_n\}$: 
$${\bf K}^{X}=\{\{x_1\},\dots,\{x_n\},\{x_1,x_2\},\dots,\{x_{n-1},x_n\}\}.$$
The polytope associated to $\hyper{K}^X$ is the $(n-1)$-dimensional {\em associahedron}. The constructs of ${\bf K}^{X}$ are in bijective correspondence with the planar trees with $n+1$ leaves.

\smallskip 
Associahedra and permutohedra as particular instances of
 {\em teleassociahedra}, introduced in \cite[Example 4.3]{PLBJ2}, which are polytopes for graphs $$\hyper{T}^d_V:=\{\{i\}\,|\, i \in V\} \cup \{\{i,j\}\,|\, i,j \in V, i < j \mbox{ and } |i-j| \leq d\},$$ where $V\inc\mathbb{Z}$ is a non-empty finite set of integers and $d$ is a positive integer, whenever such a graph is connected. 
 
 \medskip
 The preceding examples are all graph associahedra. We now turn to several families of genuine hypergraph polytopes.

\smallskip

The simplest one is obtained from the hypergraph
$$\hyper{S}^X=\setc{\set{x}}{x\in X}\union\set{\set{X}},$$
which contains no non-trivial proper hyperedges. The polytope associated to $\hyper{S}^X$ is the $(n-1)$-dimensional {\em simplex}. 
The constructs of $\hyper{S}^X$ are in bijective correspondence with the non-empty subsets of $X$. 

\smallskip
Another family of polytopes encoded by genuine hypergraphs is determined by hypergraphs of the form 
 $$\hyper{C}^X=\{\{x_1\},\dots,\{x_n\}\}\cup\setc{\setc{x_j}{1\leq j\leq i}}{1\leq i\leq n},$$ for a finite ordered set $X=\set{x_1<\cdots<x_n}$. The polytope associated to $\hyper{C}^X$ is the $(n-1)$-dimensional {\em hypercube}. The constructs of $\hyper{C}^X$ are in bijective correspondence with the words of
length $n$ over the alphabet $\{+,-,\bullet\}$.

\smallskip
Finally, we introduce a family that appears to be new. Let $X=\set{x_1<\cdots<x_n}$ be a finite subset of $\mathbb{Z}$. The polytope for the hypergraph
\begin{align*}
\hyper{Q}^X=\{\{x_1\},\dots,\{x_n\}\} & \cup\setc{\{x_i,x_j\}}{|x_j-x_i|\leq 1} \\ &\cup\setc{\{x_i,x_j,x_k\}}{i<j<k, |x_k-x_i|\leq 3}.
\end{align*} 
is called the $(n-1)$-dimensional {\em quasi-associahedron}, when this hypergraph is connected.
 We illustrate hypergraphs of quasi-associahedra and give one corresponding geometric realisation in Figure \ref{FigQA}.
\begin{figure}[H]
\centering
\raisebox{0.75cm}{\begin{tikzpicture}
 \draw[fill=red!20,fill opacity=0.4] (-0.3,0) [out=-90,in=180] to (0,-0.3) [out=0,in=-180] to (1,-0.3) [out=0,in=-90]to (1.3,0) [out=90,in=-90] to (1.3,1) [out=90,in=0] to (1,1.3) [out=180,in=90] to (0.7,0.6) [out=-90,in=0] to (0,0.3) [out=180,in=90] to (-0.3,0);
 \draw[xscale=-1,xshift=-1cm,fill=blue!20,fill opacity=0.4] (-0.3,0) [out=-90,in=180] to (0,-0.3) [out=0,in=-180] to (1,-0.3) [out=0,in=-90]to (1.3,0) [out=90,in=-90] to (1.3,1) [out=90,in=0] to (1,1.3) [out=180,in=90] to (0.7,0.6) [out=-90,in=0] to (0,0.3) [out=180,in=90] to (-0.3,0);
\node (x) [circle,draw=none,inner sep=0.3mm,minimum size=1.3mm] at (0,1) {\small $2$};
\node (u) [circle,draw=none,inner sep=0.3mm,minimum size=1.3mm] at (1,1) {\small $3$};
\node (y) [circle,draw=none,inner sep=0.3mm,minimum size=1.3mm] at (0,0) {\small $1$};
\node (z) [circle,draw=none, inner sep=0.3mm,minimum size=1.3mm] at (1,0) {\small $4$};
\draw (y)--(x)--(u)--(z);
\end{tikzpicture}}
\quad\enspace
\raisebox{0.6cm}{\begin{tikzpicture}
 \draw[rotate=-72,xshift=-1.05cm,yshift=-0.55cm,fill=blue!20,fill opacity=0.4] (0,-0.2) [out=180,in=-90] to (-0.2,0) [out=90,in=-90,looseness=0.9] to (0.3,1) to[out=90,in=180,looseness=1] (0.5,1.76) [out=0,in=90] to (0.8,1)[out=-90,in=90]to (1.2,0)to[out=-90,in=0,looseness=1](1,-0.2) to[out=180,in=0] (0,-0.2);
 \draw[xscale=-1,rotate=-143,xshift=-0.86cm,yshift=-1.15cm,fill=red!20,fill opacity=0.4] (0,-0.2) [out=180,in=-90] to (-0.2,0) [out=90,in=-90] to (0.3,1) to[out=90,in=180] (0.5,1.76) [out=0,in=90] to (0.8,1)[out=-90,in=90]to (1.2,0)to[out=-90,in=0,looseness=1](1,-0.2) to[out=180,in=0] (0,-0.2);
 \draw[xscale=-1,rotate=-71,xshift=-1.035cm,yshift=-0.5cm,fill=green!20,fill opacity=0.4] (0,-0.2) [out=180,in=-90] to (-0.2,0) [out=90,in=-90] to (0.29,0.95) to[out=90,in=180,looseness=0.9] (0.5,1.75) [out=0,in=90,looseness=1] to (0.8,1)[out=-90,in=90]to (1.2,0)to[out=-90,in=0,looseness=1](1,-0.2) to[out=180,in=0] (0,-0.2);
\draw[rotate=-143,xshift=-0.83cm,yshift=-1.17cm,fill=teal!20,fill opacity=0.4] (0,-0.2) [out=180,in=-90] to (-0.2,0) [out=90,in=-90] to (0.3,1) to[out=90,in=180] (0.5,1.75) [out=0,in=90] to (0.8,1)[out=-90,in=90]to (1.2,0)to[out=-90,in=0,looseness=1](1,-0.2) to[out=180,in=0] (0,-0.2);
\node (x) [circle,draw=none,inner sep=0.3mm,minimum size=1.3mm] at (-0.5,-0.1) {\small $1$};
\node (y) [circle,draw=none,inner sep=0.3mm,minimum size=1.3mm] at (-0.8,0.8) {\small $2$};
\node (z) [circle,draw=none,inner sep=0.3mm,minimum size=1.3mm] at (-0.05,1.4) {\small $3$};
\node (u) [circle,draw=none, inner sep=0.3mm,minimum size=1.3mm] at (0.7,0.8) {\small $4$};
\node (v) [circle,draw=none, inner sep=0.3mm,minimum size=1.3mm] at (0.4,-0.1) {\small $5$};
\draw (x)--(y)--(z)--(u)--(v);
\end{tikzpicture}} \quad {\resizebox{4cm}{!}{\begin{tikzpicture}[thick,scale=5]
\coordinate (A1) at (0,2);
\coordinate (A11) at (-0.39,1.5);
\coordinate (A12) at (0.39,1.5);
\coordinate (A121) at (0.25,1.498);
\coordinate (A122) at (0.33,1.46);
\coordinate (A13) at (0,1.25);
\coordinate (A131) at (-0.153,1.35); 
\coordinate (A132) at (0.153,1.35); 
\coordinate (A2) at (0,0); 
\coordinate (A21) at (-0.387,0.213); 
\coordinate (A22) at (0.387,0.213); 
\coordinate (A23) at (0,0.5); 
\coordinate (A231) at (-0.153,0.388); 
\coordinate (A232) at (0.153,0.388); 
\coordinate (A3) at (-1.1,0.6);
\coordinate (A31) at (-0.9,0.49);
\coordinate (A32) at (-0.8,0.976);
\coordinate (A33) at (-0.6,0.6);
\coordinate (A331) at (-0.65,0.66);
\coordinate (A332) at (-0.7,0.56);
\coordinate (A4) at (1.1,0.6);
\coordinate (A42) at (0.9,0.49);
\coordinate (A41) at (0.8,0.976);
\coordinate (A411) at (0.825,0.845);
\coordinate (A412) at (0.745,0.875);
\coordinate (A43) at (0.6,0.6);
\coordinate (A431) at (0.65,0.66);
\coordinate (A432) at (0.7,0.56);
\draw[draw=red] (A32)--(A3)--(A31);
\draw[draw=red] (A11)--(A1)--(A4);
\draw[draw=red] (A21)--(A2)--(A22);
\draw[draw=red] (A2) -- (A1);
\draw[draw=red] (A13)--(A1);\draw[draw=red] (A41)--(A4)--(A42);\draw[draw=red,dashed] (A3)--(A4);\draw[dashed] (A412)--(A431)--(A432)--(A42);
\draw[dashed](A431)--(A432)--(A332)--(A331)--cycle;
\draw (A21)--(A231)--(A232)--(A22);
\draw (A231) -- (A131) -- (A132) -- (A232) -- cycle;
\draw[dashed] (A32) -- (A331)--(A332) -- (A31); 
\draw (A11) -- (A131) --(A132) -- (A122);
\node (x) at (-0.4,0.9) {\Huge {\bf 1}};
\node (y) at (0.4,0.9) {\Huge {\bf 2}}; 
\node (z) at (-0.425,1.75) {\Huge {\bf 3}};
\node (u) at (-0.92,0.2) {\Huge {\bf 4}}; 
\draw (A121)--(A122)--(A411);\draw[dashed](A121)--(A412)--(A411);
\draw (-0.85,0.2) edge[->, bend right, >=stealth, line width=4pt, shorten >=-1mm] (-0.5,0.4);
\draw[yscale=-1,yshift=-1.75cm,xshift=0.17cm] (-0.6,0.1) edge[->, bend left, >=stealth, line width=4pt, shorten >=-1mm] (-0.5,0.5);
\fill [fill=white,opacity=0.75] (A231)--(A21)--(A31)--(-0.5,0.55)--cycle;
\fill [fill=white,opacity=0.75] (A11)--(A131)--(-0.6,0.9)--(A32)--cycle;
 \draw[draw=black,fill=none] (A121)-- (A11) -- (A32)--(A31) --(A21)-- (A22) -- (A42)--(A411);
\end{tikzpicture}}}
 
\caption{Quasi-associahedra on $\llbracket 1,4 \rrbracket$ (on the left) and $\llbracket 1,5 \rrbracket$ (in the middle) and the associated polytope in dimension $3$ (on the right) \label{FigQA}}
 
\end{figure}
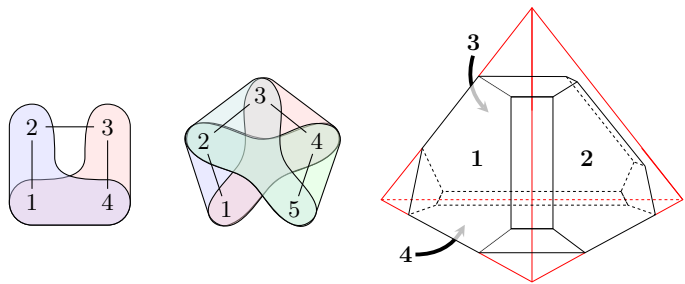
 The number of constructs of these polytopes for $X=\llbracket 1,k \rrbracket$ is counted in Table \ref{tableQA}. To the best of our knowledge, the resulting sequences did not appear in the OEIS and have been added by the authors as \cite[A399243]{oeis}. 
 
\begin{table}
\centering
\begin{tabular}{|c|c|c|c|c|c|}
\hline
k & 1 & 2 & 3 & 4 & 5 \\ \hline
dim 0 & 1 & 2 & 5 & 18& 72 \\ \hline
dim 1 & 0 & 1 & 5 & 27& 144 \\ \hline
dim 2 & 0& 0 & 1 & 11& 93 \\ \hline
dim 3 & 0 & 0 & 0 & 1 & 21 \\ \hline
dim 4 & 0 & 0 & 0 & 0 & 1 \\ \hline
Total & 1 & 3 & 11 & 57 & 336 \\ \hline
\end{tabular}
\caption{Number of constructs of a fixed dimension for the quasi-associahedra on $X=\llbracket 1,k \rrbracket$. \label{tableQA}}
\end{table}

\section{Generalised flip order} \label{GFO-section}
In the rest of the paper, we assume that a total order is given on the set of vertices of the hypergraph in consideration; we refer to such a hypergraph as {\em ordered}. The intuition behind the generalised flip order is that it organises the constructs of an ordered hypergraph ${\bf H}$ according to ``how much they respect the ambient ordering on the vertices of ${\bf H}$''.
\subsection{Definition of the generalised flip order}\label{gfo_def}

We first recall the relation on faces of hypergraph polytopes introduced by the first author and Laplante-Anfossi, which is based on the immediate subface relation (cf. \S \ref{constructs-subsection}). We need a more precise notation for the latter: if $S \subsetdot T$ results from the contraction of an edge between $X$ and $Y$ in $S$ ($X$ parent, $Y$ child), then we write $S \subsetdot_{X,Y} T$.
\begin{definition}\label{cov}
	Let  $\hyper{H}$ be an ordered hypergraph and let  $S,S':\hyper{H}$. We say that $S$ is covered by $S'$, and write $S \lessdot S'$, if and only if there exist  sets $U$ and $V$ such that $U>V$ and either $S\, \subsetdot_{U,V} \,S'$   or $S' \,\subsetdot_{V,U} \,S$ .  In the first case, we call the covering relation a \emph{fusion}, and a \emph{root fusion} if $U=\text{root}(S)$. 
	In the second case, we call the covering relation a \emph{split}, and a \emph{root split} if $U\cup V=\text{root}(S)$. In both cases, we refer to the condition $U>V$ as the {\em order condition} of the covering.

The generalised flip order (GFO) is the binary relation $\leq$ defined on the set of all constructs of a hypergraph as the reflexive and transitive closure of the covering relation $\lessdot$.
 \end{definition}

In the following example, we illustrate fusion and split.

\begin{example}\label{runex} Here are some instances of fusion and split for our running example hypergraph ${\bf H}^\maltese$ from Example \ref{hypergraphex1}:
\begin{center}
		\begin{tikzpicture}
			\node (a) [inner sep=2pt] at (0,0) {\small $2$};
			\node (b) [inner sep=2pt] at (-0.35,0.6) {\small $1$};
			\node (c) [inner sep=2pt] at (0.35,0.6) {\small $4$};
			\node (d) [inner sep=2pt] at (0.35,1.2) {\small $3$};
			\draw[thick] (d)--(c)--(a)--(b); 
		\end{tikzpicture} \raisebox{0.7cm}{$\lessdot$} \begin{tikzpicture} 
			\node (a) [inner sep=2pt] at (0,0) {\small $\{1,2\}$};
			\node (b) [inner sep=2pt] at (0,0.6) {\small $4$};
			\node (c) [inner sep=2pt] at (0,1.2) {\small $3$};
			\draw[thick] (a)--(b)--(c); 
		\end{tikzpicture} \quad\quad \begin{tikzpicture} 
			\node (a) [inner sep=2pt] at (0,0.6) {\small $1$};
			\node (b) [inner sep=2pt] at (0,0) {\small $\{2,3,4\}$};
			\draw[thick] (a)--(b); 
		\end{tikzpicture} \raisebox{0.3cm}{$\lessdot$} \begin{tikzpicture}
			\node (a) [inner sep=2pt] at (0,0) {\small $2$};
			\node (b) [inner sep=2pt] at (-0.35,0.6) {\small $1$};
			\node (c) [inner sep=2pt] at (0.35,0.6) {\small $\{3,4\}$};
			\draw[thick] (c)--(a)--(b); 
		\end{tikzpicture} \quad\enspace
		\begin{tikzpicture}
			\node (a) [inner sep=2pt] at (0,0) {\small $4$};
			\node (b) [inner sep=2pt] at (0,0.6) {\small $\{1,2,3\}$};
			\draw[thick] (a)--(b); 
		\end{tikzpicture} \raisebox{0.3cm}{$\lessdot$} \raisebox{0.2cm}{\begin{tikzpicture}
				\node (a) [inner sep=2pt] at (0,0) {\small $\{1,2,3,4\}$};
		\end{tikzpicture}} \raisebox{0.3cm}{$\lessdot$} \begin{tikzpicture} 
			\node (a) [inner sep=2pt] at (0,0.6) {\small $4$};
			\node (b) [inner sep=2pt] at (0,0) {\small $\{1,2,3\}$};
			\draw[thick] (a)--(b); 
		\end{tikzpicture}
	\end{center}
\vspace{-.1cm}	The last example above illustrates how fusion followed by the  split  involving the same two nodes that were previously fused together results in the {\em flip} of two nodes of a construct.  Such a flip is not always possible. For instance, we have 
	\begin{center}
\begin{tikzpicture} 
 \node (a) [inner sep=0pt] at (0,0.6) {\small $\{1,2\}$};
 \node (b) [inner sep=0pt] at (0,0) {\small $\{3,4\}$};
 \draw[thick] (a)--(b); 
 \end{tikzpicture}\,\raisebox{0.3cm}{$\lessdot$}\, \raisebox{0.15cm}{\begin{tikzpicture} 
 \node (b) [inner sep=2pt] at (0,0) {\small $\{1,2,3,4\}$};
 \end{tikzpicture}}\,\raisebox{0.3cm}{$\lessdot$}\, \begin{tikzpicture} 
  \node (a) [inner sep=1pt] at (0,0.6) {\small $4$};
    \node (b) [inner sep=1pt] at (0,0) {\small $\{1,2,3\}$};
      \draw[thick] (a)--(b);  
  \end{tikzpicture}\,\raisebox{0.3cm}{$\lessdot$}\,  \begin{tikzpicture} 
 \node (a) [inner sep=2pt] at (-0.6,0.6) {\small $3$};
 \node (c) [inner sep=2pt] at (0.6,0.6) {\small $4$};
 \node (b) [inner sep=2pt] at (0,0) {\small $\{1,2\}$};
 \draw[thick] (a)--(b)--(c); 
 \end{tikzpicture}
\end{center}
\vspace{-.2cm}	hence the nodes $\{1,2\}$ and $\{3,4\}$ cannot be flipped. Here are some non-examples for which the order condition fails:
	\begin{center} 
		 \begin{tikzpicture} 
 \node (a) [inner sep=2pt] at (-0.6,0.6) {\small $2$};
 \node (c) [inner sep=2pt] at (0.6,0.6) {\small $4$};
 \node (b) [inner sep=2pt] at (0,0) {\small $\{1,3\}$};
 \draw[thick] (a)--(b)--(c); 
 \end{tikzpicture} \raisebox{0.3cm}{$\centernot\lessdot$} \begin{tikzpicture} 
			\node (a) [inner sep=1pt] at (0,0.6) {\small $4$};
			\node (b) [inner sep=1pt] at (0,0) {\small $\{1,2,3\}$};
			\draw[thick] (a)--(b); 
		\end{tikzpicture} \quad\quad\quad
		\begin{tikzpicture} 
			\node (a) [inner sep=1pt] at (0,0.6) {\small $1$};
			\node (b) [inner sep=1pt] at (0,0) {\small $\{2,3,4\}$};
			\draw[thick] (a)--(b); 
		\end{tikzpicture} \raisebox{0.3cm}{$\centernot\lessdot$} \begin{tikzpicture} 
 \node (a) [inner sep=2pt] at (-0.6,0.6) {\small $1$};
 \node (c) [inner sep=2pt] at (0.6,0.6) {\small $3$};
 \node (b) [inner sep=2pt] at (0,0) {\small $\{2,4\}$};
 \draw[thick] (a)--(b)--(c); 
 \end{tikzpicture}
	\end{center}
\end{example}
At first glance, Definition \ref{cov} appears perfectly symmetric: the covering relation is described either by a fusion or by a split and these two actions are dual to each other. However, there is no side condition associated with a fusion, while, if we want to split a node $Z=U\cup V$, then, in addition to requiring that $U<V$, we must also make sure that $V$ lies in a single connected component in the appropriate restricted hypergraph.  We formalise this in the following definition. 
\begin{definition} \label{splittable-definition}
	Let $S:\hyper{H}$. A node $Z=U\cup V$  of $S$, where $U<V$,   is {\em $(U,V)$-splittable in} $S$ if either $Z=\text{root}(S)$  and $V$ lies entirely in one of the connected components of $\hyper{H}{\setminus U}$,
	  or $S=X(S_1,\ldots,S_n)$, $Z$ is a node of some $S_i$ and $Z$ is $(U,V)$-splittable in $S_i$.
\end{definition}

We  note that for every construct $S$ and every node $Z$ of $S$ such that $|Z|\geq 2$, $Z$ is $(Z\backslash\{\max(Z)\},\{\max(Z)\})$-splittable in $S$.   
 
\smallskip

 In the next definition, we formalise the 
actions of splitting
a vertex and collapsing an edge in a construct. 
\begin{definition}\label{basic_split}
Let ${\bf H}$ be a hypergraph, let $X\subseteq H$, ${\bf H}\backslash X\leadsto H_{1},\dots, H_{n}$ and let $S_i:{\bf H}_{i}$. 
Consider the construct $S=X(S_1,\dots,S_n):{\bf H}$.

\begin{itemize} 
\item[a)] Let
$Z$ be an $(U,V)$-splittable
 node of $S$. The non-planar rooted tree $S[Z\rightarrow U(V)]$, obtained from $S$ by the $(U,V)$-splitting of the vertex $V$, is defined recursively as follows. If there exists an index $i$, $1\leq i\leq n$, such
that $Z$ is a node of $C_i$, we define
$$S[Z\rightarrow U(V)]:=X(S_1,\dots,S_{i-1},S_i[Z\rightarrow U(V)],S_{i+1},\dots,S_n).$$ 
Assume that $Z=X$ and let
$\{i_1,\dots,i_p\}\cup \{j_1,\dots,j_q\}$ be the partition of the set
$\{1,\dots,n\}$ such that the hypergraphs ${\bf H}_{i_s}$, for
$1\leq s\leq p$ are those contained in the connected component of $\hyper{H}\backslash U$ that contains $V$.

We define
\begin{equation}S[Z\rightarrow U(V)]:=U(V(S_{i_1},\dots
S_{i_p}),S_{j_1},\dots,S_{j_q}).\label{splitdef}\end{equation} If, exceptionally, 
$\{i_1,\dots,i_p\}=\emptyset$ (resp. $\{j_1,\dots,j_r\}=\emptyset$, $\{i_1,\dots,i_p\}=\{j_1,\dots,j_r\}=\emptyset$), $S[Z\rightarrow U(V)]$ takes the form
 $U(V,S_{j_1},\dots,S_{j_q})$ (resp.\linebreak $U(V(S_{i_1},\dots,S_{i_p}))$, $U(V)$). 
\item[b)] If $Y$ and $Z$ are two adjacent nodes of $S$, with $Y$ being the parent of $Z$, we shall denote by $S[Y(Z)\rightarrow Y\cup Z]$ the non-planar rooted tree obtained from $S$ by collapsing the edge between $Y$ and $Z$ and labelling with $Y\cup Z$ the node obtained by merging $Y$ and $Z$.
\end{itemize}
\end{definition}

\begin{lemma}\label{splitformally} If $Z$ is an $(U,V)$-splittable
 node of $S:{\bf H}$, then $S[Z\rightarrow U(V)]$ is a construct of ${\bf H}$. Likewise, if $Y$ and $Z$ are two adjacent nodes of $S$, with $Y$ being the parent of $Z$, then $S[Y(Z)\rightarrow Y\cup Z]$ is a construct of ${\bf H}$. 
\end{lemma}
\begin{proof}
Both claims follow easily by induction on the number
of vertices of $S$. In the proof of $S[Z\rightarrow U(V)]:{\bf H}$, the only interesting case is the one given by 
$S=Z(S_1,\dots,S_n)$. In that
case, the argument is based on the fact that the set of vertices
$V\cup \bigcup_{i\in \{i_1,\dots,i_q\}} {H}_i$
determines a connected component ${\bf H}'$ of ${\bf H}$ and,
furthermore, that $V(S_{i_1},\dots S_{i_q}):{\bf H}'$.
\end{proof}

\begin{convention}\label{nota}
In the sequel, we shall often transform constructs by performing successively multiple actions from Definition \ref{basic_split}. To denote the result of such a transformation, we shall use the square bracket notation from that definition in successive manner, omitting the parentheses. In addition, we shall use the shorthand $S[U(V)\rightarrow U\cup V\rightarrow V(U)]$ for $S[U(V)\rightarrow U\cup V][U\cup V\rightarrow V(U)]$. 
\end{convention}

\begin{example}\label{Exple27}
It is easy to check that, for ${\bf H}^\maltese$ (cf. Example \ref{hypergraphex1} and Example \ref{runex}), it holds that, for every construct $C:{\bf H}^\maltese$ and every node $Z=U\cup V$ of $C$, if $U<V$, then $Z$ is $(U,V)$-splittable.

For an example of non-splittability coming from a connectivity issue, consider ${\bf H}^\dag=\{\{1\},\{2\},\{3\},\{4\},\{2,3\},\{3,1\},\{1,4\}\}$ and the unique single-node construct $S=\{1,2,3,4\}:{\bf H}^\dag$. Although $\{1,2\}<\{3,4\}$, the unique node $\{1,2,3,4\}$ of $S$ is not $(\{1,2\},\{3,4\})$-splittable, since $\{3,4\}$ is not contained in a single connected component of ${\bf H}^\dag\backslash\{1,2\}$. The complete Hasse diagram of the GFO for ${\bf H}^\dag$ can be found in Figure \ref{hassediag4} in the Appendix.
\end{example}

In the following two remarks, we collect some easy observations that will be used for proving that the GFO is indeed a partial order.

\begin{remark} \label {lem:StabilityByContext}
\begin{itemize}
\item[a)] We note that, by Definition \ref{cov}, if we have two constructs $S=X(S_1,\ldots,S_n)$ and $S'=X(S'_1,\ldots,S'_n)$
where $S_i\lessdot S'_i$ for some $i$ and $S_j=S'_j$ for all $j\neq i$, then $S\lessdot S'$.
We call such a covering instance a {\em non-root covering}. Therefore, a covering $S\lessdot S'$ can be of three kinds: a root fusion, a root split, or a non-root covering. 
\item[b) ]It follows readily that if we have two constructs $S=X(S_1,\ldots,S_n)$ and $S'=X(S'_1,\ldots,S'_n)$
where $S_i\leq S'_i$ for each $i$, then $S\leq S'$. 
Finally we also observe that if $X(S_1,\ldots,S_n)=S\leq S'=X(S'_1,\ldots,S'_n)$ is obtained via a sequence of non-root coverings, then the sequence is in fact an interleaving of sequences $S_i\leq S'_i$, for all $i$, as what happens in each $S_i$ is independent of what happens in any other $S_j$.
\end{itemize}
\end{remark}

\begin{remark}\label{l1}
If $C,C':{\bf H}$ are such that $C\lessdot C'$, then, setting $X=\textrm{root}(C)$ and
 $X'=\textrm{root}(C')$, we have $\min(X)\geq \min(X')$ and $\max(X)\geq\max(X')$. More precisely, considering the three possible kinds of the covering $C\lessdot C'$ (cf. Remark \ref{lem:StabilityByContext}(a)), if $C\lessdot C'$ is  
\begin{itemize}
\item[-] a root split, then $\min(X)=\min(X')$ and $\max(X)>\max(X')$; 
\item[-] a root fusion, then $\min(X)>\min(X')$ and $\max(X)=\max(X')$;  
\item[-] a non-root covering, then $\min(X)=\min(X')$ and $\max(X)=\max(X')$. 
 \end{itemize}
\end{remark}

Relying on Remark \ref{l1}, we next prove another property of the  relation $\lessdot$.

\begin{lemma}\label{l2}
If $C_1,\dots,C_n:{\bf H}$ are such that $C_1 \lessdot C_2 \lessdot \cdots \lessdot C_n$ and $\textrm{root}(C_1)=\textrm{root}(C_n)=X$,
 then $\operatorname{root}(C_i)=X$ for all $1\leq i\leq n$.
\end{lemma}
\begin{proof}
Suppose that $\textrm{root}(C_i)=X_i$ for all $1\leq i\leq n$. The claim follows by showing that each covering in $C_1 \lessdot C_2 \lessdot \cdots \lessdot C_n$ must be non-root, since such coverings do not modify the root label. Suppose the opposite, and let $C_i\lessdot C_{i+1}$ be the first covering instance involving the root. By Remark \ref{l1}, we have either $\min(X_1)=\cdots =\min(X_i)>\min(X_{i+1})\geq \cdots\geq\min(X_n)$ or
 $\max(X_1)=\cdots =\max(X_i)>\max(X_{i+1})\geq \cdots\geq\max(X_n)$, and therefore either
 $\min(X_1)>\min(X_n)$ or $\max(X_1)>\max(X_n)$, contradicting $X_1=X_n=X$.
\end{proof}

We are now in position to prove the main result of this section. 

\begin{prop} \label{lessdot-order}
 The GFO is a partial order.
\end{prop} 

\begin{proof}
 As the transitivity and reflexivity are clear, we only have to check the anti-symmetry. 
We prove that there is no non-trivial sequence of coverings $S=C_1\lessdot \cdots \lessdot C_n =S$ forming a loop, by induction on $|H|$. 
 If $|H|=1$, then $n=1$ and the claim holds trivially. The inductive case is treated as follows. By Lemma \ref{l2} and Remark \ref{lem:StabilityByContext}(b), we get that all the
$C_i$'s have the same root $X$, and that, writing $S=X(S_1,\ldots,S_n)$, the assumed sequence 
 $S=C_1\lessdot \cdots \lessdot C_n =S$ has to be an interleaving of putative sequences 
 $S_i\lessdot \cdots \lessdot S_i$, for all $i$. But those sequences are all empty by induction, and so is the one from $S$ to $S$. This concludes the proof.
\end{proof}
We represent in Figure \ref{fig:Hassediag1} and Figure \ref{fig:Hassediag2} the Hasse diagrams of the GFO on the faces of all the non-isomorphic different polytopes of dimension $2$. 
\begin{figure}[H] 
 \begin{tabular}{cc}
\resizebox{4cm}{!}{\begin{tikzpicture}[scale=0.85,
  hasse/.style={
    draw,
    rounded corners,
    fill=#1,
    inner sep=2pt
  },
  tree/.style={font=\footnotesize},
  edge/.style={thick}
]

\node[hasse=MutedLavender!25] (c1) at (0,5) {
  \begin{tikzpicture}[tree]
    \node (a) at (0,0) {$1$};
    \node (b) at (-0.35,0.6) {$2$};
    \node (c) at (0.35,0.6) {$3$};
    \draw[edge] (b)--(a)--(c);
  \end{tikzpicture}
};

\node[hasse=DirtyYellow!25] (a1) at (-3.8,1) {
  \begin{tikzpicture}[tree]
    \node (a) at (0,0) {$\{1,3\}$};
    \node (b) at (0,0.7) {$2$};
    \draw[edge] (a)--(b);
  \end{tikzpicture}
};

\node[hasse=SageGreen!25] (a2) at (0,1) {
  \begin{tikzpicture}[tree]
    \node {$\{1,2,3\}$};
  \end{tikzpicture}
};

\node[hasse=SageGreen!25] (a3) at (3.8,1) {
  \begin{tikzpicture}[tree]
    \node (a) at (0,0) {$2$};
    \node (b) at (-0.35,0.6) {$1$};
    \node (c) at (0.35,0.6) {$3$};
    \draw[edge] (b)--(a)--(c);
  \end{tikzpicture}
};

\node[hasse=SlateBlue!25] (b1) at (1.85,3.1) {
  \begin{tikzpicture}[tree]
    \node (a) at (0,0) {$\{1,2\}$};
    \node (b) at (0,0.7) {$3$};
    \draw[edge] (a)--(b);
  \end{tikzpicture}
};

\node[hasse=DirtyYellow!25] (b2) at (1.85,-1.1) {
  \begin{tikzpicture}[tree]
    \node (a) at (0,0) {$\{2,3\}$};
    \node (b) at (0,0.7) {$1$};
    \draw[edge] (a)--(b);
  \end{tikzpicture}
};

\node[hasse=EggShell!25] (c2) at (0,-3) {
  \begin{tikzpicture}[tree]
    \node (a) at (0,0) {$3$};
    \node (b) at (-0.35,0.6) {$1$};
    \node (c) at (0.35,0.6) {$2$};
    \draw[edge] (b)--(a)--(c);
  \end{tikzpicture}
};

\draw[edge]
  (a1)--(c1)--(b1)--(a2)--(b2)--(c2)--(a1);

\draw[edge]
  (b1)--(a3)--(b2);
\end{tikzpicture}} \quad\quad  & \quad\quad \resizebox{4cm}{!}{\begin{tikzpicture}[scale=0.85,
  hasse/.style={
    draw,
    rounded corners,
    fill=#1,
    inner sep=2pt
  },
  tree/.style={font=\footnotesize},
  edge/.style={thick}
]

\node[hasse=MutedLavender!25] (c1) at (0,5) {
  \begin{tikzpicture}[tree]
    \node (a) at (0,0) {$1$};
    \node (b) at (-0.35,0.6) {$2$};
    \node (c) at (0.35,0.6) {$3$};
    \draw[edge] (b)--(a)--(c);
  \end{tikzpicture}
};

\node[hasse=SlateBlue!25] (a1) at (-1.85,3.1) {
  \begin{tikzpicture}[tree]
    \node (a) at (0,0) {$\{1,3\}$};
    \node (b) at (0,0.7) {$2$};
    \draw[edge] (a)--(b);
  \end{tikzpicture}
};

\node[hasse=SageGreen!25] (a2) at (0,1) {
  \begin{tikzpicture}[tree]
    \node {$\{1,2,3\}$};
  \end{tikzpicture}
};

\node[hasse=SageGreen!25] (a3) at (3.8,1) {
  \begin{tikzpicture}[tree]
    \node (a) at (0,0) {$2$};
    \node (b) at (-0.35,0.6) {$1$};
    \node (c) at (0.35,0.6) {$3$};
    \draw[edge] (b)--(a)--(c);
  \end{tikzpicture}
};

\node[hasse=SlateBlue!25] (b1) at (1.85,3.1) {
  \begin{tikzpicture}[tree]
    \node (a) at (0,0) {$\{1,2\}$};
    \node (b) at (0,0.7) {$3$};
    \draw[edge] (a)--(b);
  \end{tikzpicture}
};

\node[hasse=DirtyYellow!25] (b2) at (1.85,-1.1) {
  \begin{tikzpicture}[tree]
    \node (a) at (0,0) {$\{2,3\}$};
    \node (b) at (0,0.7) {$1$};
    \draw[edge] (a)--(b);
  \end{tikzpicture}
};

\node[hasse=DirtyYellow!25] (b3) at (-1.85,-1.1) {
   \begin{tikzpicture}[tree]
    \node (a) at (0,0) {$3$};
    \node (b) at (0,0.7) {$\{1,2\}$};
    \draw[edge] (b)--(a);
  \end{tikzpicture}
};

\node[hasse=EggShell!25] (c2) at (0,-3) {
  \begin{tikzpicture}[tree]
    \node (a) at (0,0) {$3$};
    \node (b) at (0,0.6) {$2$};
    \node (c) at (0,1.2) {$1$};
    \draw[edge] (a)--(b)--(c);
  \end{tikzpicture}
};

\node[hasse=SageGreen!25] (d1) at (-3.8,1) {
  \begin{tikzpicture}[tree]
    \node (a) at (0,0) {$3$};
    \node (b) at (0,0.6) {$1$};
    \node (c) at (0,1.2) {$2$};
    \draw[edge] (a)--(b)--(c);
  \end{tikzpicture}
};

\draw[edge]
 (c2)--(b3)--(a2)--(b2)--(c2);
 
\draw[edge]
(b3)--(d1)--(a1)--(c1)--(b1)--(a3)--(b2);
\draw[edge]
(a2)--(b1);
\end{tikzpicture}}\\
Simplex \raisebox{-5pt}{\begin{tikzpicture}
\node[inner sep=2pt]  (1) at (0,0) {\footnotesize $1$};
\node[inner sep=2pt]  (2) at (0.6,0) {\footnotesize $2$};
\node[inner sep=2pt]  (3) at (1.2,0) {\footnotesize $3$};
\draw[rounded corners=10pt] (-0.3,-0.3) rectangle (1.5,0.3);
\end{tikzpicture}} &  Cube  \raisebox{-5pt}{\begin{tikzpicture}
\node[inner sep=2pt]  (1) at (0,0) {\footnotesize $1$};
\node[inner sep=2pt]  (2) at (0.6,0) {\footnotesize $2$};
\node[inner sep=2pt]  (3) at (1.2,0) {\footnotesize $3$};
\draw[rounded corners=10pt] (-0.3,-0.3) rectangle (1.5,0.3);
\draw (1)--(2);
\end{tikzpicture}}  \\[0.5cm]
\resizebox{4cm}{!}{\begin{tikzpicture}[scale=0.85,
  hasse/.style={
    draw,
    rounded corners,
    fill=#1,
    inner sep=2pt
  },
  tree/.style={font=\footnotesize},
  edge/.style={thick}
]

\node[hasse=espresso!25] (c1) at (0,5) {
  \begin{tikzpicture}[tree]
    \node (a) at (0,0) {$1$};
    \node (b) at (-0.35,0.6) {$2$};
    \node (c) at (0.35,0.6) {$3$};
    \draw[edge] (b)--(a)--(c);
  \end{tikzpicture}
};

\node[hasse=DirtyYellow!25] (a1) at (-3.8,1) {
  \begin{tikzpicture}[tree]
    \node (a) at (0,0) {$\{1,3\}$};
    \node (b) at (0,0.7) {$2$};
    \draw[edge] (a)--(b);
  \end{tikzpicture}
};

\node[hasse=SageGreen!25] (a2) at (0,1) {
  \begin{tikzpicture}[tree]
    \node {$\{1,2,3\}$};
  \end{tikzpicture}
};

\node[hasse=SlateBlue!25] (a3) at (3.8,1) {
  \begin{tikzpicture}[tree]
    \node (a) at (0,0) {$2$};
    \node (b) at (0,0.7) {$\{1,3\}$};
    \draw[edge] (b)--(a);
  \end{tikzpicture}
};

\node[hasse=DustyRed!25] (b1) at (1.65,3.5) {
  \begin{tikzpicture}[tree]
    \node (a) at (0,0) {$\{1,2\}$};
    \node (b) at (0,0.7) {$3$};
    \draw[edge] (a)--(b);
  \end{tikzpicture}
};

\node[hasse=DirtyYellow!25] (b2) at (1.65,-1.5) {
  \begin{tikzpicture}[tree]
    \node (a) at (0,0) {$\{2,3\}$};
    \node (b) at (0,0.7) {$1$};
    \draw[edge] (a)--(b);
  \end{tikzpicture}
};

\node[hasse=EggShell!25] (c2) at (0,-3) {
  \begin{tikzpicture}[tree]
    \node (a) at (0,0) {$3$};
    \node (b) at (-0.35,0.6) {$1$};
    \node (c) at (0.35,0.6) {$2$};
    \draw[edge] (b)--(a)--(c);
  \end{tikzpicture}
};

\node[hasse=SageGreen!25] (d1) at (2.65,-0.25) {
  \begin{tikzpicture}[tree]
    \node (a) at (0,0) {$2$};
    \node (b) at (0,0.6) {$3$};
    \node (c) at (0,1.2) {$1$};
    \draw[edge] (a)--(b)--(c);
  \end{tikzpicture}
};

\node[hasse=MutedLavender!25] (d2) at (2.65,2.25) {
  \begin{tikzpicture}[tree]
    \node (a) at (0,0) {$2$};
    \node (b) at (0,0.6) {$1$};
    \node (c) at (0,1.2) {$3$};
    \draw[edge] (a)--(b)--(c);
  \end{tikzpicture}
};

\draw[edge]
(b2)--(c2);
\draw[edge]
(b2)--(d1)--(a3)--(d2)--(b1)--(c1)--(a1)--(c2);
\draw[edge]
(b1)--(a2)--(b2);
\end{tikzpicture}} \quad\quad & \quad\quad  \resizebox{4cm}{!}{\begin{tikzpicture}[scale=0.85,
  hasse/.style={
    draw,
    rounded corners,
    fill=#1,
    inner sep=2pt
  },
  tree/.style={font=\footnotesize},
  edge/.style={thick}
]

\node[hasse=MutedLavender!25] (c1) at (0,5) {
  \begin{tikzpicture}[tree]
    \node (a) at (0,0) {$1$};
    \node (b) at (0,0.6) {$2$};
    \node (c) at (0,1.2) {$3$};
    \draw[edge] (a)--(b)--(c);
  \end{tikzpicture}
};

\node[hasse=SlateBlue!20] (a1) at (-1.85,3.1) {
  \begin{tikzpicture}[tree]
    \node (a) at (0,0) {$1$};
    \node (b) at (0,0.7) {$\{2,3\}$};
    \draw[edge] (a)--(b);
  \end{tikzpicture}
};

\node[hasse=SageGreen!25] (a2) at (0,1) {
  \begin{tikzpicture}[tree]
    \node {$\{1,2,3\}$};
  \end{tikzpicture}
};

\node[hasse=SageGreen!25] (a3) at (3.8,1) {
  \begin{tikzpicture}[tree]
    \node (a) at (0,0) {$2$};
    \node (b) at (-0.35,0.6) {$1$};
    \node (c) at (0.35,0.6) {$3$};
    \draw[edge] (b)--(a)--(c);
  \end{tikzpicture}
};

\node[hasse=SlateBlue!20] (b1) at (1.85,3.1) {
  \begin{tikzpicture}[tree]
    \node (a) at (0,0) {$\{1,2\}$};
    \node (b) at (0,0.7) {$3$};
    \draw[edge] (a)--(b);
  \end{tikzpicture}
};

\node[hasse=DirtyYellow!25] (b2) at (1.85,-1.1) {
  \begin{tikzpicture}[tree]
    \node (a) at (0,0) {$\{2,3\}$};
    \node (b) at (0,0.7) {$1$};
    \draw[edge] (a)--(b);
  \end{tikzpicture}
};

\node[hasse=DirtyYellow!25] (b3) at (-1.85,-1.1) {
   \begin{tikzpicture}[tree]
    \node (a) at (0,0) {$\{1,3\}$};
    \node (b) at (0,0.7) {$2$};
    \draw[edge] (b)--(a);
  \end{tikzpicture}
};

\node[hasse=EggShell!25] (c2) at (0,-3) {
  \begin{tikzpicture}[tree]
    \node (a) at (0,0) {$3$};
    \node (b) at (-0.35,0.6) {$1$};
    \node (c) at (0.35,0.6) {$2$};
    \draw[edge] (b)--(a)--(c);
\end{tikzpicture}
};

\node[hasse=SageGreen!25] (d1) at (-3.8,1) {
  \begin{tikzpicture}[tree]
    \node (a) at (0,0) {$1$};
    \node (b) at (0,0.6) {$3$};
    \node (c) at (0,1.2) {$2$};
    \draw[edge] (a)--(b)--(c);
  \end{tikzpicture}
};

\draw[edge]
 (c2)--(b3);
\draw[edge](a2)--(b2)--(c2);
 
\draw[edge]
(b3)--(d1)--(a1)--(c1)--(b1)--(a3)--(b2);
\draw[edge]
(a2)--(b1);
\end{tikzpicture}}\\
Cube  \raisebox{-5pt}{\begin{tikzpicture}
\node[inner sep=2pt]  (1) at (0,0) {\footnotesize $1$};
\node[inner sep=2pt]  (2) at (0.6,-0.05) {\footnotesize $2$};
\node[inner sep=2pt]  (3) at (1.2,0) {\footnotesize $3$};
\draw[rounded corners=10pt] (-0.3,-0.3) rectangle (1.5,0.3);
\draw (1) to[out=25,in=155] (3);
\end{tikzpicture}} &  Cube  \raisebox{-5pt}{\begin{tikzpicture}
\node[inner sep=2pt]  (1) at (0,0) {\footnotesize $1$};
\node[inner sep=2pt]  (2) at (0.6,0) {\footnotesize $2$};
\node[inner sep=2pt]  (3) at (1.2,0) {\footnotesize $3$};
\draw[rounded corners=10pt] (-0.3,-0.3) rectangle (1.5,0.3);
\draw (2)--(3);
\end{tikzpicture}}
\end{tabular}

 
\caption{Hasse diagrams of the GFO for the simplex and the hypercubes in dimension 2.}
\label{fig:Hassediag1}
\end{figure}
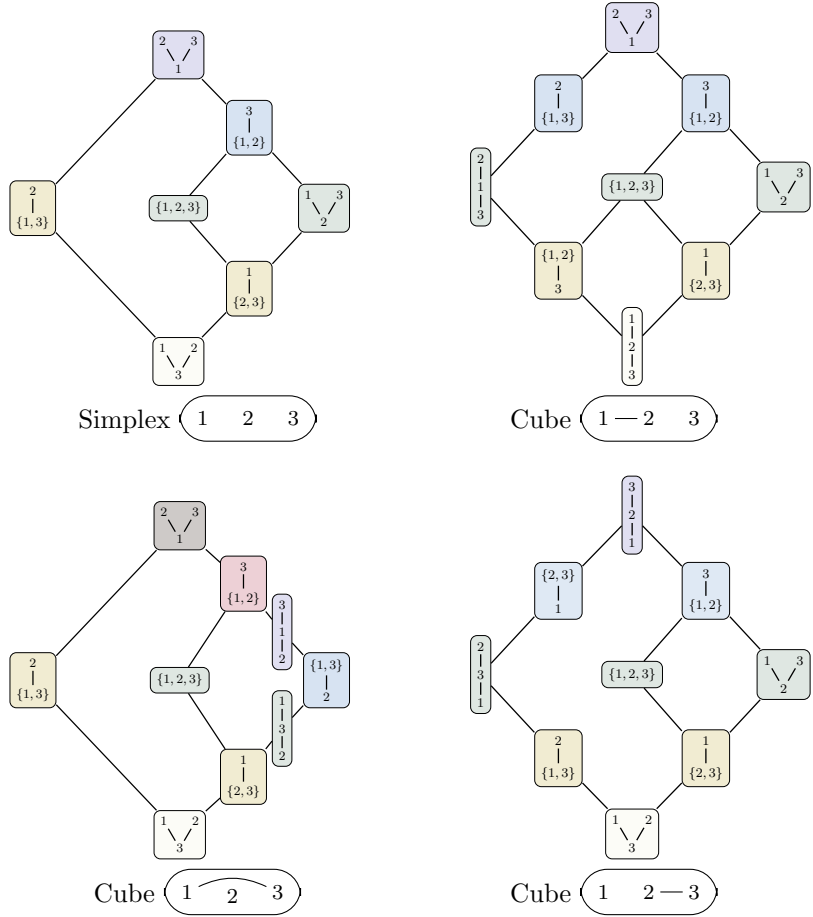
\begin{figure}[H]
\input{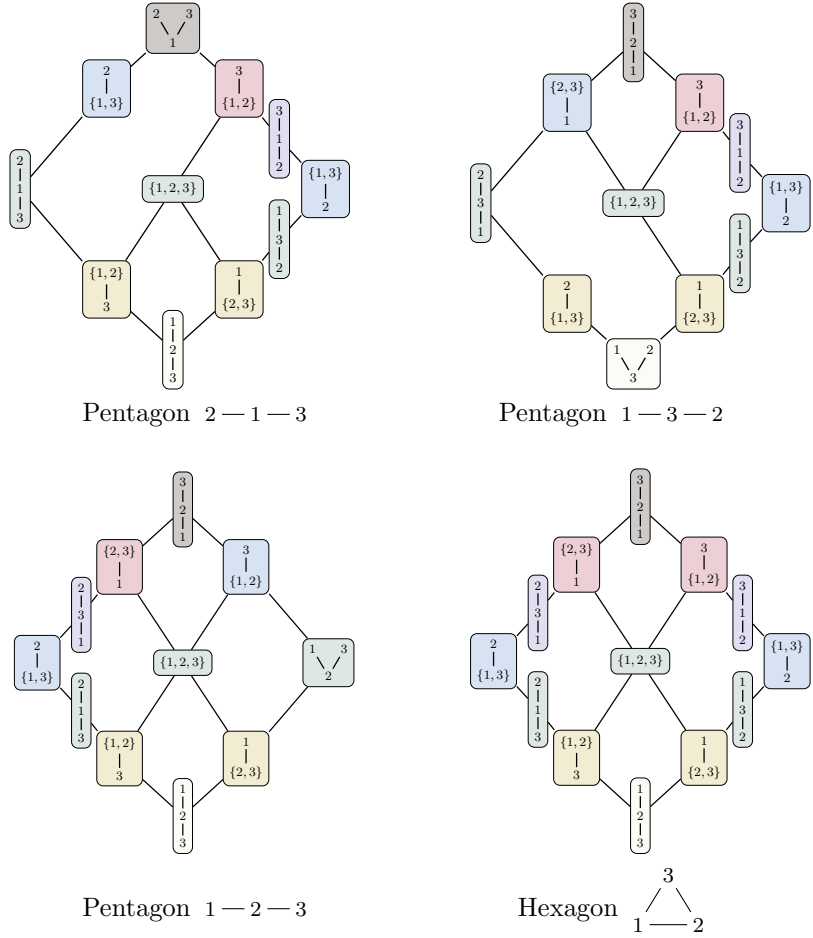}
 \caption{Hasse diagrams of the GFO for the pentagons and the hexagon in dimension $2$.}
 \label{fig:Hassediag2}
\end{figure}
\begin{proposition} \label{MaxWeakOrder}
The generalised flip order has a maximum $\top$ (resp. a minimum $\bot$), which is the construction defined recursively by choosing, at each step, the minimum (resp. the maximum) vertex of the current hypergraph as the root label.
\end{proposition}

\begin{proof}
That $\top$ is a maximal element is obvious: no node can be split and no edge can be contracted. Symmetrically, $\bot$ is minimal. We show that for any construct $S$ different from $\top$, there exists
$T$ such that $S\lessdot T$. There are two cases: if $S$ is not a construction, it has at least one node $X$ with $|X|\geq 2$, which can always be split as $X=(X\setminus\max(X))\sqcup\set{\max(X)}$. If $S$ is a construction, we claim that it has at least two nodes $a$ and $b$ such that $a$ is the parent of $b$ and $a>b$. 
We prove the claim as follows. Since $S\neq\top$, there exist nodes $a',b'$ such that $a'$ is an ancestor of $b'$ and $a'>b'$. Let $a'=a_0,a_1,\ldots,a_n=b'$ be the sequence relating $a'$ to $b'$ in $S$. Then we set $a=a_{i_0}$ and $b=a_{i_0+1}$, where $i_0$ is the first $i$ such that $a_i>a_{i+1}$, which proves the claim.
By the claim, one can construct a covering by contracting the edge relating $a$ and $b$. 

Now, by finiteness of the set of constructs of a given finite hypergraph, every chain $S=S_0\lessdot S_1\lessdot\ldots\lessdot S_n\lessdot\ldots$ must stop, 
and by the above arguments can stop only at $\top$, showing that $S\leq\top$ for every $S$, i.e. $\top$ is not only maximal but is the maximum. We get that $\bot$ is the minimum by a symmetric reasoning.
\end{proof}

We end this section by spelling out the key motivating examples of the GFO, and by giving another example of GFO (on simplices).

\begin{remark} \label{extends-PR-K}
 The generalised flip order coincides with the partial order defined on all faces of the associahedra and permutohedra in \cite{PalaciosRonco}. Indeed, via the bijective correspondences recalled in \S \ref{fundamental-examples-section}, the definition of the GFO literally instantiates to their Definition 30 and to the formulation given in their Lemma 17. We also note that, for permutohedra, the order was already present in \cite{KLNPS}. 
\end{remark}

\begin{example} 
The covering relation $\lessdot$ on the simplex $\hyper{S}^X$ can be described as follows (using the description of constructs as non-empty subsets of $X$):
$$\begin{array}{ll}
X\lessdot X\cup\set{x} & \mbox{if}\; x<X\\
X \lessdot (X\setminus\set{\max(X)}) & \mbox{if}\;|X|\geq 2.
\end{array}$$
 \end{example}

 \subsection{Generalised flip order is {\em the generalised flip order}} \label{2.2}

The order that we just defined generalises the flip order defined by Barnard and McConville \cite{BM}. We recall the definition of this order in our notations, and in the extended setting of nestohedra. Their definition was given on graph associahedra, but the upgrade to nestohedra is almost transparent.
\begin{definition} [The flip order] \label{flip-order}
Let $S:\hyper{K}$ be a construction, and let $\{x\}$ and $\{y\}$ be two adjacent nodes of $S$, such that $\{x\}$ is a parent of $\{y\}$ and $y<x$. 
 	The non-rooted planar tree $\mathsf{Flip}_S(x,y)$ is defined as follows: 
$${\mathsf{Flip}}_S(x,y) = S[x(y)\rightarrow \{x,y\}\rightarrow y(x)].$$
	The flip order   $\leq_{\operatorname{flip}}$ is defined as the reflexive and transitive closure of the covering relation $S\lessdot_{\operatorname{flip}}\mathsf{Flip}_S(x,y)$.
\end{definition}

\begin{remark} \label{natural order}
Barnard and McConville prove that the above definition gives actually a partial order --  i.e., that the transitive relation of the covering relation $S\lessdot_{\operatorname{flip}}\mathsf{Flip}_S(x,y)$ is cycle-free --
by showing that it can be alternatively defined as the order given on the 0-skeleton by any orientation vector of decreasing coordinates, for a suitable realisation of the graph associahedron as a  Minkowski sum.
It turns out that the same holds for the class of all generalised permutohedra, namely that there is a natural order on their set of vertices arising from  such orientation vectors and not depending on the particular choice of the vector provided the relative order of the coordinates is respected.  It is known that graph associahedra, through their Minkowski sum realisation, are a subclass of generalised permutohedra. It is therefore more appropriate to quote the flip order as a combinatorial  characterisation, in the case of graph associahedra, of this more general order which is implicit in the original work of Postnikov \cite{P09}, rather than as  ``patient zero".
\end{remark}

In Proposition \ref{GFO-FO} below, we shall compare the flip order with our generalised flip order. For one direction, we shall need the following definition.
\begin{definition} \label{hereditarily-ordered-def}
An ordered hypergraph $\hyper{H}$ is {\em hereditarily} ordered if for
any subset $\emptyset\neq X\inc H$, the  decomposition $\hyper{H},X\leadsto H_1,\ldots,H_n$ can be ordered in such a way that $H_1<H_2<\cdots<H_n$.
\end{definition}
\begin{remark} The notion of hereditarily ordered hypergraph appeared already, without being named, in \cite{PLBJ1}, where we defined ordered universes as being made of hypergraphs satisfying the property coined in Definition \ref{hereditarily-ordered-def}. This condition is indeed very natural, in view of the invariants needed for a correct inductive definition of the shuffle product in the ordered setting.
\end{remark}
We put Definition \ref {hereditarily-ordered-def} to immediate use in the following lemma. We have observed in \S \ref{gfo_def} that, for a construct $S:{\bf H}$ and a node $Y$ of $S$, a splitting $Y=Y_1\cup Y_2$, for $Y_1<Y_2$, is not always possible. But under the hypothesis that $\hyper{H}$ is hereditarily ordered, an
 arbitrary node partition $Y=Y_1 < Y_2$ can be dealt with by iterated  splitting. 
\begin{lemma} \label{squashing-lemma}
Let $\hyper{K}$ be a hereditarily ordered hypergraph, let $S:\hyper{K}$ and let $Y$ be a node of $S$. Introduce the shorthand ${\bf H}:={\bf K}_{\valup_S(Y)}$ and suppose that ${\bf H},Y\leadsto H_1,\dots,H_n$. Let $Y(S_1,\ldots,S_n)$, where $S_i:{\bf H}_i$ for all $1\leq i\leq n$, be the subconstruct of $S$ rooted at $Y$. Suppose that $Y=Y_1\cup Y_2$ is a partition of $Y$ such that $Y_1<Y_2$. Suppose, in addition, that ${\bf H},Y_1 \leadsto H'_1,\ldots H'_p$, let $J=\setc{i}{1\leq i\leq p \;\textrm{and} \;H'_i\cap Y_2\neq\emptyset}=\set{i_1<\cdots <i_q}$, and write $Y_2^l := Y_2\cap H'_{i_l}$. Then there exist $j_1,\ldots , j_q$ and $k_1,\ldots,k_q$, such that $0\leq j_1\leq \cdots \leq j_q$ and 
$j_i\leq k_i\leq j_{i+1}$, for all $1\leq i\leq q-1$ and $j_q\leq k_q\leq n$
$$
T'  := Y_1(S_1,\!\varlist, S_{j_1},
 Y_2^{1}(S_{j_1+1},\!\varlist,S_{k_1}),\!\varlist, S_{j_q}, Y_2^{q}(S_{j_q+1},\!\varlist,S_{k_q}),S_{k_q+1},\!\varlist, S_n)$$
is a construct of $\hyper{H}$ and
$S\lessdot^+ T$, where $T$ is obtained by replacing the subconstruct $Y(S_1,\ldots,S_n)$ with $T'$ in $S$.
\end{lemma}
\begin{proof} The claim holds by successively splitting $Y$ along the pieces given by the decomposition $Y_2=Y_2^1< \cdots < Y_2^q$ (ordered thanks to the hereditarity assumption), and applying Lemma \ref{splitformally} at each step, as follows. By construction, $Y$ is $((Y\setminus Y_2^q),Y_2^q)$-splittable. (Indeed, setting
$Y^q=Y\setminus Y_2^q$, we have on one hand $Y^q<Y_2^q$, and on the other hand $H'_{i_q}\cap Y^q=\emptyset$, which entails that $H'_{i_q}$ remains a connected component of $\hyper{H}\setminus Y^q$, and we are done since $Y_2^q\inc H'_{i_q}$ by definition.) Therefore, by Lemma \ref{splitformally}, $S[Y\rightarrow Y^q(Y_2^q)]:{\bf K}$. In turn, by the same argument, the node $Y^q$ of $S[Y\rightarrow Y^q(Y_2^q)]$ is $((Y^q\setminus Y_2^{q-1}),Y_2^{q-1})$-splittable, etc. Repeating this procedure and tracking how the subconstructs $S_1,\dots, S_n$ distribute yields precisely $T$. Being obtained from $S$ through a sequence of elementary splits, for $T$ it holds that $S\lessdot^+ T$.
\end{proof}
\begin{definition} \label{squashing-notation}
In the setting of Lemma \ref{squashing-lemma}, we write $S\squash T$ and say that $T$ is obtained from $S$ by $(Y_1,Y_2)$-{\em squashing}, and that
the partition $Y_2^1 < \cdots < Y_2^q$ is the corresponding {\em disintegration} of $Y_2$. We denote it by $T=S[Y\rightarrow Y_1\left\langle Y_2 \right\rangle]$.
\end{definition}
\begin{convention}
We extend the Convention \ref{nota} to include the action from Definition \ref{squashing-notation} as well.
\end{convention}
\begin{remark}
We note that $(Y_1,Y_2)$-splittability is a special case of squashability, when $q=1$ and $Y_2^1=Y_2$.
We also remark that the definition of $T=S[Y\rightarrow Y_1\left\langle Y_2 \right\rangle]$ makes sense even without assuming that ${\bf H}$ is ordered (and hence for any partition $Y=Y_1\cup Y_2$): indeed, it suffices to order the indices $i$ in the decomposition ${\bf H},Y\leadsto H_1,\dots,H_i,\ldots,H_n$ in such a way that 
in a planar representation of $T$ (as displayed in the statement of Lemma \ref{squashing-lemma}) these components ``appear'' from left to right.
\end{remark}
\begin{example} Consider the hypergraph ${\bf H}^\dag$ as in the second part of Example \ref{Exple27}. 
The $(\{1,2\},\{3,4\})$-squashing of the unique node of the single-node construct $S=\{1,2,3,4\}:{\bf H}^\dag$ produces the construct $\{1,2\}(3,4):{\bf H}^\dag$. The corresponding sequence of elementary splits is $\{1,2,3,4\}\lessdot \{1,2,3\}(4)\lessdot \{1,2\}(3,4)$.
\end{example}
We continue to develop the material needed for the comparison of the flip order $\leq_{\operatorname{flip}}$ with our generalised flip order $\leq$. We start by associating a multiset of natural numbers $\#(P)$ to any sequence
$P=(C_1\lessdot \cdots \lessdot C_n)$ as follows:
\begin{itemize}
\item if $n=1$, then $\#(P)$ is empty;
\item if $C_1\lessdot C_2$ is a split, then $\#(P)=\#(C_2\lessdot \cdots \lessdot C_n)$;
\item if $C_1\lessdot C_2$ is a fusion creating a node $U\cup V$, then $\#(P)=|U\cup V| +\#(C_2\lessdot \cdots \lessdot C_n)$ (here, $+$ stands for multiset addition).
\end{itemize}
We observe that, by construction, $\#(P)$ is a multiset of natural numbers that are greater than or equal to $2$: it is the multiset of cardinalities of vertices obtained by a fusion in the sequence. 

\begin{lemma}\label{claim1} If $P=(C_1\lessdot \cdots\lessdot C_n)$ is such that $C_n$ is a construction and if $Y=Y_1 \cup Y_2$, with $Y_1 <Y_2$, is a node of $C_1$, then there exists a sequence $P'=(C_1[Y\rightarrow Y_1\langle Y_2\rangle]\lessdot \cdots\lessdot C_n)$ such that $\#(P')\leq\#(P)$ in the multiset ordering\footnote{We recall that the multiset order on finite multisets $M:X\rightarrow\mathbb{N}$, where $X$ is a partially ordered set, is generated by pairs $(N,M)$ where $N$ is obtained by removing an (occurrence of) an element $x$ of $M$ and replacing it by a multiset $Y$ of elements all strictly smaller than $x$, i.e., 
$N=(M\backslash\set{x})+Y$, where $+$ denotes addition of multisets. For an example, for $X=\set{a<b<c}$ we have $\set{a,a,b,b,b,c} < \set{a,b,c,c}$, with $x=c$ and $Y=\set{a,b,b}$.}. 
\end{lemma} 
\begin{proof}
 The proof goes by induction on the length $n$ of $P$. 

\smallskip

For the base case, given by $n=2$, since
 $C_2$ is a construction, the unique covering contained in $P$ must be a split. More precisely, it must hold that $C_2=C_1[Y\rightarrow Y_1(Y_2)]$ for singletons $Y_1$ and $Y_2$. Therefore, the only possible squashing of $C_1$ is precisely the $(Y_1,Y_2)$-splitting leading to $C_2$. So, we take $P'=C_1[Y\rightarrow Y_1(Y_2)]$ and we see easily that $\emptyset= \#(P')\leq \#(P)=\emptyset$.

\smallskip

In order to address the inductive step, we distinguish four cases determined by the nature of the first covering $C_1\lessdot C_2$ of $P$.
The reader may want to illustrate the different cases with pictures such as Figure \ref{illustrative-picture} below.
\begin{enumerate}
\item If the covering $C_1\lessdot C_2$ does not involve $Y$, then 
\begin{equation}C_1[Y\rightarrow Y_1\langle Y_2 \rangle]\lessdot C_2[Y\rightarrow Y_1\langle Y_2 \rangle]\label{we0}\end{equation} is a valid sequence of coverings and $\#(\eqref{we0})=\#(C_1\lessdot C_2)$. By induction applied to the sequence $Q=(C_2\lessdot \cdots \lessdot C_n)$ with respect to the node $Y=Y_1<Y_2$ of $C_2$, we get a sequence $Q'=(C_2[Y\rightarrow Y_1\langle Y_2 \rangle  ]\lessdot \cdots \lessdot C_m)$ such that $\#(Q')\leq\#(Q)$. We define $P'$ by prefixing $Q'$ with \eqref{we0}; it follows immediately that $\#(P')\leq\#(P)$.
\item If $C_1\lessdot C_2$ is a $(Y,Z)$-fusion, then we have $Z<Y_1<Y_2$. We distinguish two cases according to the position of $Z$ in $C_1[Y\rightarrow Y_1\langle Y_2 \rangle  ]$.
\begin{enumerate}
\item If $Z$ is a child of $Y_1$, then \begin{equation} 
C_1[Y\rightarrow Y_1\langle Y_2 \rangle  ]\lessdot C_1[Y\rightarrow Y_1\langle Y_2 \rangle  ][Y_1(Z)\rightarrow Y_1\cup Z] 
\label{we1}\end{equation}
is a valid sequence of coverings, and, since $|Y_1\cup Z|<|Y\cup Z|$, we have that $\#(\eqref{we1})<\#(C_1\lessdot C_2)$.
 By induction applied to the sequence $Q=(C_2\lessdot \cdots \lessdot C_n)$ with respect to the node $Y\cup Z=(Y_1\cup Z)<Y_2$ of $C_2$, we get a sequence $Q'=(C_2[Y\cup Z \rightarrow (Y_1\cup Z)\langle Y_2\rangle]\lessdot \cdots \lessdot C_m)$ such that $\#(Q')\leq\#(Q)$. Now, observing that $$C_1[Y\rightarrow Y_1\langle Y_2 \rangle  ][Y_1(Z)\rightarrow Y_1\cup Z]=C_2[Y\cup Z \rightarrow (Y_1\cup Z)\langle Y_2\rangle],$$

we define the desired sequence $P'$ by prefixing $Q'$ with \eqref{we1}. The inequality $\#(P')<\#(P)$ follows immediately.

\item If $Y_2^1<\cdots <Y_2^q$ is the disintegration of $Y_2$ in $D:=C_1[Y\rightarrow Y_1\langle Y_2 \rangle  ]$ and if $Z$ is a child of $Y_2^1$, then
\begin{equation}
D \lessdot^+ D[Y_2^1(Z)\rightarrow Y_2^1\!\cup\! Z \rightarrow Z\langle Y_2^1\rangle][Y_1(Z)\rightarrow Y_1\!\cup\! Z]
\end{equation}\label{we2}
is a valid sequence of coverings. In addition, since $|Y_2^1\cup Z|<|Y\cup Z|$ and $|Y_1\cup Z|<|Y\cup Z|$ it holds that $\#(\eqref{we2})<\#(C_1\lessdot C_2)$. 
 We conclude again by induction, analogously as in the previous case, noticing that the last construct in \eqref{we2} is precisely $C_2[Y\cup Z \rightarrow (Z\cup Y_1)\langle Y_2\rangle]$.
\end{enumerate}
\item If $C_1\lessdot C_2$ is a $(Z,Y)$-fusion, then it holds that $Y_1<Y_2<Z$ and
\begin{equation} 
 C_1[Y\rightarrow Y_1\langle Y_2 \rangle  ] \lessdot^+ C_1[Y\rightarrow Y_1\langle Y_2 \rangle  ][Z(Y_1) \rightarrow Z\cup Y_1\rightarrow Y_1\langle Z\rangle] \label{we3}
\end{equation}
is a valid sequence of coverings. Moreover, since $|Z\cup Y_1|<|Z\cup Y|$, we have that $\#(\eqref{we3})<\#(C_1\lessdot C_2)$. Denote by $S$ the right-hand side of \eqref{we3} and let $Y_2^1<\cdots <Y_2^q$ and $Z^1<\cdots <Z^r$ be the corresponding disintegrations of $Y_2$ and $Z$, respectively. We distinguish two cases.
\begin{enumerate}
\item If each node $Y_2^l$, $1\leq l\leq q$, is a child of $Y_1$ in $S$, then $$S=C_2[ Y\cup Z \rightarrow Y_1\langle Y_2\cup Z \rangle],$$ and we conclude again by induction.
\item Suppose that for some $1\leq j\leq q$, each node $Y_2^l$ for $j\leq l\leq q$ is a child of $Z^1$ in $S$. Then we continue the sequence \eqref{we3} by performing the $(Z^1,Y_2^j)$-fusion, followed by the $((Z^1\cup Y_2^j),Y_2^{j+1})$-fusion, etc., until the last node $Y_2^q$ gets fused with $Z^1$, leading to 
\begin{equation}
 S\lessdot^+ S[Z^1(Y_2^j)\rightarrow Z^1\cup Y_2^j]\cdots[(Z^1\cup\!\!\bigcup_{j\leq l\leq q-1}Y_2^l)(Y_2^{q})\rightarrow Z^1\cup\!\!\bigcup_{j\leq l\leq q}Y_2^l].
 \label{we4}
\end{equation}
Denoting with $T$ the right-hand side of \eqref{we4}, we have that 
$$T=C_2[Y\cup Z \rightarrow Y_1\langle Y_2\cup Z\rangle ],$$
 and we define $P'$ by concatenating \eqref{we3}, \eqref{we4} and $Q'$, where $Q'$ is obtained by induction like before. 
\end{enumerate}
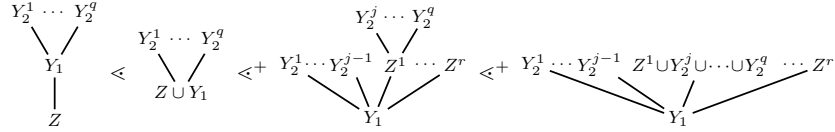
\begin{figure}[H]
\centering
\resizebox{11cm}{!}{\begin{tikzpicture}
\node (Z) [inner sep=0.6mm] at (0,-0.1) {\footnotesize $Z$}; 
\node (Y1)[inner sep=0.6mm] at (0,0.75) {\footnotesize $Y_1$}; 
\node (Y21)[inner sep=0.6mm] at (-0.5,1.6) {\footnotesize $Y^1_2$}; 
\node (Y2q)[inner sep=0.6mm] at (0.5,1.6) {\footnotesize $Y^q_2$};
\node (do)[inner sep=0.6mm] at (0,1.6) {\scriptsize $\cdots$};
\draw[thick] (Z)--(Y1); \draw[thick] (Y21)--(Y1)--(Y2q);
\end{tikzpicture} 
\raisebox{0.8cm}{$\lessdot$} \raisebox{0.4cm}{\begin{tikzpicture}
\node (Z) [inner sep=0.6mm] at (0,0.65) {\footnotesize $Z\cup Y_1$}; 
\node (Y21)[inner sep=0.6mm] at (-0.5,1.5) {\footnotesize $Y^1_2$}; 
\node (Y2q)[inner sep=0.6mm] at (0.5,1.5) {\footnotesize $Y^q_2$};
\node (do)[inner sep=0.6mm] at (0,1.5) {\scriptsize $\cdots$};
 \draw[thick] (Y21)--(Z)--(Y2q);
\end{tikzpicture}} \raisebox{0.8cm}{$\lessdot^{+}$} \begin{tikzpicture}
\node (Z) [inner sep=0.6mm] at (0,-0.1) {\footnotesize $Y_1$}; 
\node (Y21)[inner sep=0.6mm] at (-1.3,0.75) {\footnotesize $Y^1_2$}; 
\node (Y2q)[inner sep=0.1mm] at (-0.35,0.75) {\footnotesize $Y^{j-1}_2$};
\node (do)[inner sep=0.6mm] at (-0.9,0.75) {\scriptsize $\cdots$};
\node (Z1) [inner sep=0.6mm] at (0.35,0.75) {\footnotesize $Z^1$}; 
\node (Y2j)[inner sep=0.6mm] at (-0.1,1.5) {\footnotesize $Y^j_2$}; 
\node (Y2q')[inner sep=0.1mm] at (0.8,1.5) {\footnotesize $Y^{q}_2$};
\node (do1)[inner sep=0.6mm] at (0.35,1.5) {\scriptsize $\cdots$};
\node (Zr) [inner sep=0.6mm] at (1.3,0.75) {\footnotesize $Z^r$}; 
\node (do1)[inner sep=0.6mm] at (0.85,0.75) {\scriptsize $\cdots$};
\draw[thick] (Z1)-- (Z)--(Zr); \draw[thick] (Y21)--(Z)--(Y2q); \draw[thick] (Y2j)-- (Z1)--(Y2q');
\end{tikzpicture} \raisebox{0.8cm}{$\lessdot^{+}$} \begin{tikzpicture}
\node (Z) [rectangle, inner sep=0.6mm] at (0,-0.1) {\footnotesize $Y_1$}; 
\node (Y21)[inner sep=0.6mm] at (-2.3,0.75) {\footnotesize $Y^1_2$}; 
\node (Y2q)[inner sep=0.1mm] at (-1.25,0.75) {\footnotesize $Y^{j-1}_2$};
\node (do)[inner sep=0.6mm] at (-1.825,0.75) {\scriptsize $\cdots$};
\node (Z1) [inner sep=0.6mm] at (0.35,0.75) {\footnotesize $Z^1\!\cup \! Y^j_2 \!\cup\!\cdots\! \cup\! Y_2^q$}; 
\node (Zr) [inner sep=0.6mm] at (2.3,0.75) {\footnotesize $Z^r$}; 
\node (do1)[inner sep=0.6mm] at (1.85,0.75) {\scriptsize $\cdots$};
\draw[thick] (Z1)-- (Z)--(Zr); \draw[thick] (Y21)--(Z)--(Y2q); 
\end{tikzpicture}} 
\caption{Case 3.(b). }
\label{illustrative-picture}
\end{figure}
\item If $C_1\lessdot C_2$ is the $(U_1,U_2)$-splitting of $Y$, then there are three subcases.
\begin{enumerate}
\item If $U_1=Y_1$ and $U_2=Y_2$, then $ C_1[Y\rightarrow Y_1\langle Y_2 \rangle  ]=C_2$ and the conclusion is immediate.
\item If $U_1=Y_1\cup Z$ and $Y_2=Z\cup U_2$, with $Y_1<Z<U_2$, then since $U_2\inc Y_2$ and $\max(U_2)=\max(Y_2)$, we have that $U_2$ is a subset of the last block $Y_2^q$ of the disintegration of 
$Y_2$. There are again two cases.
\begin{enumerate}
\item If $U_2=Y_2^q$, then $C_1[Y\rightarrow Y_1\langle Y_2\rangle]=C_2[U_1\rightarrow Y_1\langle Z\rangle]$ and we conclude by induction.
\item If $Y_2^q=V\cup U_2$, with $V<U_2$, then one checks easily that $(U_1,U_2)$-splittability of $Y$ in $C_1$ entails $(V,U_2)$-splittability of $Y_2^q$ in $C_1[Y\rightarrow Y_1\langle Y_2\rangle]$. We get that $C_1[Y\rightarrow Y_1\langle Y_2\rangle][Y_2^q\rightarrow V(U_2)]=C_2[U_1\rightarrow Y_1\langle Z\rangle]$ and we conclude likewise by induction. 
\end{enumerate}
\item If $Y_1=U_1\cup Z$ and $U_2=Z\cup Y_2$, with $Y_1<Z<U_2$, we note similarly as in the previous case that $Y_1$ is $(U_1,Z)$-splittable in $Y_1$, and that the construct $T'_2$ obtained from $T'_1$ by this splitting is the $(Z,Y_2)$-squashing of $U_2$ in $T_2$. We conclude by induction.
\end{enumerate}
 \end{enumerate} 
This concludes the proof of the lemma.
\end{proof}

\begin{lemma}\label{claim2} For every $P=(C_1\lessdot \cdots\lessdot C_n)$ such that $C_1$ and $C_n$ are constructions and $\#(P)$ contains at least one occurrence of $p>2$, there exists a sequence $P'$ parallel to $P$ (i.e., starting from $C_1$ and ending in $C_n$) such that $\#(P')<\#(P)$ in the multiset ordering.
\end{lemma}
\begin{proof} 
Since $\#(P)$ contains an occurrence of $p>2$, there exists an $1\leq m\leq n-1$, such that the covering $C_m \lessdot C_{m+1}$ is an $(X,Y)$-fusion, with $|X\cup Y|>2$. Let us introduce the notation $P^l:=(C_1\lessdot \cdots\lessdot C_m)$ and $P^r:=(C_{m+1}\lessdot \cdots\lessdot C_n)$. We note that $\#(P^r)<\#(P)$. Now, since $|X\cup Y|>2$, we have either $|X|>1$ or $|Y|>1$.
\begin{enumerate}
\item If $|Y|>1$, then $Y$ is $(Y_1,Y_2)$-splittable for some pair $(Y_1,Y_2)$; in addition, since $C_m \lessdot C_{m+1}$ was an $(X,Y)$-fusion, we have that $Y_1<Y_2<X$. Consider the following sequence of coverings:
\begin{equation}
C_m \lessdot^+ C_m[Y\rightarrow Y_1(Y_2)][X(Y_1)\rightarrow X\cup Y_1 \rightarrow Y_1\langle X \rangle].
\label{we10}
\end{equation}
Denote by $S$ the right-hand side of \eqref{we10} and let $X^1<\cdots <X^q$ be the corresponding disintegration of $X$. We distinguish two cases according to the position of $Y_2$ in $S$.
\begin{enumerate}
\item If $Y_1$ is the parent of $Y_2$, then $S=C_{m+1}[X\cup Y\rightarrow Y_1\langle X\cup Y_2\rangle]$. By Lemma \ref{claim1}, we get a sequence $\widetilde{P^r}=(S\lessdot \cdots \lessdot C_m)$ such that
$\#(\widetilde{P^r})\leq \#(P^r)$. We define $\widetilde{P}$ to be the concatenation of $P^l$, \eqref{we10} and $\widetilde{P^r}$. Noticing that
$\#\eqref{we10} < \#(C_m\lessdot C_{m+1})$, we verify
$$\begin{array}{lllll}
\#(\widetilde{P}) & = &\#(P^l)+\#\eqref{we10}+\#(\widetilde{P^r}) &&\\
& < & \#(P^l)+\#(C_m\lessdot C_{m+1})+\#(P^r) &=&\#(P).
\end{array}$$

\item Otherwise, $Y_2$ is a child of $X^1$. Then we continue \eqref{we10} as follows: 
\begin{equation}
S \lessdot^+ S[X_1(Y_2)\rightarrow X^1\cup Y_2\rightarrow Y_2\langle X^1\rangle].\label{we11}
\end{equation}
Since $S[X_1(Y_2)\rightarrow X^1\cup Y_2]=C_{m+1}[X\cup Y\rightarrow Y_1 \langle X\cup Y_2\rangle ]$, denoting by $S'$ the right-hand side construct of \eqref{we11}, we have that $$S'=C_{m+1}[X\cup Y \rightarrow Y_1 \langle X\cup Y_2\rangle ][X^1\cup Y_2\rightarrow Y_2\langle X^1\rangle].$$ Applying Lemma \ref{claim1} twice, we get
a sequence $\widetilde{P^r}=(S'\lessdot \cdots\lessdot C_m)$ such that $\#(\widetilde{P^r})\leq\#(P^r)$. We define $\widetilde{P}$ to be the concatenation of $P^l$, \eqref{we10}, \eqref{we11} and $\widetilde{P^r}$, and we conclude as in the previous case.
\end{enumerate}
\item If $|X|>1$, then $X$ is $(X_1,X_2)$-splittable for some pair $(X_1,X_2)$, and we have $Y<X_1<X_2$. Consider the position of $Y$ in $D:=C_m[X\rightarrow X_1(X_2)]$.
\begin{enumerate}
\item If $Y$ is a child of $X_2$, then 
\begin{equation}
C_m \lessdot^+ D[X_2(Y) \rightarrow X_2\!\cup\! Y \rightarrow Y\langle X_2\rangle][X_1(Y)\rightarrow X_1\!\cup \!Y].
\label{we12}
\end{equation}

Denoting by $S$ the right-hand side construct of \eqref{we12}, we have that $S=C_{m+1}[X\cup Y \rightarrow (X_1\cup Y)\langle X_2\rangle]$ and we conclude analogously as before. 
\item If $Y$ is a child of $X_1$, then 
\begin{equation}
C_m \lessdot^+ C_m[X\rightarrow X_1(X_2)][X_1(Y) \rightarrow X_1\cup Y\rightarrow Y\langle X_1\rangle].\label{we13}
\end{equation}
Denoting by $S$ the right-hand side construct of \eqref{we13} and the corresponding disintegration of $X_1$ by $X_1^1<\cdots <X_1^q$, we distinguish two cases according to the position of $X_2$ in $S$.
\begin{enumerate}
\item If $X_2$ is a child of $Y$ in $S$, then $$S=C_{m+1}[X\cup Y \rightarrow Y\langle X\rangle],$$ and we conclude by Lemma \ref{claim1}, as in  case 1(a).
\item If $X_2$ is a child of $X_1^q$, then $$S=C_{m+1}[X\cup Y \rightarrow Y\langle X\rangle ][X_1^q\cup X_2\rightarrow X_1^q(X_2)],$$ and we conclude by using  Lemma \ref{claim1} twice, as in  case 1(b). 
\end{enumerate}
\end{enumerate}
\end{enumerate}
This concludes the proof of the lemma.
\end{proof}

\begin{proposition} \label{GFO-FO}
Let $\hyper{H}$ be an ordered hypergraph and let $S,T:\hyper{H}$ be two constructions.
The following two properties hold:
\begin{enumerate}
\item If $S\lessdot_{\operatorname{flip}} T$, then $S\lessdot^+ T$.
\item If $\hyper{H}$ is hereditarily ordered and if $S\lessdot^+ T$, then $S \lessdot_{\operatorname{flip}}^+ T$.
\end{enumerate} 
\end{proposition}

\begin{proof}
The first property is immediate: if $S\lessdot_{\operatorname{flip}} {\mathsf{Flip}}_S(x,y)$, then we have  
$$S \lessdot S[x(y)\rightarrow\{x,y\}]\lessdot S[x(y)\rightarrow \{x,y\}\rightarrow y(x)]={\mathsf{Flip}}_S(x,y).$$

The proof of the converse direction consists in showing how to transform progressively a covering
sequence $C_1\lessdot \cdots\lessdot C_n$, where $C_1$ and $C_n$ are constructions, into one whose successive pairs of coverings are of the form of a fusion followed by a splitting of the resulting node. In the first step, by applying Lemma \ref{claim2} repeatedly (and by well-foundedness of the multiset ordering), we transform the starting covering sequence into a sequence $P=(C_1\lessdot \cdots\lessdot C_n)$
such that $\#(P)=\{2,\ldots,2\}$, i.e., such that all the fusions of $P$ are contractions of edges between singleton nodes. If $n\geq 2$, the first covering $C_1\lessdot C_2$ must be a $(\set{a},\set{b})$-fusion, since no splitting can be applied to a construction. We note that, by minimality of $\#(P)$, the node $X=\set{a,b}$ cannot be part of a fusion in the rest of the sequence; instead, since $C_n$ is a construction, $X$ must eventually be split, at the latest in the final covering step $C_{n-1}\lessdot C_n$. 
So there is a $(\set{b},\set{a})$-splitting step $C_j\lessdot C_{j+1}$, such that all the steps $C_i \lessdot C_{i+1}$, for $2\leq i<j$, leave $X$ untouched and hence commute with the $(\set{b},\set{a})$-splitting step. Therefore, we can rearrange $P$ (without changing its length $n$ or its weight $\#(P)$) by moving the $(\set{b},\set{a})$-splitting step leftwards until it follows immediately after the $(\set{a},\set{b})$-fusion, so that the first two covering steps in the resulting sequence
$P_1=(C_1\lessdot C_2 \lessdot C'_3 \cdots\lessdot C'_n=C_n)$ form exactly a flip of $a$ and $b$. Moreover, $C'_3$ is again a construction. 
We can repeat this procedure with the shorter sequence $(C'_3 \lessdot\cdots\lessdot C'_n=C_n)$, and so on, until we reach a sequence $P_m=(T_1\lessdot T_2 \lessdot T_3\lessdot \cdots\lessdot T_{2m+1}$) with $2m+1=n$, $T_1=C_1$, $T_2=C_2$, $T_3=C_3'$ and $T_{2m+1}=C_n$. This sequence reads as $C_1 \lessdot_{\operatorname{flip}} T_3 \lessdot_{\operatorname{flip}} \cdots \lessdot_{\operatorname{flip}} T_{2m-1}\lessdot_{\operatorname{flip}} C_n$, which completes the proof.
 \end{proof}
\begin{remark} \label{subsume-BM}
As an immediate corollary of Proposition \ref{lessdot-order}
and Proposition \ref{GFO-FO}, we get an alternative proof of Lemma 2.8 of \cite{BM} (see also Proposition 1 in \cite{CLA2}), stating that the
 reflexive and transitive closure of the flipping relation $\lessdot_{\operatorname{flip}}$ of Definition \ref{flip-order} is a partial order.
 \end{remark}

\subsection{Generalised flip order in terms of generalised inversions} \label{inversion-subsection}
 
For {\em right-filled} graphs, i.e., graphs ${\bf G}=(V,E)$ with the property that, if $\{i,k\}\in E$ then $\{j,k\}\in E$ for each $1\leq i<j<k\leq n$, Barnard and McConville also provide a characterisation of $\leq_{\operatorname{flip}}$ in terms of {\em inversions}, as follows. For a construction $T:{\bf G}$ and nodes $a$ and $b$ of $T$, they define $b<_{T} a$ if the unique
path from $b$ to the root of $T$ passes through $a$ and
$$\Inv(T):=\{(a,b)\,|\, a<b \mbox{ and } b<_{T} a\},$$
and they prove that  $S\leq_{\operatorname{flip}} T$  if and only if $\Inv(S)\subseteq \Inv(T)$. 

\begin{remark}
The above characterisation   does not apply to non-right-filled graphs, such as the pentagon \raisebox{-2pt}{\begin{tikzpicture}
\node[inner sep=2pt] (1) at (0,0) {\footnotesize $2$};
\node[inner sep=2pt] (2) at (0.6,0) {\footnotesize $1$};
\node[inner sep=2pt] (3) at (1.2,0) {\footnotesize $3$};
\draw (1)--(2)--(3);
\end{tikzpicture}} (see Figure \ref{fig:Hassediag2}).
 Indeed, while  $2(1(3))\leq_{\operatorname{flip}} 1(2,3)$, the corresponding sets of inversions are incomparable:  $\Inv(2(1(3)))=\{(1,3),(2,3)\}$  and $\Inv(1(2,3))=\{(1,2),(1,3)\}$.  
\end{remark}

Our comparison of the vertex-restriction of the generalised flip order with the flip order and the work of Barnard and McConville suggest that we should be able to provide a characterisation of the former in terms of (generalised) inversions. We now give such a characterisation, by generalising the chain of arguments used in \cite{BM}. 
\smallskip

 For a construct $S:{\bf H}$, $a,b\in H$ and the nodes $A$ and $B$ of $S$ such that $a\in A$ and $b\in B$, write
\begin{itemize}
\item $b<_S a$ if the unique path from $B$ to the root of $S$ passes through $A$, and

\item $a=_S b$ if $A=B$,
\end{itemize}
and define
$$\Inv_{\bf H}(S):=\{(a,b)\in H\times H\,|\,a<b \mbox{ and } b<_S a\},$$
  $$\Invt_{\bf H}(S):=\{(a,b)\in H\times H\,|\,a<b \mbox{ and } a=_S b\}.$$
We refer to the elements of $\Inv_{\bf H}(S)\cup \Invt_{\bf H}(S)$ as {\em generalised inversions}\footnote{The term ``generalised inversion'' is justified by the encoding $\delta$ of constructs of
a hereditarily ordered hypergraph ${\bf H}$ as ordered partitions, defined recursively as follows: if $S=X(S_1,\dots,S_n):{\bf H}$, 
we set
$\delta(S)=(\delta(S_1),\dots,\delta(S_n),X)$ (this encoding corresponds to a postorder traversal of $S$). Then $b<_S a$ reads as an inversion in $\delta(S)$ in the classical sense, whereas $a=_S b$ means that $a$ and $b$ belong to the same block of the ordered partition.}. In addition, the elements of $\Inv_{\bf H}(S)$ (resp. $\Invt_{\bf H}(S)$) will be called {\em good pairs} (resp. {\em neutral pairs}) of $S$.
 If $(a,b)$ are vertices of ${\bf H}$ such that $a<b$ and $a<_S b$, we say that $(a,b)$ is a {\em bad pair} of $S$. Finally, we say that $a$ and $b$ are incomparable in $S$ if neither $b\leq_S a$ nor $a<_S b$ holds.
\smallskip

\begin{definition}

An ordered hypergraph ${\bf H}$ is {\em right-filled} if, for any triple of vertices $a,b,c\in H$ such that $a<b<c$ and any hyperedge $X\in {\bf H}$ such that $a,c\in X$, there exists a hyperedge $Y\in{\bf H}$ such that $\{b,c\}\subseteq Y\subseteq (X\backslash\{a\})\cup\{b\}$. 
\end{definition}

\begin{remark}
As a side illustration of the previous definition, we note that if ${\bf H}$ is right-filled, and if $E=\set{x_1<\cdots<x_n}$ is any hyperedge of ${\bf H}$, then all sets $E_i=\set{x_i<\cdots<x_n}$, for $1<i<n$, are connected in ${\bf H}$: indeed, applying the definition to $E$ and to $x_i,x_{i+1},x_n$, we get hyperedges $Y_i$ such that
$\set{x_{i+1},x_n}\inc Y_i \inc E_{i+1}$, from which the conclusion follows easily. So we have that ${\bf H}_E$ as a polytope is at least as truncated as the hypercube $\hyper{C}^{\set{x_n<\cdots<x_1}}$. 
 \end{remark}

In the following easy lemma, we show that the property of being right-filled is inherited by subhypergraphs (including not necessarily connected ones). 

\begin{lemma}\label{s000}
If ${\bf H}$ is an ordered right-filled hypergraph and  $K\subseteq H$, then the hypergraph ${\bf H}_{K}$ is also right-filled (with respect to the induced ordering on $K$).
\end{lemma}
\begin{proof}
Let $a,b,c\in K$ be such that $a<b<c$ and let $X$ be a hyperedge of ${\bf H}_{K}$ such that $a,c\in X$. Since ${\bf H}_{K}\subseteq {\bf H}$, by the right-filled property of ${\bf H}$, there exists a hyperedge $Y\in{\bf H}$ such that $\{b,c\}\subseteq Y\subseteq (X\backslash\{a\})\cup\{b\}$. The conclusion follows by the definition of ${\bf H}_{K}$, since $(X\backslash\{a\})\cup\{b\}\subseteq K$.
\end{proof}

\begin{lemma}\label{s00}
Suppose that ${\bf H}$ is an ordered right-filled hypergraph, that  $K\inc H$ is connected in ${\bf H}$ and that, for some $a,b,c\in H$, it holds that $a,c\in K$, $b\not\in K$ and $a<b<c$. Then $K\cup \{b\}$ remains connected in ${\bf H}$.
\end{lemma}
\begin{proof}
 Let $E_1,\dots ,E_{m-1}$ be a chain of hyperedges of ${\bf K}:={\bf H}_K$ connecting $a$ and $c$. Then there must be at least one hyperedge $E_i$ such that $E_i$ contains elements $x$ and $y$ such that $x<b<y$. By the right-filled property applied to $E_i$, there exists a hyperedge $E$ such that $\{b,y\}\subseteq E\subseteq (E_i\backslash\{x\})\cup \{b\}$, making $K\cup E$, and hence also $K\cup\{b\}$, connected in $\hyper{H}$.
\end{proof}

 \begin{remark}
By Lemma \ref{s00}, if an ordered hypergraph ${\bf H}$ is right-filled, then it is hereditarily ordered. In what follows, the notation $C=X(C_1,\dots,C_n):{\bf H}$ will also imply that ${\bf H},X\leadsto H_1<\cdots <H_n$ and $C_i:{\bf H}_i$. 
\end{remark}

\begin{example} The hypergraph $\hyper{H}^\maltese$ is hereditarily ordered, but not right-filled (since $1,3\in \{1,2,3\}\in{\bf H}$ and $1<2<3$, right-filledness would require that $\{2,3\}\in{\bf H}$). The hypergraph  $\hyper{H}^{\text{\ding{170}}}$ is right-filled. The Hasse diagram of the GFO for $\hyper{H}^{\text{\ding{170}}}$ can be found in Figure \ref{hassediag5} in the Appendix. 
\end{example}

\begin{lemma}\label{s11}
Suppose that ${\bf H}$ is a connected ordered right-filled hypergraph and let $a,b\in H$ and $X\subsetneq H$ be such that $X<\{b\}<\{a\}$. Then there exists a chain of hyperedges of ${\bf H}$ connecting $b$ and $a$ that avoids all the vertices from $X$.
\end{lemma}
\begin{proof}
Since ${\bf H}$ is connected, there exists a chain of hyperedges $E_1,\dots ,E_m$ of ${\bf H}$ connecting $b$ and $a$.  
We prove the claim by induction on $m$.
\begin{itemize}
\item Suppose that $m=1$, i.e., that $a,b\in E_1$. If $E_1\cap X=\emptyset$, the desired chain is trivially $E_1$. If $E_1\cap X\neq\emptyset$, since $x<b<a$ for all $x\in X$, we can successively apply the right-filled property in order to eliminate all $x\in X$ from $E_1$, leading to a hyperedge $E^{\ast}_1$ such that $\{a,b\}\subseteq E^{\ast}_1\subseteq E_1\backslash X$. 
\item Suppose that $m>1$ and pick $v\in E_{m-1}\cap E_m$. Now compare $v$ and $b$.
\begin{itemize}
 \item If $v<b<a$, then, by the right-filled property for $E_m$, there exists a hyperedge containing both $a$ and $b$, and we proceed as in the base case.
\item If $b< v$, we consider the chain $E_1,\dots ,E_{m-1}$ connecting $b$ and $v$. By induction, there exists a chain $\hat{E}$ connecting them which avoids $X$. Now consider the hyperedge $E_m$. If $X\cap E_m= \emptyset$, the desired chain is obtained by concatenating $\hat{E}$ with $E_m$. If $X\cap E_m\neq \emptyset$, we can eliminate all the $x$'s from $E_m$ just like we did in the base case, leading to a hyperedge $E^{\ast}_m$ such that $\{a,v\}\subseteq E^{\ast}_m\subseteq E_m\backslash X$. The desired chain is now obtained by concatenating $\hat{E}$ with $E^{\ast}_m$.\vspace{-0.5cm}
\end{itemize}
\end{itemize}
\end{proof}

\begin{lemma}\label{eee}
For constructs $S,T:{\bf H}$ it holds that $S=T$ if and only if $\Inv_{\bf H}(S)=\Inv_{\bf H}(T)$ and $\Invt_{\bf H}(S)=\Invt_{\bf H}(T)$.
\end{lemma}
\begin{proof}
The equality of constructs directly implies the equality of the corresponding sets of good and neutral pairs. The other direction is proven by induction on $|H|$. The equality $\Invt_{\bf H}(S)=\Invt_{\bf H}(T)$ implies that the partitions of the set $H$ determined by the nodes of $S$ and $T$ match. Since $\Inv_{\bf H}(S)=\Inv_{\bf H}(T)$, it must also be the case that $\text{root}(S)=\text{root}(T)$, and the conclusion then follows by induction. 
\end{proof}

\begin{remark}\label{nor}
The map from constructs to their sets of neutral pairs is not monotone. The failure of monotonicity is evident already on the smallest example given by ${\bf H}:=\{\{1\},\{2\},\{1,2\}\}$, for which we have $\{1,2\}\lessdot 1(2)$ and $\Invt_{\bf H}(\{1,2\})=\{(1,2)\} \nsubseteq \emptyset = \Invt_{\bf H}(1(2))$. Nevertheless, the neutral pair $(1,2)$ of the construct $\{1,2\}$ is not ``lost'' when making the split leading to the construct $1(2)$ - it now appears as a good pair. 
 \end{remark}
Remark \ref{nor} suggests that the characterisation of the GFO in terms of inversions (i.e., good pairs) and neutral pairs must consider these two kinds of pairs {\em together}, in some structured way. The following two lemmas pave the way to our characterisation result.

\begin{lemma}\label{s14}
Let ${\bf H}$ be an ordered right-filled hypergraph and let $S,T:{\bf H}$ be such that $\Inv_{\bf H}(S)= \Inv_{\bf H}(T)$ and $\Invt_{\bf H}(S)\subsetneq \Invt_{\bf H}(T)$. Then there exists a construct $U:{\bf H}$ such that $U\lessdot^+ T$, $\Inv_{\bf H}(S)= \Inv_{\bf H}(U)$ and $\Invt_{\bf H}(S)\subseteq\Invt_{\bf H}(U)\subsetneq \Invt_{\bf H}(T)$. 
\end{lemma}

\begin{proof}
 Pick $(a,b)\in \Invt_{\bf H}(T)\backslash \Invt_{\bf H}(S)$ and let $X=\{x_1<\cdots <x_n\}$ be the vertex of $T$ such that $a=x_i$ and $b=x_j$ for some $1\leq i<j\leq n$.

We define $U:=T[X\rightarrow B\langle A\rangle]$, 
 where the choice of decomposition $X=B\cup A$ is such that $\{x_1,\dots,a\}\subseteq A$, $\{b,\cdots,x_n\}\subseteq B$ and for $x\in \{x_{i+1},\dots,x_{j-1}\}$, we set $x\in A$ if $(a,x)$ is a neutral pair of $S$; otherwise, we set $x\in B$. By construction, $(a,b)\not\in \Invt_{\bf H}(U)$. In addition, no new neutral pairs are produced by the squashing that defines $U$, so we have that $\Invt_{\bf H}(U)\subsetneq \Invt_{\bf H}(T)$. 

\smallskip

Suppose that the subconstruct of $T$ rooted at $X$ is given by $X(C_1,\dots,C_t):{\bf K}$, where ${\bf K},X\leadsto K_1,\cdots ,K_t$ and $C_i:{\bf K}_i$ and let $A^1<\cdots <A^q$ be the disintegration of $A$ involved in the $(B,A)$-squashing of $X$ that defines $U$. For each $1<j<q$, write $I_j\subseteq \{1,\dots,t\}$ for the set of indices such that, for $i\in I_j$, the set $A^j\cup K_i$ is not connected in ${\bf K}\backslash B$ and $A^j<K_i$. In order to show that $U$ has other desired properties, we use the following three claims.
\begin{itemize}
\item[($\heartsuit$)] For each $1\leq k<l\leq n$, if $(x_k,x_l)$ is neutral in $S$, then, for all $k\leq k' < l' \leq l$, $(x_{k'},x_{l'})$ is also neutral in $S$.
\item[($\clubsuit$)] If $(x,y)\in A^r\times A^s$ for some $1<r<s<q$, then $(x,y)$ is not neutral in $S$.
\item[($\spadesuit$)] If $(u,v)\in A^j\times K_i$ for some $1<j<q$ and $i\in I_j$, then $(u,v)$ is not good in $S$.
\end{itemize}
 
We first note that the property ($\heartsuit$) implies that $A<B$. Indeed, if there would exist $k<l$ in the interval $[i+1,j-1]$ such that $x_k\in B$ and $x_l\in A$, then, by the definition of $U$, we would have that $(a,x_l)$ is neutral in $S$ while $(a,x_k)$ is not, but this is impossible by ($\heartsuit$), as it requires $(a,x_k)$ to also be neutral in $S$. Having established that $A<B$, we conclude that $T$ can be recovered from $U$ by successive fusions of the $(A^i)$'s with their parent node, which entails that $U\lessdot^{+} T$.

 Next, the properties ($\heartsuit$) and ($\clubsuit$) guarantee that no neutral pair of $S$ gets abolished in $U$, i.e., that $\Invt_{\bf H}(S)\subseteq\Invt_{\bf H}(U)$. To see this, we observe that, having that $\Invt_{\bf H}(S)\subseteq\Invt_{\bf H}(T)$, there are two scenarios in which a neutral pair $(x,y)$ of $S$ would hypothetically be abolished in $U$:
\begin{itemize}
\item $x\in A_r$ and $y\in A_s$ for some $1<r<s<q$: this is impossible by ($\clubsuit$);
\item $x\in A$ and $y\in B$ - in this case, we further distinguish subcases relative to the exact position of $x$ and $y$ in $X$:
\begin{itemize}
\item if $x\leq a$ and $y\geq b$, then, by ($\heartsuit$), $(a,b)$ would have to be neutral in $S$, which contradicts the assumption that $(a,b)\in \Invt_{\bf H}(T)\backslash \Invt_{\bf H}(S)$;
\item if $x\leq a$ and $y=x_k$ for some $k\in[i+1,j-1]$, then the contradiction arises as the definition of $B$ implies that $(a,y)$ is not neutral in $S$, while, at the same time, ($\heartsuit$) requires $(a,y)$ to be neutral in $S$ (since $(x,y)$ is neutral in $S$);
\item if $x=x_k$ for some $k\in[i+1,j-1]$ and $y\geq b$, then the definition of $A$ implies that $(a,x)$ is neutral in $S$ and ($\heartsuit$) implies that $(x,b)$ is also neutral in $S$ (since $(x,y)$ is neutral in $S$), which, altogether, by transitivity, implies that $(a,b)$ is neutral in $S$, contradicting once again the assumption that $(a,b)\in \Invt_{\bf H}(T)\backslash \Invt_{\bf H}(S)$;
\item if $x=x_k$ and $y=x_l$ for some $k<l$ in the interval $[i+1,j-1]$, then, by the definitions of $A$ and $B$, we have that $(a,x)$ is neutral in $S$ and $(a,y)$ is not neutral in $S$, while, at the same time, the transitivity of neutrality implies that $(a,y)$ must be neutral in $S$ (since $(x,y)$ is neutral in $S$), contradiction. 
\end{itemize} 
\end{itemize}

 Finally, the property ($\spadesuit$) guarantees that no good pair of $S$ gets abolished in $U$, i.e., that $\Inv_{\bf H}(S)=\Inv_{\bf H}(U)$. 

\smallskip

We are left to prove the properties ($\heartsuit$), ($\clubsuit$) and ($\spadesuit$). 

\smallskip
\noindent
{\em  Proof of} ($\heartsuit$). Suppose that there exist indices $1\leq k\leq k'<l'\leq l\leq n$ such that $(x_k,x_l)$ is neutral in $S$ and $(x_{k'},x_{l'})$ is not neutral in $S$. Let $V$, $Y$ and $Z$ be the vertices of $S$ such that $x_k,x_l\in V$, $x_{k'}\in Y$ and $x_{l'}\in Z$, and let $W$ be the first common ancestor of $Y$ and $Z$ in $S$. Then the condition $\Inv_{\bf H}(S)= \Inv_{\bf H}(T)$ implies that the subtrees of $S$ rooted at $V$ and $W$ are disjoint, but this is in contradiction with Lemma \ref{s00}.

\medskip 
\noindent
{\em Proof of} ($\clubsuit$) {\em and} ($\spadesuit$). Having defined $$N:=\{(x,y)\in A_i\times A_j\,|\, i<j\}\enspace \mbox{ and } \enspace M:=\{(x,y)\in A^j\times K_i\,|\, 1<j<q, i\in I_j\},$$ our goal is to show that $\Invt_{\bf H}(S)\cap N=\Inv_{\bf H}(S)\cap M=\emptyset$. Suppose, for contradiction, that there exists $(x,y)\in (\Invt_{\bf H}(S)\cap N)\cup (\Inv_{\bf H}(S)\cap M)$. Let $Y$ be the node of $S$ such that $x\in Y$, and suppose that the subconstruct of $S$ rooted at $Y$ is given by $Y(T_1,\dots,T_r):{\bf G}$, where ${\bf G},Y\leadsto G_1<\cdots <G_r$ and $T_i:{\bf G}_i$. Since ${\bf G}$ is connected, there exists a chain of hyperedges $E_1,\dots,E_m\in{\bf G}$, such that $x\in E_1$, $y\in E_m$, and there exists a choice of $q_i\in E_i\cap E_{i+1}$, $1\leq i\leq m-1$. Our proof proceeds by induction on $m$.

\smallskip

Suppose that $m=1$. Since $x$ and $y$ are separated in $U$, i.e., belong to different connected components of ${\bf K}\backslash B$, $E_1$ must intersect at least one vertex on the path, in $U$, starting from $B$ and going to the root vertex of $U$. We proceed by analysing the cases relative to whether $(x,y)$ is a neutral or a good pair of $S$. We shall work out the details only for the case that it is neutral, with the indication that the proof strategy is the same in the case of a good pair. 
\begin{itemize}
\item Suppose that there exists an $e\in E_1\cap B$. Since no element of $B$ can appear strictly above $Y$ in $S$, as this would produce a good pair not present in $T$, it must be the case that $e\in Y$. We now proceed by analysing the position of $x,y$ and $e$ in $X$, relative to $a$ and $b$.
\begin{itemize}
\item if $e<b$, then, by definition of $B$, we have that $(a,e)$ is not neutral in $S$, and hence $a$ must appear either above $Y$, or below $Y$ or on a branch disjoint from $Y$; the first two possibilities are in contradiction with the fact that $\Inv_{\bf H}(S)= \Inv_{\bf H}(T)$, and the third one with the right-filled property in the form of Lemma \ref{s00};
\item if $e\geq b$ and $x<y\leq a$ or $x\leq a<y$, then, by ($\heartsuit$), since $(x,e)$ is neutral in $S$, $(a,b)$ would also have to be neutral in $S$, which contradicts the starting assumption;
\item if $e\geq b$ and $a\leq x<y$, then, by definition of $A$, we have that $a\in Y$, and hence it must be the case that $e>b$ with $b$ appearing either below $Y$ or on a branch disjoint from $Y$; the first possibility is in contradiction with the fact that $\Inv_{\bf H}(S)= \Inv_{\bf H}(T)$ and the second one with the right-filled property in the form of Lemma \ref{s00}; 
\end{itemize}
\item If $E_1\cap B=\emptyset$, let $V$ be any vertex of $U$ below $B$ such that $E_1\cap V\neq \emptyset$. Observe that $V$ is then also a vertex of $T$ below $X$, and moreover that the path from $V$ to the parent of $B$ in $U$ is the same as the path from $V$ to the parent of $X$ in $T$. Pick an element $e\in E_1\cap V$. Note that if $e<y$, then $(e,y)$ is a good pair in $T$ and hence in $S$, and hence $e$ must appear strictly below $Y$ is $S$, contradicting the fact that $e\in E_1\subseteq G$. Likewise, if $e>y$, then $(e,y)$ is not a good pair in $T$ and hence can neither be good in $S$. Therefore, since $e\in E_1\subseteq G$, the only possibility is that $e\in Y$, but this is contradicting the fact that $\Invt_{\bf H}(S)\subsetneq \Invt_{\bf H}(T)$. 
\end{itemize}
This finishes the proof for the case $m=1$.

\smallskip

Suppose now that $m>1$ and consider $q_{m-1}$. We differentiate two cases.
 
\smallskip
\indent Case (1). If $q_{m-1}<x$, then, by applying the right-filled property of ${\bf G}$ (ensured by Lemma \ref{s000}) to $E_{m}$, we conclude that there exists a hyperedge $E$ of ${\bf G}$ such that $x,y\in E$, which takes us back to our base case.

\smallskip

\indent Case (2). If $q_{m-1}>x$, then $(x,q_{m-1})$ is either a neutral or a good pair of $S$. Suppose that $(x,y)$ is neutral in $S$.
\begin{itemize}
\item If $q_{m-1}$ is strictly above $Y$, and
\begin{itemize}
\item if $x<y<q_{m-1}$, then $(x,q_{m-1}),(y,q_{m-1})\in \Inv_{\bf H}(S)$, and hence, since ($\spadesuit$) holds by induction, $(x,q_{m-1}),(y,q_{m-1})\in \Inv_{\bf H}(U)$, but this configuration is impossible as it produces a cycle in the tree $U$ (given that $x$ and $y$ are separated there);
\item if $x<q_{m-1}<y$, then $(x,q_{m-1})\in \Inv_{\bf H}(S)$, and hence, by induction, $(x,q_{m-1})\in \Inv_{\bf H}(U)$, which means that $q_{m-1}$ and $y$ are separated in $U$, which is impossible by our base case (applied to the pair $(q_{m-1},y)$ and the hyperedge $E_m$ containing them).
\end{itemize}

\item Suppose that $q_{m-1}\in Y$, and hence also $q_{m-1}\in X$. We first show that, in fact, $q_{m-1}\in A$. Suppose, for contradiction, that $q_{m-1}\in B$. Then, if $q_{m-1}\geq b$ and
\begin{itemize}
\item if $x<y<a$ or $x<a<y$, we get a contradiction with the fact that $(a,b)$ is not neutral in $S$, using ($\heartsuit$);
\item if $a<x<y$, by definition of $A$, we get that $a\in Y$, and hence that $(a,q_{m-1})$ is neutral in $S$, which, using ($\heartsuit$) once again, leads to a contradiction with the fact that $(a,b)$ is not neutral in $S$;
\end{itemize}
while, if $q_{m-1}< b$, by definition of $B$, we have that $(a,q_{m-1})$ is not neutral in $S$ and hence
\begin{itemize}
\item if $x<y<a$ or $x<a<y$, we  again  get a contradiction  using ($\heartsuit$);
\item if $a<x<y$, then, by definition of $A$, we get that $a\in Y$, i.e., that $(a,x)$ is neutral, but since $(x,q_{m-1})$ is also neutral, we get a contradiction by transitivity of neutrality. 
\end{itemize}

Now, since $q_{m-1},y\in Y$ and $q_{m-1},y\in E_m$, our base case implies that $q_{m-1}$ and $y$ cannot be separated in $U$, i.e., that $q_{m-1}\in A_j$. Therefore, $x$ and $q_{m-1}$ are separated in $U$, and this is impossible by induction.
\end{itemize}
This concludes the proof for the case $(x,y)\in \Invt_{\bf H}(S)$. We proceed analogously if $(x,y)\in \Inv_{\bf H}(S)$, in which case the induction argument uses ($\clubsuit$).
 \end{proof}

\begin{lemma}\label{s13}
Let ${\bf H}$ be an ordered right-filled hypergraph and let $S,T:{\bf H}$ be such that $\Inv_{\bf H}(S)\subsetneq \Inv_{\bf H}(T)$ and $\Invt_{\bf H}(S)\subseteq \Invt_{\bf H}(T)\cup \Inv_{\bf H}(T)$. Then there exists a construct $U:{\bf H}$ such that $U\lessdot^+ T$, $\Inv_{\bf H}(S)\subseteq \Inv_{\bf H}(U)\subsetneq \Inv_{\bf H}(T)$ and $\Invt_{\bf H}(S)\subseteq \Invt_{\bf H}(U)\cup \Inv_{\bf H}(U)$. 
\end{lemma}
\begin{proof}
Pick $(a,b)\in \Inv_{\bf H}(T)\backslash \Inv_{\bf H}(S)$ such that the distance $d$ in $T$ between the nodes $A$ and $B$ containing $a$ and $b$, respectively, is minimal. We first prove that, in this case, $d=1$, i.e., $B$ is a child of $A$. To get a contradiction, suppose that there exists a node $C$ between $A$ and $B$ in $T$ and pick an element $c\in C$. 
\begin{itemize}
\item If $c>b>a$, then, by minimality of $d$, $(a,c)$ must be a good pair in $S$; in addition, since $\Inv_{\bf H}(S)\subsetneq \Inv_{\bf H}(T)$ and $(a,b)$ is not good in $S$, both $(b,c)$ and $(b,a)$ must be pairwise incomparable in $S$. But this configuration is impossible by Lemma \ref{s00}. Likewise, if $c<a<b$, an analogous analysis leads to a contradiction with Lemma \ref{s00}.
\item If $a<c<b$, then, by minimality of $d$, both $(a,c)$ and $(c,b)$ must be good pairs in $S$, but this forces $(a,b)$ to be a good pair in $S$ as well, contradicting the assumption that $(a,b)\not\in \Inv_{\bf H}(S)$.
\end{itemize}

Having chosen $(a,b)\in \Inv_{\bf H}(T)\backslash \Inv_{\bf H}(S)$ such that $B$ is a child of $A$, suppose that $A=\{x_1<\cdots <x_n\}$ and $B=\{y_1<\dots<y_m\}$ with $a=x_i$ and $b=y_j$ for some $1\leq i\leq n$ and $1\leq j\leq m$. In order to define $U$, we shall use the following property.
\begin{itemize}
\item[($\diamondsuit$)] If $(a,b)$ is not a good pair in $S$, then neither are $(a,y_k)$ for $j+1\leq k\leq m$, nor $(x_l,b)$ for $1\leq l\leq i-1$.
\end{itemize}

\noindent
{\em  Proof of} ($\diamondsuit$). Suppose, for contradiction, that $(a,y_k)$ is good in $S$. Since $(a,b)$ is not good in $S$, $b$ can appear either below $a$ or on a branch disjoint from $a$ in $S$. The first possibility is in contradiction with the fact that $\Inv_{\bf H}(S)\subsetneq \Inv_{\bf H}(T)$, since $(b,y_k)$ would be a good pair in $S$ but not in $T$. The second possibility violates Lemma \ref{s00}. The proof that $(x_l,b)$ cannot be good in $S$ is analogous. This finishes the proof of ($\diamondsuit$).

\medskip

 The property ($\diamondsuit$) allows us to define decompositions 
 $A=A_l\cup A_r$ and $B=B_l\cup B_r$, with $A_r$ and $B_l$ possibly empty, such that the pairs $(x,y)\in A_l\times B_r$ are not good in $S$ and $B_r>A_l$. Among all such decompositions, we take the one for which $B_r$ and $A_l$ are maximal subsets of $B$ and $A$ with desired properties. Let $A^1_l<\cdots <A^q_l$ be the disintegration of $A_l$ in the last performed squashing in defining the construct $$U'=T[B\rightarrow B_r\langle B_l \rangle][A\rightarrow A_r\langle A_l \rangle],$$ where, if $A_r=\emptyset$ (resp. $B_l=\emptyset$), we set $A_r\langle A_l\rangle :=A$ and trivially declare $A$ to be the single block in the corresponding (trivial) disintegration (resp. $B_r\langle B_l\rangle := B$). If there exists $1\leq p\leq q$ such that $A_l^p$ is the parent of $B_r$ in $U'$ and $a\in A_l^p$, we define $U:=U'[A_l^p(B_r)\rightarrow A_l^p\cup B_r]$; otherwise, we set $U:=U'$. 
 
 The verification of the required properties of $U$ uses similar (lengthy) arguments to those used in the proof of Lemma \ref{s14} and are omitted.
 \end{proof} 

\begin{example}
To ease the understanding of the proof of Lemma \ref{s13}, we successively build $U$ (in terms of the corresponding sequence of covering relations) for constructs $T=\{1,2\}(\{3,4\})$ and $S=\{2,4\}(1,3)$ of the hypergraph ${\bf H}^{\text{\ding{170}}}$ from Figure \ref{hassediag5}, for $(a,b)=(1,4)$ and $(a,b)=(2,3)$. For $(a,b)=(1,4)$, we have that $A_l=\{1\}$, $A_r=\{2\}$, $B_l=\{3\}$, $B_r=\{4\}$, resulting in the following (inverse) covering sequence leading from $T$ to $U$: 
$$\{1,2\}(\{3,4\}) \gtrdot \{1,2\}(4(3)) \gtrdot 2(1,4(3)) \gtrdot \{2,4\}(1,3).$$ For $(a,b)=(2,3)$, we have that $A_r=B_l=\emptyset$, and we trivially have $$\{1,2\}(\{3,4\}) \gtrdot \{1,2,3,4\}.$$
\end{example}

We have now all the ingredients for the characterisation we sought for\footnote{In \cite{KLNPS}, Krob, Latapy, Novelli, Phan and Schwer state an equivalent characterisation of the GFO in the case of permutohedra.}.
 \begin{theorem}\label{gfo-inv}
If ${\bf H}$ is a right-filled hypergraph, then, for any two constructs $S,T:{\bf H}$, it holds that $S\leq T$ if and only if $\Inv_{\bf H}(S)\subseteq \Inv_{\bf H}(T)$ and $\Invt_{\bf H}(S)\subseteq \Invt_{\bf H}(T)\cup \Inv_{\bf H}(T)$.
\end{theorem}
\begin{proof}
In one direction, we observe that 
\begin{itemize}
\item a fusion $S\lessdot S[U(V)\rightarrow U\cup V]$ does not affect the good or neutral pairs of $S$, it only adds new neutral and good pairs, while
\item a split $S\lessdot S[X\cup Y\rightarrow X(Y)]$ turns certain neutral pairs of $S$ into good pairs. Here, the property of being right-filled ensures that all the good pairs of $S$ remain good pairs $S[X\cup Y\rightarrow X(Y)]$. Indeed, if $(a,b)\in \Inv_{\bf H}(S)$ and if $A$ and $B$ are the nodes of $S$ such that $a\in A$ and $b\in B$, it could happen that $A=X\cup Y$ is an ancestor of $B$ in $S$ with $a\in Y$. In this case, if it turns out that, after the split, the connected component of $Y$ is different from the one of $B$, the pair $(a,b)$ would no longer be good. But thanks to Lemma \ref{s11}, since $X<Y$, this configuration is impossible.
 \end{itemize}

We prove the other direction by induction on $n=|\Inv_{\bf H}(T)\backslash \Inv_{\bf H}(S)|$. The base case is given by $n=0$, i.e., $\Inv_{\bf H}(S)= \Inv_{\bf H}(T)$. In this case, the inclusion $\Invt_{\bf H}(S)\subseteq \Invt_{\bf H}(T)\cup \Inv_{\bf H}(T)$ boils down to the inclusion $\Invt_{\bf H}(S)\subseteq \Invt_{\bf H}(T)$. We proceed by induction on $m=|\Invt_{\bf H}(T)\backslash \Invt_{\bf H}(S)|$. If $m=0$, we conclude by Lemma \ref{eee} that $S=T$ and we are done by the reflexivity of $\leq$. If $m>0$, by Lemma \ref{s14} there exists a construct $U:{\bf H}$ such that $U\lessdot^+ T$, $\Inv_{\bf H}(S)= \Inv_{\bf H}(U)$ and $\Invt_{\bf H}(S)\subseteq \Invt_{\bf H}(U)\subsetneq \Invt_{\bf H}(T)$. The inductive hypothesis implies that $S\leq U$ and hence $S\leq U\lessdot^+T$. Suppose now that $n>0$, i.e., that $\Inv_{\bf H}(S)\subsetneq \Inv_{\bf H}(T)$. By Lemma \ref{s13}, there exists a construct $U:{\bf H}$ such that $U\lessdot^+ T$, $\Inv_{\bf H}(S)\subseteq \Inv_{\bf H}(U)\subsetneq \Inv_{\bf H}(T)$ and $\Invt_{\bf H}(S)\subseteq \Invt_{\bf H}(U)\cup \Inv_{\bf H}(U)$, allowing us to conclude once again that $S\leq U$ and hence $S\leq U\lessdot^+ T$. 
\end{proof}

\begin{remark} We note that Theorem \ref{gfo-inv} provides an alternative proof of the absence of cycles in the GFO, i.e., of Proposition \ref{lessdot-order} (in the right-filled case). Indeed, if $S\leq T$ and $T\leq S$, then we get immediately $\Inv_{\bf H}(S)= \Inv_{\bf H}(T)$, and the second inequalities reduce then to $\Invt_{\bf H}(S)\inc \Invt_{\bf H}(T)$ and conversely, and we conclude by Lemma \ref{eee}.
\end{remark}

\subsection{Generalised flip order and restriction}

We now recall  from \cite[\S 4.4]{PLBJ1} the notion of restriction of a construct and   establish a commutation between the GFO and the restriction that we shall use in \S \ref{shuffle-order-section}. Again, we make use of the hereditarily ordered setting (cf. Definition \ref{hereditarily-ordered-def}).

 \begin{definition}[Restriction of a construct] \label{restriction-definition}
 For  hypergraphs ${\bf H}$ and ${\bf K}$ such that $K\subseteq H$ and any hyperedge of ${\bf K}$ is connected in ${\bf H}$, 
 and a construct $C:{\bf H}$, {\em the restriction of $C$ to $K$} is the construct $C_{\lceil_{\bf K}}:{\bf K}$ defined as follows:
\begin{itemize}
\item if $C=H$, then $C_{\lceil {\bf K}}=K$, 
\item if $C=X(C_1,\dots,C_n)$, where $\emptyset \neq X \subsetneq H$, ${\bf H},X\leadsto {H}_1,\dots { H}_n$ and $C_{i}:{\bf H}_{i}$ for $1\leq i\leq n$, and if
\begin{itemize}
\item $X\inter K=\emptyset$, then there exists $i\in\{1,\dots,n\}$ such that $K\inc H_i$, and we set $C_{\lceil_{\bf K}}={(C_{i})}_{\lceil_{\bf K}}$;
\item otherwise, supposing that ${\bf K},(X\inter K)\leadsto K_1,\dots,K_p$, we have that for each $j\in\{1,\dots,p\}$ there exists $i\in\{1,\dots,n\}$, such that $K_j$ is connected in $H_i$; we denote by $\psi:\{1,\dots,p\}\rightarrow \{1,\dots,n\}$ the induced index correspondence, and we set $$C_{\lceil_{\bf K}}=(X\inter K)((C_{\psi(1)})_{\lceil_{{\bf K}_1}},\dots,(C_{\psi(p)})_{\lceil_{{\bf K}_p}}).$$
\end{itemize}
\end{itemize}
\end{definition}

\begin{remark}\label{rest1}
For ${\bf L}$, ${\bf K}$ and ${\bf H}$ such that $\emptyset \neq L\subseteq K\subseteq H$, every hyperedge of ${\bf L}$ is connected in $\hyper{K}$ and every hyperedge of ${\bf K}$ is connected in $\hyper{H}$, and a construct $C:{\bf H}$, it holds that $(C_{\lceil_{\bf K}})_{\lceil_{\bf L}}=C_{\lceil_{\bf L}}$.
\end{remark}
\begin{remark} 
The notion of restriction generalises that of a subconstruct of a construct. More precisely, for every subconstruct $C'$ of $C$ with $\textrm{root}(C')=X$, we have $C'=\restrconstr{C}{{\valup_C(X)}}$ (cf. Remark \ref{constructs-tubings}). 
\end{remark}

\begin{proposition} \label{restriction-order-comm}
Let ${\bf H}$ be a connected hypergraph, let $K\subseteq H$ be connected in ${\bf H}$, and let $S,T:{\bf H}$. If $S\lessdot T$, then $\restrconstr{S}{{\bf K}} \leq \restrconstr{T}{{\bf K}}$.
\end{proposition}
\begin{proof}
We prove the claim by induction on $|H|$. The base case, given by $|H|=1$, is trivial. Suppose that $\hyper{H},X\leadsto H_1,\ldots,H_n$ and that $S=X(S_1,\ldots,S_n)$, where $S_i:{\bf H}_i$. 
We distinguish cases depending on the nature of the covering $S\lessdot T$ (see Remark \ref{lem:StabilityByContext}(a)).

 \begin{itemize}
\item Case 1: $S\lessdot T$ is a non-root covering, i.e., there exists an index $1\leq j\leq n$ such that $S_j\lessdot S'_j$ and $T=X(S_1,\ldots,S_{j-1},S'_j,S_{j+1},\ldots,S_n)$. We then further distinguish cases depending on the distribution of $K$ in $S$ (and $T$).
\begin{itemize}
\item If $X\cap K=\emptyset$ and $K\inc H_k$ with $k\neq j$, then, by Definition \ref{restriction-definition}, we have that $\restrconstr{S}{{\bf K}}=\restrconstr{S_k}{{\bf K}}=
\restrconstr{T}{{\bf K}}$.
\item If $X\cap K=\emptyset$ and $K\inc H_j$, we conclude by induction that $\restrconstr{S}{{\bf K}}=\restrconstr{S_j}{{\bf K}}\leq \restrconstr{S'_j}{{\bf K}}=
\restrconstr{T}{{\bf K}}$.
\item If $X\cap K\neq\emptyset$, then, by Definition \ref{restriction-definition}, we have that
$$\restrconstr{S}{{\bf K}} =(X\cap K)(\ldots,\restrconstr{(S_{\varphi_X^{\hyper{K},\hyper{H}}(i)})}{\hyper{K}_i},\ldots),$$
where $\hyper{K},X \leadsto K_1,\ldots, K_p$ and $\varphi_X^{\hyper{K},\hyper{H}}:\set{1,\ldots,p}\rightarrow \set{1,\ldots,n}$. Likewise, $\restrconstr{S'}{{\bf K}}$ is the same expression where 
$\restrconstr{(S_{\varphi_X^{\hyper{K},\hyper{H}}(i)})}{\hyper{K}_i}$ is replaced by $\restrconstr{S'_j}{{\bf K}_i}$, 
for all $i$ such that $\varphi_X^{\hyper{K},\hyper{H}}(i)=j$. Therefore, using the fact that for all those $i$ we have, by induction, that
$\restrconstr{(S_j)}{\hyper{K}_i}\lessdot \restrconstr{(S'_j)}{\hyper{K}_i}$, we conclude that
$\restrconstr{S}{{\bf K}} \leq \restrconstr{T}{{\bf K}}$.
\end{itemize}
\item Case 2: $S\lessdot T$ is a root split, i.e., there exists a decomposition $X=X_1\cup X_2$ such that $X_1<X_2$ and
$T=X_1(S_1,\ldots ,S_j,X_2(S_{j+1},\ldots ,S_l),S_{l+1},\ldots,S_n)$. We proceed again by analysing the distribution of $K$ in $S$ (and $T$).
\begin{itemize}
\item If $X\cap K=\emptyset$ and $K\inc H_k$ for any $k$, then, by Definition \ref{restriction-definition}, we have that $\restrconstr{S}{{\bf K}}=\restrconstr{S_k}{{\bf K}}=
\restrconstr{T}{{\bf K}}$.
\item If $X_1\cap K=\emptyset$ and $X_2\cap K\neq\emptyset$, then $K$ is disjoint from all $H_i$ other than $H_{j+1},\ldots,H_l$. Setting $U=X_2(S_{j+1},\ldots, S_l)$, we easily conclude that $\restrconstr{S}{{\bf K}}=\restrconstr{U}{{\bf K}}$, and hence that $\restrconstr{S}{{\bf K}}= \restrconstr{T}{{\bf K}}$.
\item If $X_1\cap K\neq \emptyset$ and $X_2 \cap K\neq \emptyset$, then, with $U$ as above and $\hyper{K},X_1\leadsto K^1_1,\ldots K^1_t$,
$\restrconstr{T}{{\bf K}}$ has the following shape:
$$\begin{array}{lll}
\restrconstr{T}{{\bf K}}& = & (X_1\cap K)(\ldots,\restrconstr{U}{{\bf K}^1_1},\ldots,\restrconstr{U}{{\bf K}^1_t},\ldots),\ldots)\\
& =& (X_1\cap K)(\ldots,(X_2\cap K^1_1)(\ldots),\ldots, (X_2\cap K^1_t)(\ldots),\ldots)
\end{array}$$
where $X\cap K= (X_1\cap K)\cup (X_2\cap K^1_1)\cup\cdots\cup(X_2\cap K^1_t)$ and $\max(K^1_i) < \min(K^1_{i+1})$.
Therefore, starting from $\restrconstr{T}{{\bf K}}$, we can perform a cascade of fusions, from left to right, resulting in $\restrconstr{S}{{\bf K}}$, i.e.,
\begin{align}
\restrconstr{S}{{\bf K}} & \leq ((X_1\cap K)\cup(X_2\cap K^1_1))(\ldots,(X_2\cap K^1_{i})(\ldots),\ldots) \\
& \lessdot \restrconstr{T}{{\bf K}}.
\end{align}
In other words, using the squashing notation of Definition \ref{squashing-notation}, we have $ \restrconstr{T}{{\bf K}}= (\restrconstr{S}{{\bf K}})[[(X_1\cap K)\langle (X_2\cap K)\rangle/(X\cap K)]$.
\item If $X_1\cap K\neq \emptyset$ and $X_2 \cap K=\emptyset$, then $X\cap K=X_1\cap K$. Moreover, referring to the notation used in the previous case, each of $\restrconstr{U}{{\bf K}^1_1},\ldots \restrconstr{U}{{\bf K}^1_t}$ has the form $\restrconstr{(S_{m_1})}{{\bf K}^1_1},\ldots \restrconstr{(S_{m_t})}{{\bf K}^1_t}$, respectively (with $j+1\leq m_1\leq\cdots\leq m_t\leq l$). One can conclude easily from these observations that $\restrconstr{S}{{\bf K}}\leq\restrconstr{T}{{\bf K}}$.
\end{itemize}
\item Case 3: $S\lessdot T$ is a root fusion, i.e., 
$$S=X(S_1,\ldots, S_{j-1},Y(S_{j_1},\ldots, S_{j_l}),S_{j+1},\ldots,S_n), \mbox{ with } X>Y,$$ and
$$T=(X\cup Y)(S_1,\ldots, S_{j-1},S_{j_1},\ldots, S_{j_l},S_{j+1},\ldots,S_n).$$ This case is treated much like the previous one and is left to the reader (most notably, in the relevant subcase, the cascade of fusions starts now from $\restrconstr{S}{{\bf K}}$ and goes from right to left).\vspace{-0.3cm}
\end{itemize}
\end{proof}

\section{Shuffle product on the faces of nestohedra} \label{strict-reminders-section}
 The purpose of this section is to show that, in the ordered framework, the shuffle product
on constructs introduced in \cite{PLBJ1}   admits a description in terms of sums over suitable intervals of the generalised
flip order. As indicated in the introduction, this extends the interval descriptions of products for associahedra and permutohedra, given by
Palacios and Ronco in \cite{PalaciosRonco}, to the setting
of arbitrary (ordered strict associative clans of) nestohedra.

 In \S \ref{comb-subsection}   we collect  definitions and results from \cite{PLBJ1}. 
  The connection with the generalised flip order is established in \S \ref{shuffle-order-section}.

\subsection{Combinatorial framework and   shuffle product}\label{comb-subsection}

The characterisation of the shuffle product in terms of the GFO requires an ordered combinatorial framework that we now establish, by recalling (the ordered variants of) the notions of team,
strict team and strict clan from \cite{PLBJ1}. 

\smallskip

 We first fix a  collection $\mathfrak{U}$ of connected hypergraphs. 
\begin{definition}\label{ost}

A {\em team} (in $\mathfrak{U}$) is a pair $\tau=(\{\hyper{H}_a\,|\, a\in A\},\hyper{H})$ such that ${\bf H}_a\in \mathfrak{U}$ for each $a\in A$ and ${\bf H}\in {\mathfrak{U}}$, and such that the ${\bf H}_a$'s form a set partition of $H$.
A team as above is called {\em strict} if, in addition, it satisfies the following property:
\begin{itemize}
\item[(S)] for every $a\in A$ and every connected subset $\emptyset\neq K\inc H_a$, we have that $K$ is also connected in $\hyper{H}$.\footnote{It is enough to verify this condition for $K$ ranging over the hyperedges of $\hyper{H}_a$.}
\end{itemize} The hypergraphs ${\bf H}_a$, $a\in A$, are called the {\em participating hypergraphs} of $\tau$, and ${\bf H}$ is called the coordinating hypergraph of $\tau$.
\end{definition}

 In the ordered framework, the collection $\mathfrak{U}$ is required to be {\em $\mathbb{Z}$-ordered} in the sense that, for all $\hyper{H}\in\mathfrak{U}$, $H\inc\mathbb{Z}$ and $\hyper{H}$ (endowed with the order induced by 
$\mathbb{Z}$) is hereditarily ordered. The ordered variant of the notion of strict team is then obtained as follows. 

\begin{definition}\label{orderedteam}
 An ordered strict team
 is a pair $\tau=(({\bf H}_1,\dots,{\bf H}_p),\hyper{H})$ such that $(\{{\bf H}_1,\dots,{\bf H}_p\},\hyper{H})$ is a strict team and $H_1<\cdots<H_p$. 
\end{definition}

\begin{example}
The team
\begin{equation*}
\left( \left(\raisebox{-1.5pt}{\resizebox{2cm}{!}{\begin{tikzpicture}[scale=0.5,baseline=-5pt]
\node (1) at (0,0) {$1$};
\node[right=0.5cm of 1] (2) {$2$};
\node[right=0.5cm of 2] (3) {$3$};
\begin{scope}{background}
\node[draw,fit=(1)(2)(3), rounded corners=10pt] {};
\end{scope}
\end{tikzpicture}}}, \raisebox{-3.2pt}{\begin{tikzpicture}\node at (0,0){\footnotesize{4}}; 
\end{tikzpicture}}\right) ,
\raisebox{-1pt}{\resizebox{2.55cm}{!}{\begin{tikzpicture}[scale=0.5,baseline=-5pt]
\node (1) at (0,0) {$1$};
\node[right=0.5cm of 1] (2) {$2$};
\node[right=0.5cm of 2] (3) {$3$};
\node[right=0.5cm of 3] (4) {$4$};
\draw (1)--(2)--(3)--(4);
\end{tikzpicture}}}
 \right)
\end{equation*}
is  ordered and strict. The team
\begin{equation*}
\left( \left( \raisebox{-3.2pt}{\begin{tikzpicture}\node at (0,0){\footnotesize{1}}; 
\end{tikzpicture}}, \raisebox{-1.5pt}{\resizebox{2cm}{!}{\begin{tikzpicture}[scale=0.5,baseline=-5pt]
\node (1) at (0,0) {$2$};
\node[right=0.5cm of 1] (2) {$3$};
\node[right=0.5cm of 2] (3) {$4$};
\draw (1)--(2);
\begin{scope}{background}
\node[draw,fit=(1)(2)(3), rounded corners=10pt] {};
\end{scope}
\end{tikzpicture}}} \right) ,
\raisebox{-1.5pt}{\resizebox{3cm}{!}{\begin{tikzpicture}[scale=0.5,baseline=-5pt]
\node (1) at (0,0) {$1$};
\node[right=0.5cm of 1] (2) {$2$};
\node[right=0.5cm of 2] (3) {$3$};
\node[right=0.5cm of 3] (4) {$4$};
\draw (1)--(2);
\begin{scope}{background}
\node[draw,fit=(1)(2)(3), rounded corners=10pt] {};
\node[draw,fit=(1)(2)(3)(4), rounded corners=10pt, inner sep=6pt] {};
\end{scope}
\end{tikzpicture}}}
 \right)
\end{equation*}
is not strict  as it fails to satisfy the property (S) of Definition \ref{ost}.
\end{example}

\begin{definition}
	A \emph{delegation} is a pair 
	$\delta=(\{C_a:{\bf H}_a\,|\, a\in A\},{\bf H})$
		such that its {\em support}  $\tau:=(\setc{\hyper{H}_a}{a\in A},\hyper{H})$ is a strict team.  

In the ordered setting, a delegation will be given as an ordered tuple of constructs $\delta=((C_1:{\bf H}_1,\dots,C_p:{\bf H}_p),{\bf H})$. As in Definition \ref{orderedteam}, this notation will imply that $H_1<\cdots < H_p$, allowing us to refer unambiguously to the leftmost and rightmost participating construct of $\delta$. 
 \end{definition}
 
 Our goal is to define the shuffle product of a delegation $(\{C_a:{\bf H}_a\,|\, a\in A\},{\bf H})$. It will be a (weighted) sum of constructs $C$ obtained each by picking a subset $B$ of $A$ and labeling the  root of $C$ by the union of  all $\text{root}(C_b)$, for $b$ ranging over $B$, and then by proceeding recursively.
 
In order to be able to formalise this, we shall need the following apparatus and closure condition on our strict teams. Given a strict team $\tau=(\{\hyper{H}_a\,|\, a\in A\},\hyper{H})$, 
pick a subset $\emptyset\neq B\subseteq A$ and for each $b\in B$ a subset $\emptyset\neq X_b\subseteq H_b$, and
consider the decompositions
\begin{equation}\begin{array}{l}\hyper{H}_b, X_b \leadsto H_{(b,1)},\ldots ,H_{(b,n_b)}\\
\hyper{H},X_B\leadsto H_1^B,\ldots,H_{n_B}^B, \textrm{where}\; X_B=\Union_{b\in B}X_b.
\end{array}\label{decomp}
\end{equation}
We set
 $\tilde{A}:=(A\backslash B)\union\setc{(b,i)}{b\in B, 1\leq i\leq n_b}$.
 The property (S) of Definition \ref{ost} implies that for each ${\tilde a\in\tilde A}$, $H_{\tilde a}$ is contained in $H_i^B$ for some $1\leq i\leq n_B$.
 This leads naturally to the following induced teams:
$$\tau_i^B=(\setc{\hyper{H}_{\tilde a}}{\tilde a\in\tilde A\;\textrm{and}\; H_{\tilde a}\inc H_i^B},\hyper{H}_i^B)\quad\quad(1\leq i\leq n_B).$$
We write $\tau, X_B \,\clandec\, \tau^B_1,\ldots,\tau^B_{n_B}$.
When $\tau$ is ordered, then each $\tau^B_i$ is ordered as well.

\begin{definition}
A strict clan (in $\mathfrak{U}$) is a set $\Xi$ of strict teams (in $\mathfrak{U}$) which is closed under decomposition, i.e.,for each $\tau\in\Xi$ and each decomposition $\tau,X_B \,\clandec\, \tau^B_1,\ldots,\tau^B_{n_B}$ as in (\ref{decomp}), we have that $\tau_i\in\Xi$ for all $1\leq i\leq n_B$.  A $\Xi$-\emph{delegation} is a   delegation whose support is in $\Xi$.

\smallskip
When the teams of the clan are all ordered, the strict clan is an ordered strict clan.
\end{definition}

 Observe that, for $\emptyset \neq B\subseteq A$ and $\tilde{A}$ as above, assuming that $X_a=\mbox{root}(C_a)$ for each $a\in A$, there is a canonical association of a construct $C_{\tilde{a}}$ to each $\tilde{a}\in \tilde{A}$, which gives rise to delegations \begin{equation}\label{deltas}\delta^B_i=(\{C_{\tilde{a}}:{\bf H}_{\tilde{a}}\,|\, \tilde{a}\in \tilde{A} \mbox{ and }H_{\tilde{a}}\subseteq H_i^B\},{\bf H}_i^B),\end{equation} for $1\leq i\leq n_B$. More precisely, for $b\in B$, we set $C_b=X_b(C_{(b,1)},\ldots,C_{(b,n_b)})$ with $C_{(b,i)}:\hyper{H}_{(b,i)}$.  

Let $\Xi$ be a strict  clan and let $q\in{\mathbb R}$.

\begin{definition}\label{shp}
 The {\it shuffle product} (or product) of a $\Xi$-\emph{delegation} $\delta=(\{C_a:{\bf H}_a\,|\, a\in A\},{\bf H})$ is defined as follows:
\begin{equation} \label{shuffle-antistrict-def} 
\ast(\delta)= \sum_{\emptyset\incs B\subseteq A} q^{{|B|}-1}\ast_B(\delta), \mbox{ where } \ast_B(\delta)=(\Union_{b\in B}X_b)(\ast(\delta^B_1),\dots,\ast(\delta^B_{n_{B}})), 
\end{equation}
and where the delegations $\delta_i^B$, for $1\leq i\leq n_B$, are induced as in \eqref{deltas}.
\end{definition}
 Note that, in Definition \ref{shp}, $\ast(\delta)$ is a linear combination of constructs, living in the vector space spanned by the constructs of $\hyper{H}$. 

\smallskip
In order to be able to prove associativity of shuffle product, we need our clans to satisfy an additional condition.
\begin{definition}
A strict clan $\Xi$ is \emph{associative} if for all $\tau = (\setc{\hyper{H}_a}{a\in A},\hyper{H})\in \Xi$, $a_0\in A$ and
$\tau' = (\setc{\hyper{H}_{(a_0,a')}}{a'\in A'},\hyper{H}_{a_0}) \in\Xi$, we have that
$$\tau'':=(\{{\bf H}_{a}\,|\,a\in A\backslash \{a_0\}\}\cup \{\hyper{H}_{(a_0,a')}\,|\,a'\in A'\},{\bf H})\in\Xi.$$ We refer to $\tau''$ as the grafting of $\tau'$ to $\tau$ along $a_0$.
\end{definition}
 The associativity of shuffle product of $\Xi$-delegations, where the clan $\Xi$ is associative,  is proven in \cite[Theorem 4.17]{PLBJ1} (see also Proposition \ref{associativity-diag} below). 

\begin{example}
Teleassociahedra (and, in particular, permutohedra and associahedra) and quasi-associahedra form ordered strict associative clans. The details for permutohedra and associahedra can be easily worked out (or found in \cite[\S 4]{PLBJ1}). The details for teleassociahedra in general are given in \cite[Example 4.3]{PLBJ2}. The strict clan associated with quasi-associahedra is obtained as follows. As our collection $\mathfrak{U}$, we fix the class of all quasi-associahedra. The corresponding teams have the form $$(({\bf Q}^{X_1},\dots,{\bf Q}^{X_p}),{\bf Q}^{\bigcup X_i}),$$ where the $X_1<\cdots <X_p$. In order to see that such a team is  strict, we note that, by definition, each hyperedge of each participating hypergraph ${\bf Q}^{X_i}$ is present in ${\bf Q}^{{\bigcup X_i}}$. We define the clan of quasi-associahedra as the set of all teams
defined as above. Clearly, this set of teams is closed under decomposition and is associative.
\end{example}
\begin{remark}
Simplices and hypercubes do not form strict clans. In fact, teams formed by members of these respective families (as well as other families of polytopes, such as cyclohedra) satisfy a connectedness condition different from (S), which we address in \cite{PLBJ2}.
\end{remark}

\begin{remark}\label{rembin}
In the ordered setting, the definition of the shuffle product can be spelled out more concretely. We illustrate this here for a binary team $((\hyper{H}_1,\hyper{H}_2), \hyper{H})$.
 
\smallskip

Consider a delegation $\delta=((S:\hyper{H}_1, T:\hyper{H}_2), \hyper{H})$, where the constructs $S$ and $T$ are given by $S=X(S_1, \ldots ,S_k)$ and $T=Y(T_1, \ldots, T_l)$.
The decompositions associated with the product $\ast(S,T)$, denoted more conventionally by $S \ast T$, are given by:
\begin{align}\label{NotIJ}
 &\hyper{H}_1, X \leadsto H_{(1,1)}, \ldots, H_{(1,k)}, & &\hyper{H}, X\leadsto H'_{1}, \ldots, H'_{n-1}, H'_n, \\
 &\hyper{H}_2, Y \leadsto H_{(2,1)}, \ldots, H_{(2,l)},& 
 &\hyper{H}, Y \leadsto H''_{n}, H''_{n+1} \ldots, H''_{n+m}.
 \end{align}
Because the underlying teams are strict and $H_1<H_2$, it follows that ${H}_2 \subseteq H'_n$. Consequently, there exists an interval partition (i.e., a partition whose blocks are contiguous intervals) $I_1, \ldots, I_n$ of $\llbracket 1, k\rrbracket$, where $I_n$ may be empty, such that:
\begin{itemize}
\item $I_j<I_{j+1}$, for all $1 \leq j \leq n-1$ (when non-empty),
\item $H'_j=\bigcup_{i \in I_j} H_{(1,i)}$ for $1 \leq j \leq n-1$, and
\item $H'_n=\left(\bigcup_{i \in I_n} H_{(1,i)}\right) \cup H_2$. 
\end{itemize}
By symmetric reasoning, we infer ${H}_1 \subseteq {H}''_n$, implying the existence of an interval partition $J_1, \ldots, J_m$ of $\llbracket 1, l\rrbracket$, where $J_1$ may be empty, such that 
\begin{itemize}
\item $J_j<J_{j+1}$, for all $1 \leq j \leq m-1$ (when non-empty),
\item $H''_{n+j}=\bigcup_{i \in J_j} H_{(2,i)}$ for $2 \leq j \leq m$, and 
\item $H''_n=H_1\cup \left(\bigcup_{i \in J_1} H_{(2,i)}\right)$. 
\end{itemize}
Since $H_2 \subseteq H'_n$, removing $X\cup Y$ from $\hyper{H}$ is the same as first removing $X$ from $\hyper{H}$ and then removing $Y$ from $\hyper{H}'_n$. Thus, supposing that the latter decomposition is given by $\hyper{H}'_n, Y \leadsto H'_{(n,1)}, \ldots, H'_{(n,p)}$, we get the decomposition 
$$\hyper{H}, X\cup Y \leadsto H'_{1}, \ldots, H'_{n-1}, H'_{(n,1)}, \ldots, H'_{(n,p)}.$$
Then the definition given above for $\ast(\delta)$ spells out as
\begin{equation} \label{shuffle-antistrict-def} 
 S \ast T = S \ast_{\set{1}} T+S \ast_{\set{2}} T+q \times S \ast_{\set{1,2}} T 
\end{equation}
where
$$\begin{array}{rcl}
 S \ast_{\set{1}} T &= & X(\ast_{i \in I_1} S_{i}, \ldots, \left( \ast_{i \in I_n} S_{i} \right) \ast T ) \\
S \ast_{\set{2}} T &= & Y(S \ast \left(\ast_{i \in J_1} T_{i} \right), \ldots, \ast_{i \in J_m} T_{i} )\\
S \ast_{\set{1,2}} T &=& X \cup Y \left(\ast_{i \in I_1} S_{i}, \ldots, \left( \ast_{i \in I_n} S_{i}\right) \ast \left( \ast_{i \in J_1} T_{i}\right), \ldots, \ast_{i \in J_m} T_{i} \right).
\end{array}$$ 
Here, $q \in \mathbb{R}$ is some fixed parameter, and for any interval $I=\llbracket a,b\rrbracket$, $\ast_{i\in I} S_i$ stands for $\ast(S_a,S_{a+1},\ldots,S_b)$.
\end{remark}

\medskip

We finish this section by recalling an alternative characterisation of the operations $\ast$ and $\ast_B$. We associate with $U$ a ``measure'' $\mu^\tau(U)$ as follows (with the notation as above). We set $B=\setc{b\in A}{X\cap H_b\neq\emptyset}$ and $X_b=X\cap H_b$ for each $b\in B$ (so that $n=n_B$), and we set
$$\mu^\tau(U)=(|B|-1)+ \sum_{1\leq i\leq n_B}\mu^{\tau_i}(U_i).$$
\begin{proposition}[\cite{PLBJ1}, Proposition 4.24] \label{restriction-product-characterisation} Let $\Xi$ be a strict clan, and
let $\delta=(\{C_a:{\bf H}_a\,|\, a\in A\},{\bf H})$ be a $\Xi$-delegation of support $\tau$. Then we have:
$$\begin{array}{lll} 
\ast(\delta)& = & \displaystyle\sum_{\substack {U:\hyper{H} \\[0.1cm] \forall a\in A, \restrconstr{U}{\hyper{H}_a}=C_a}} q^{\mu^\tau(U)}\, U,
\end{array}$$
and, for each $\emptyset\neq B\inc A$, writing $X_b=\textrm{root}(C_b)$ for $b\in B$ and $X_B=\cup_{b\in B}X_b$, we have 
$$\begin{array}{lll} 
\ast_B(\delta)& = &\displaystyle \sum_{\substack{U:\hyper{H},\; \textrm{root}(U)= X_B \\[0.1cm] \forall a\in A, \restrconstr{U}{\hyper{H}_a}=C_a}} \, q^{\mu^\tau(U)-|B|+1} U.
\end{array}$$
\end{proposition}

\subsection{Shuffle product via generalised flip order} \label{shuffle-order-section}

 We start by introducing extremal constructs
associated with an ordered delegation $\delta$, namely the constructs
$\backslash(\delta)$ and $\slash(\delta)$, together with their
variants indexed by a subset $B$.
These constructs will be the endpoints of the intervals
appearing in the interval characterisation of the shuffle product below.

In what follows, $\Xi$ will denote an ordered strict associative clan.
 
\begin{definition} 
Given a $\Xi$-delegation $\delta=((C_1:{\bf H}_1,\dots, C_p:{\bf H}_p),{\bf H})$, we define the constructs $\diagdown (\delta),\diagup (\delta):{\bf H}$ recursively as follows:
$$
 \diagdown (\delta) = \operatorname{root}(C_1) \left(\diagdown(\delta^{\{1\}}_1), \ldots, \diagdown (\delta^{\{1\}}_{n_{\set{1}}}) \right),
$$
$$
 \diagup (\delta) = \operatorname{root}(C_p) \left(\diagup(\delta^{\{p\}}_1), \ldots, \diagup (\delta^{\{p\}}_{n_{\set{p}}}) \right),
$$
where the recursion uses $\Xi$-delegations induced as in \eqref{deltas}, by taking $B=\{1\}$ and $B=\{p\}$ respectively.
\end{definition}
Informally, $\diagdown (\delta)$ is obtained by systematically placing
the leftmost participating construct as low as possible in the tree,
whereas $\diagup(\delta)$ is obtained by placing the rightmost
participating construct as low as possible.

\begin{remark} In the binary case, i.e., for $\delta=((S:\hyper{H}_1, T:\hyper{H}_2), \hyper{H})$, $S=X(S_1, \ldots , S_p)$, $T=Y(T_1,\ldots, T_q)$, we can spell out the operations $\diagdown$ and $\diagup$ as follows, using the notation of Remark \ref{rembin}:
$$
 S \diagdown T = X\left(\diagdown_{i \in I_1} S_i, \ldots,\left( \diagdown_{i \in I_n} S_i \right) \diagdown T\right)
$$
$$
 S \diagup T = T\left( S \diagup \left(\diagup_{i \in J_1} T_{i} \right),\ldots, \diagup_{i \in J_m} T_i \right).
$$ where, for an interval $I=\llbracket a,b\rrbracket$, $(\diagdown_{i\in I} S_i)\diagdown T$ stands for $\ast(S_a,S_{a+1},\ldots,S_b,T)$ (this notation being justified by the next proposition).
\end{remark}

The following proposition states the associativity of $\diagdown$ and $\diagup$. The proof is omitted as 
it follows (in a simpler setting) the strategy of the proof of associativity of $\ast$ given in \cite{PLBJ1}, and as this result is not used in the sequel.

\begin{prop} \label{associativity-diag}
 Suppose that $\tau = ((\hyper{H}_1,\dots,\hyper{H}_p),\hyper{H})\in \Xi$ and that, for some $1\leq i\leq p$, $\tau'= (({\bf{H}}_{(i,1)},\dots, {\bf{H}}_{(i,k)}),\hyper{H}_{i})\in\Xi$. Suppose that 
 we are given $\Xi$-delegations $$\delta'=((C_{(i,1)},\dots,C_{(i,k)}),{\bf H}_i)$$ of $\tau'$ and $$\delta''=((C_1,\dots,C_{i-1},C_{(i,1)},\dots,C_{(i,k)},C_{i+1},\dots,C_p),{\bf H})$$ of the grafting $\tau''$ of $\tau'$ to $\tau$ along $i$.
Then the following equations hold:
\begin{itemize}
 \item Associativity of $\diagdown$:\enspace $
 \diagdown(\delta'') = 
\diagdown( C_1,\dots,C_{i-1},\diagdown(\delta'),C_{i+1},\dots,C_p ),$
 \item Associativity of $\diagup$:\enspace  $
 \diagup(\delta'')=
\diagup( C_1,\dots,C_{i-1},\diagup(\delta'),C_{i+1},\dots,C_p).$
 \end{itemize}
\end{prop}

We spell out the link with usual associativity of a binary operation. Suppose that   we are given constructs $S:\hyper{H}_1$, $T:\hyper{H}_2$ and $U:\hyper{H}_3$. Then by Proposition \ref{associativity-diag} applied twice, we have  $
 \left( S \diagdown T \right) \diagdown U = \diagdown(S,T,U) = S \diagdown \left(T \diagdown U\right),
$
and likewise for $\diagup$.

\begin{remark}
Contrary to what happens for the associahedron,  $\diagdown$ and $\diagup$ are not duplicial (in the sense of \cite[Definition 5.1.1]{L08}). It can be checked for the permutohedron that the equality $(x \diagup y)\diagdown z = x\diagup(y\diagdown z)$, mixing $\diagup$ and $\diagdown$, is not satisfied. Indeed, in the permutohedron on $\{1,2,3\}$, we have:
$$
 (1\diagup 2) \diagdown 3 =2(1) \diagdown 3 = 2(1 \diagdown 3) = 2(1(3))
$$
and 
$$
 1\diagup (2 \diagdown 3) = 1\diagup (2(3))=2(1 \diagup 3) = 2(3(1)).
$$
This counter-example in fact extends to any hypergraph containing hyperedges which are not intervals.
\end{remark}

We return to our present goals.
For the sake of exposition, we give different notations for the different definitions of the shuffle product.
For a $\Xi$-delegation $\delta=((C_1:{\bf H}_1,\dots, C_p:{\bf H}_p),{\bf H})$ with $\text{root}(C_i)=X_i$, and a subset $\emptyset\neq B\subseteq \{1,\dots,p\}$, we write
\begin{align}
 \ast^{\lceil}(\delta) &=\sum_{U:\hyper{H} \;{\rm and}\; \forall 1\leq i\leq p, \restrconstr{U}{\hyper{H}_i}=C_i} q^{\mu^\tau(U)}\, U \label{prc}\\
\ast_B^{\lceil}(\delta)&= \sum_{U:\hyper{H} \;, {\rm root}(U)=X_B \; {\rm and}\; \forall 1\leq i\leq p, \restrconstr{U}{\hyper{H}_i}=C_i} q^{\mu^\tau(U)-|B|+1}\, U\\
\ast^\leq(\delta) &= \sum_{U:\hyper{H} \;{\rm and}\; \diagup(\delta) \leq U \leq\diagdown(\delta)} q^{\mu^\tau(U)}\, U \\
\ast^\leq_B(\delta) &= \sum_{U:\hyper{H} \;{\rm and}\; \diagup^B(\delta) \leq U \leq\diagdown_B(\delta)} q^{\mu^\tau(U)-|B|+1}\, U,
\end{align}
where $X_B=\bigcup_{b \in B} X_b$ and $\diagup^B(\delta)$ and $\diagdown_B(\delta)$ are defined as follows:
$$ 
\diagup^B(\delta)  =  X_B\left(\diagup(\delta^B_1), \ldots, \diagup(\delta^B_{n_B}) \right), 
\diagdown_B(\delta)   =   X_B\left(\diagdown(\delta^B_1), \ldots, \diagdown(\delta^B_{n_B}) \right),$$
with the induced delegations $\delta_i^B$, $1\leq i\leq n_B$, as in \eqref{deltas}.

\medskip

 We  now state our main result, whose proof relies on Lemma \ref{interval-inclusion} that follows.
\begin{theorem}\label{interval-product}
Given a $\Xi$-delegation $\delta=((C_1:{\bf H}_1,\dots, C_p:{\bf H}_p),{\bf H})$ and a subset $\emptyset\neq B\subseteq \{1,\dots,p\}$, the following equations hold:
\begin{equation}
\ast(\delta)\,=\,\ast^{\lceil}(\delta)\, = \,\ast^\leq(\delta) \label{eq:equalityProd}
\end{equation}
\begin{equation}
\ast_B(\delta)\, = \,\ast_B^{\lceil}(\delta)\, = \,\ast_B^\leq(\delta).\label{eq:equalityProdB}
\end{equation}  
\end{theorem}

\begin{lemma} \label{interval-inclusion}
Let $\delta=((C_1:{\bf H}_1,\dots, C_p:{\bf H}_p),{\bf H})$ be a 
 $\Xi$-delegation and let $\emptyset\neq B\subseteq \{1,\dots,p\}$. Then
\begin{itemize}
\item[a)] the following inequalities hold: \begin{align}
 \diagup(\delta)&\leq\diagup^B(\delta), \label{ineq1}\\
 \diagup^B(\delta) &\leq \diagdown_B(\delta), \label{ineq3} \\
 \diagdown_B(\delta)&\leq\diagdown(\delta). \label{ineq2}
\end{align}
\item[b)] for any $1\leq k\leq p$, we have that $$\restrconstr{\diagup(\delta)}{{\bf H}_i}=\restrconstr{\diagup^B(\delta)}{{\bf H}_i}=C_i \quad \mbox{ and } \quad \restrconstr{\diagdown(\delta)}{{\bf H}_i}=\restrconstr{\diagdown_B(\delta)}{{\bf H}_i}=C_i.$$ 
\end{itemize}
\end{lemma}

Before proving Lemma \ref{interval-inclusion}, we collect a few useful observations in the following lemma, which uses the notation of \eqref{decomp}. Recall that we place ourselves in the (hereditarily) ordered setting.

\begin{lemma} \label{subface-decomp}

Let $\delta=((C_1:{\bf H}_1,\dots, C_p:{\bf H}_p),{\bf H})$ be a 
 $\Xi$-delegation and let $\emptyset\neq B\subseteq \{1,\dots,p\}$. Suppose moreover that $B$ 
has at least two elements, and let $b_0\in B$ and $B'=B\setminus\set{b_0}$, so that
$X_B=X_{B'}\cup X_{b_0}$. Then we have the following properties:
\begin{enumerate}
\item
 There exists $1\leq i\leq n_{B'}$ such that $X_{b_0}\inc H_i^{B'}$.
 \item Setting $\hyper{H}_i^{B'}, X_{b_0}\leadsto H_1^{b_0},\ldots,H_m^{b_0}$, we have $n_B=n_{B'}+m-1$ and
 $$\begin{array}{c}
 H_1^B=H_1^{B'}, \dots , H_{i-1}^B=H_{i-1}^{B'}, \\
 H_i^B=H_1^{b_0} ,\dots , H_{i+m-1}^B=H_m^{b_0},\\
 H_{i+m}^B=H_{i+1}^{B'} , \dots , H_{n_B}^B=H_{n_{B'}}^{B'}.
 \end{array}$$
 so that the decomposition $\hyper{H}, X_B\leadsto H^B_1,\ldots,H^B_{n_B}$ can also be read as
 $$\hyper{H}, X_B\leadsto H^{B'}_1,\ldots,H^{B'}_{i-1},H^{b_0}_1,\ldots ,H^{b_0}_m,H^{B'}_{i+1},\ldots,H^{B'}_{n_{B'}}.$$
 \item For $i$ as above, we have that
 $$\begin{array}{lll}
 \delta_1^B=\delta_1^{B'},\ldots, \delta_{i-1}^B=\delta_{i-1}^{B'} &&
 \delta_{i+m}^B=\delta_{i+1}^{B'}, \ldots, \delta_{n_B}^B=\delta_{n_{B'}}^{B'}
 \end{array}$$
 and if $C=X_B(C_1,\ldots,C_{n_B})$ is a construct of $\hyper{H}$ with root $X_B$, then
 $X_{B'}(C_1,\ldots,C_{i-1},X_{b_0}(C_i,\ldots,C_{i+m-1}),C_{i+m},\ldots, C_{n_B})$ is a well-defined construct of
 $\hyper{H}$ which is by definition a facet of $C$.

 \item Further, 
 we have
 $(\delta_i^{B'})_l^{b_0}= \delta_{i+l-1}^B \quad (1\leq l\leq m)$.
 \end{enumerate}
\end{lemma}
\begin{proof} By easy reasoning on connected components.
\end{proof}

\begin{proof} [Proof of Lemma \ref{interval-inclusion}]
We prove statements (a) and (b) together by induction on $|H|$. By transitivity, the three inequalities in statement (a) imply the inequality 
\begin{align}
 \diagup(\delta)&\leq \diagdown(\delta). \label{ineq4}
\end{align}
We prove the four inequalities together. The base case holds vacuously.
The inequality (\ref{ineq3}) follows by induction from the inequalities (\ref{ineq4}) applied to $\diagup(\delta^B_i)$ and $\diagdown(\delta^B_i)$ for $1\leq i \leq n_B$ as both constructs $\diagup^B(\delta)$ and $\diagdown_B(\delta)$ have the same root $X_B=\bigcup_{b \in B} X_b$.

\smallskip
We are left with (\ref{ineq1}) and (\ref{ineq2}). We prove (\ref{ineq2}) only, as the proof of (\ref{ineq1}) is analogous.
We shall prove the following two claims:
\begin{enumerate}
\item If $|B| \geq 2$, then $\diagdown_B(\delta) \leq \diagdown_{B \setminus\set{\max(B)}}(\delta)$.
\item If $1\nin B$, $
 \diagdown_B(\delta)
 \leq
 \diagdown_{B \cup\{1\}}(\delta)$.
\end{enumerate}
Inequality (\ref{ineq2}) follows from these two claims: starting from $B$, adding $1$ to $B$ at the beginning if needed and then iteratively deleting the maximum of the current set, we get a sequence of sets $B_0, \ldots, B_p$ satisfying $B_0=B$, $B_p=\{1\}$ and $\diagdown_{B_i}(\delta)
 \leq
 \diagdown_{B_{i+1}}(\delta)$ for any $i$. 
 Hence $$\diagdown_{B}(\delta)=\diagdown_{B_{0}}(\delta)\leq\diagdown_{B_{p}}(\delta)=
 \diagdown_{\set{1}}(\delta)= \diagdown(\delta).$$
{\em  Proof of Claim 1.}
  We apply Lemma \ref{subface-decomp}, with $b_0=\max(B)$ and $B'~=~B\setminus~\set{b_0}$:
\begin{align*}
 \diagdown_B(\delta) &= X_B \left(\diagdown(\delta^B_1), \ldots, \diagdown(\delta^B_{n_B}) \right) \\
&\lessdot X_{B'}(\ldots, \diagdown(\delta^B_{i-1}),X_{b_0}(\diagdown(\delta^B_i),\ldots,\diagdown(\delta^B_{i+m-1})),\diagdown(\delta^B_{i+m}),\ldots)\\
& = X_{B'}(\ldots, \diagdown(\delta^{B'}_{i-1}),X_{b_0}(\diagdown((\delta^{B'}_i)^{b_0}_1),\ldots,\diagdown((\delta^{B'}_i)^{b_0}_{m})),\diagdown(\delta^{B'}_{i+1}),\ldots)\\
&= X_{B'}(\diagdown( \delta^{B'}_1),\ldots, \diagdown(\delta^{B'}_{i-1}),\diagdown_{\set{b_0}}(\delta^{B'}_i),\diagdown(\delta^{B'}_{i+1}),\ldots,\diagdown(\delta^{B'}_{n_{B'}}))\\
&\leq X_{B'}(\diagdown( \delta^{B'}_1),\ldots, \diagdown(\delta^{B'}_{i-1}),\diagdown(\delta^{B'}_i),\diagdown(\delta^{B'}_{i+1}),\ldots,\diagdown(\delta^{B'}_{n_{B'}})) \\
&= \diagdown_{B'}(\delta).
\end{align*}

 Above, the initial $\lessdot$ is justified by the inequality $B'<b_0$, which, because we are in the ordered setting, implies $X_{B'}<X_{b_0}$. The final $\leq$ follows by induction from inequality (\ref{ineq2}) applied to $\set{b_0}$ and $\delta^{B'}_i$.
 
 The proof of Claim 2 is similar, replacing $B,b_0,B'$ with $B\cup\set{1},1,B$, and using $\set{1} < B$. 

\medskip
We now address the statement (b) of Lemma \ref{interval-inclusion}. We first prove the equality $\restrconstr{\diagdown(\delta)}{{\bf H}_i}=C_i$ by induction on $|H|$. The initialisation is immediate.
We write $C_1=X_1(C_{(1,1)},\ldots,C_{(1,n_1)})$, {$\hyper{H}, X_1 \leadsto H_1^{\set{1}}, \ldots , H_n^{\set{1}}$ and $\hyper{H}_1, X_1 \leadsto H_{(1,1)}, \ldots ,H_{(1,n_1)}$}. 
We distinguish two cases.
\begin{itemize}
\item $k\neq 1$.
By strictness we have that $H_k\subseteq H_{i_k}^{\set{1}}$ for some $i_k$, hence $\restrconstr{\diagdown(\delta)}{{\bf H}_k}=\restrconstr{\diagdown(\delta_{i_k}^{\{1\}})}{{\bf H}_k}=C_k$ by definition of restriction and by induction hypothesis. 
\item $k=1$. By strictness each $H_{(1,j)}$ is included in some $H_{i_j}^{\{1\}}$, and the construct
sitting in position $(1,j)$ in the induced team $\delta_{i_j}^{\set{1}}$ is $C_{(1,j)}$. We then have $\restrconstr{\diagdown(\delta_{i_j}^{\{1\}})}{{\bf H}_{(1,j)}}=C_{(1,j)}$ by induction hypothesis.
Then we get (by definition of restriction)
\begin{align*}
\restrconstr{\diagdown(\delta)}{{\bf H}_1} & = X_1(\restrconstr{\diagdown(\delta_{i_1}^{\set{1}})}{{\bf H}_{(1,1)}},\ldots,
 \restrconstr{\diagdown(\delta_{i_{n_1}}^{\set{1}})}{{\bf H}_{(1,n_1)}})\\
& = X_1(C_{(1,1)},\ldots,C_{(1,n_1)}) \; = \; C_1. 
\end{align*}
\end{itemize}
 
The proof of 
$\restrconstr{\diagup(\delta)}{{\bf H}_i}=C_i$ is similar and left to the reader.

\medskip
The equalities $\restrconstr{\diagup^B(\delta)}{{\bf H}_i}=C_i$ and $\restrconstr{\diagdown_B(\delta)}{{\bf H}_i}=C_i$ follow from the equalities just proved, from the inequalities $\diagup(\delta)\leq \diagup^B(\delta)\leq \diagdown(\delta)$ that we established above, and from Proposition \ref{restriction-order-comm} and Proposition \ref{lessdot-order}.
\end{proof}

We are now in position for proving the main result of this section.

\begin{proof}[Proof of Theorem \ref{interval-product}]
We already proved the left-hand side equalities of \eqref{eq:equalityProd} and \eqref{eq:equalityProdB}  in \cite{PLBJ1} (cf. Proposition \ref{restriction-product-characterisation} above). 

\medskip
We need some preparation to establish the right-hand side equalities  \eqref{eq:equalityProd}${}^r$  (resp. \eqref{eq:equalityProdB}$^r$) of \eqref{eq:equalityProd} (resp. \eqref{eq:equalityProdB}).
We define the following properties on constructs of $\hyper{H}$, relative to a delegation $\delta$:
\begin{itemize}
\item $Q(U)$ whenever the coefficient of $U$ in $\ast(\delta)$ is non-zero.
\item $Q^\lceil(U)$ holds whenever $\restrconstr{U}{\hyper{H}_i}= C_i$ for all $1\leq i\leq p$;
\item $Q^\lessdot(U)$ holds whenever $\diagup(\delta) \leq U \leq\diagdown(\delta)$;

\end{itemize}
and their relative versions:
\begin{itemize}
\item $Q_B(U)$ whenever the coefficient of $U$ in $\ast_B(\delta)$ is non-zero.
\item $Q_B^\lceil(U)$ holds whenever $Q^\lceil(U)$ holds and ${\rm root}(U)=X_B$;
\item $Q_B^\lessdot(U)$ holds whenever $\diagup^B(\delta) \leq U \leq\diagdown_B(\delta)$.
\end{itemize}
We note that $Q(U)$ holds if and only if $Q_B(U)$ holds for some unique $B$.
From the discussion above, we have that $Q$ coincides with $Q^{\lceil}$, and $Q_B$ coincides with $Q_B^\lceil$.

We remark that \eqref{eq:equalityProd}${}^r$ and \eqref{eq:equalityProdB}${}^r$ (together) are equivalent to the following two properties: 
\begin{itemize}
\item[(A)] $Q^{\lessdot}(U) \Rightarrow Q^\lceil(U)$ and $Q_B^{\lessdot}(U) \Rightarrow Q_B^\lceil(U)$
\item[(B)] $Q^{\lceil}(U) \Rightarrow Q^\lessdot(U)$ and $Q_B^{\lceil}(U) \Rightarrow Q_B^\lessdot(U)$, which we conveniently reformulate as $Q(U) \Rightarrow Q^\lessdot(U)$ and $Q_B(U)\Rightarrow Q_B^\lessdot(U)$.
\end{itemize}
The two implications in (A) follow immediately 
 from Proposition \ref{restriction-order-comm} and from Proposition \ref{lessdot-order}.

For (B), we shall reason by induction on $U$. Suppose that 
$Q_B(U)$ holds. Then $U=X_B(U_1,\ldots,U_{n_B})$ and $Q(U_i)$ holds for all $i$ relatively to $\delta_i^B$.

Applying the induction hypothesis to $U_1, \ldots, U_{n_B}$, we get that the properties $Q^\lessdot(U_i)$ are true. Hence $Q^\lessdot_{B}(U)$ holds by Remark \ref{lem:StabilityByContext} and by definition of $\diagup^B(\delta)$ and $\diagdown_B(\delta)$. Finally, by the inequalities 
established in Lemma \ref{interval-inclusion}, we have that the following implication holds: if, for some $B$, $Q^\lessdot_B(U)$ holds, then $Q^\lessdot(U)$ holds. This finishes the proof of the theorem. 
\end{proof}

\begin{example} We illustrate the inequality \eqref{ineq2} of Lemma 4.6 in the classical setting of associahedra. For $\delta=((S:\hyper{K}^{V_1}, T:\hyper{K}^{V_2}), \hyper{K}^{V_1\cup V_2})$, $S=X(S _1,\ldots,S_k)$ and
 $T=Y(T_1,\ldots,T_l)$, we have: 
$$\begin{array}{rcl}
S\diagdown_{\set{1,2}} T &= &(X\cup Y)(S _1,\ldots,S_{k-1},S_k\diagdown T_1,T_2,\ldots , T_l) \\
 &\lessdot& X(S _1,\ldots,S_{k-1},Y(S_k\diagdown T_1,T_2,\ldots ,T_l)).
\end{array} $$ 
 
 \noindent We note that $Y$ is the root of the second construct in the delegation $(S_k,T)$, and not of the first.
 So there is quite a way still from there to $$S\diagdown T=X(S_1,\ldots,S_{k-1},S_k\diagdown T),$$ but 
 we do have $Y(S_k\diagdown T_1,T_2,\ldots, T_l)\leq S_k\diagdown T$ by induction, 
 and
 hence $$X(S _1,\ldots,S_{k-1},Y(S_k\diagdown T_1,T_2,\ldots , T_l))\leq S\diagdown T \quad \mbox{ and } \quad S\diagdown_{\{1,2\}} T\leq S\diagdown T.$$ 
 \end{example}

The following example shows that the interval formula for the product may not hold outside the strict setting.
\begin{example}\label{hcf}
Consider the 2-dimensional hypercube $${\bf C}^{\{1,2,3\}}=\{\{1\},\{2\},\{3\},\{1,2\},\{1,2,3\}\}$$ and the corresponding Hasse diagram given in Figure \ref{fig:Hassediag1}. Observe that the team $(\{{\bf C}^{\{1\}},{\bf C}^{\{2,3\}}\},{\bf C}^{\{1,2,3\}})$ is not strict. We have $$1\ast \{2,3\}=1(2,3)-\{1,2,3\}+\{2,3\}(1).$$ However, $$[\{2,3\}(1),1(2,3)]\supsetneq \{1(2,3),\{1,2,3\},\{2,3\}(1)\}.$$

\end{example}

\section{Discussion} \label{comparison-section}
We devote this section to positioning our generalised flip order with respect to other partial orders on the faces of certain polytopes proposed in the literature.
Specifically, we compare the GFO with the weak order previously introduced by Palacios and Ronco in \cite{PalaciosRonco} on the faces of $W$-permutohedra, and later further studied by Dermenjian, Hohlweg and Pilaud in \cite{DHP-WFO}, as well as with the generalised Tamari order on the faces of graph-associahedra, due to Ronco \cite{RoncoGTO}.

We shall henceforth denote the generalised flip order explicitly by $\leq_{GFO}$, in order to distinguish it from the other orders appearing in the discussion.

\subsection{Facial weak order}

The weak Bruhat order on permutations was first extended to all the faces of the permutohedra by Krob-Latapy-Novelli-Phan-Schwer \cite{KLNPS}. This order was then extended by Palacios-Ronco \cite{PalaciosRonco} to all $W$-permutohedra (which are polytopes associated with finite Coxeter systems), and further studied in this generality
by Dermenjian-Hohlweg-Pilaud \cite{DHP-WFO}, who called it facial weak order. We present here a version of the facial weak order that applies to all nestohedra, and investigate the link between this order and ours.

\begin{lemma}\label{lemmeWeak}
Let ${\bf H}$ be a hereditarily ordered hypergraph, let $S:{\bf H}$ and write $\downarrow^{\subseteq}\!({S}):=\{T:{\bf H}\,|\, T \subseteq S\}$ for the set of faces contained in $S$ (in the subface poset of constructs introduced in Definition \ref{rqueSubfaces}). 
The minimum $\bot_S$ (resp. maximum $\top_S$) of $\downarrow^{\subseteq}\!({S})$ in the generalised flip order is the construction defined by picking the root of $S$, squashing it if needed in order to place its 
largest (resp. smallest) element at the root, and then proceeding recursively in the immediate subconstructs of $S$ if the root of $S$ is a singleton, or of its squashing if the root is not a singleton.
In symbols, we have:
\begin{itemize}
\item
if 
$S = \set{m}(S_1,…,S_n)$, then 
$\bot_S:=\set{m}(\bot_{S_1},…,\bot_{S_n})$;
\item
if $\operatorname{root}(S) = X$, with $|X|>1$, $m=max(X)$ and $X’= X\setminus\set{m}$, then
$\bot_S := \bot_{S’}$ where $S’=S[\set{m}\langle X'\rangle/X]$.
\end{itemize}

\end{lemma}

\begin{proof} We prove the lemma for $\bot_S$; the proof for $\top_S$ is symmetric.
We proceed by induction on the number of vertices in the construct $S$. If $S$ has only one vertex (namely, $H$), then $\downarrow^{\subseteq}\!({S})$ is precisely the set of all constructs of ${\bf H}$ and hence the lemma follows from Lemma \ref{MaxWeakOrder}. Assume now that $S=X(S_1,\dots,S_n)$ has more than one vertex. 
\begin{enumerate}
\item If $|X|=1$, then for any construct $T\in \,\downarrow^{\subseteq}\!({S})$, we have that $\text{root}(T)=X$. Writing $T=X(T_1,\dots,T_n)$, we moreover have that $T_i\in\, \downarrow^{\subseteq}\!(S_i)$ for every $i$. By induction, the minimum of $\downarrow^{\subseteq}(S_i)$ is $\bot_{S_i}$. Therefore, by Remark \ref{lem:StabilityByContext}, we conclude that $\bot_S=X(\bot_{S_1},\dots,\bot_{S_n})\leq_{\mathit{GFO}} T$.
\item Assume now that $|X|\geq 2$, write $m:=\max(X)$ and let $T\in\,\downarrow^{\subseteq}\!(S)$. 
We distinguish two subcases:
\begin{enumerate}
\item If $m\in \text{root}(T)$, then $m$ is also the maximum of $\text{root}(T)$ (since $\text{root}(T)\subseteq X$ by definition of the subface poset). We define $T':=T[X\rightarrow m\langle X\backslash\{m\}\rangle]$. Since $T'\subseteq T\subseteq S$, we have that $T'\in \,\downarrow^{\subseteq}\!(S)$. 
Since $T'$ falls in case (1) that we already treated, we have $\bot_S \leq_{\mathit{GFO}} T'$.
On the other hand, since ${\bf H}$ is hereditarily ordered, we can ``undo'' the squash that defines $T'$ in terms of fusions of the GFO. Hence we have $\bot_S \leq_{\mathit{GFO}} T'\leq_{\mathit{GFO}} T$. 
\item Otherwise, let $Y$ be the vertex of $T$ containing $m$ and let $Z$ be the parent of $Y$ in $T$. By definition of the subface poset, we have that $Y,Z\in \text{root}(S)$ and $Y,Z<\{m\}$. Consider the construct $T'':=T[Y\rightarrow m\langle Y\backslash\{m\} \rangle][Z(m)\rightarrow Z\cup\{m\}]$. Observe that moving from $T$ to $T''$ strictly decreases the distance between the node containing $m$ and the root of the construct. Therefore, repeating this procedure finitely many times yields a construct $T'''$ with $m\in \text{root}(T''')$ and $T'''\in \,\downarrow^{\subseteq}\!(S)$, which brings us back to case (2a) or (1), so that we have $\bot_S\leq_{\mathit{GFO}} T'''$. It is easy to see that $T'''\leq_{\mathit{GFO}} T$ and hence
$\bot_S\leq_{\mathit{GFO}} T$. This concludes the proof.\vspace{-0.5cm}
\end{enumerate}
\end{enumerate}
\end{proof}

The idea of the next definition has been suggested by Vincent Pilaud 
 \cite{private-comm}.

\begin{definition} \label{fwo-def}
[Facial weak order] For a hypergraph H and  $S: {\bf H}$, let
\begin{align*}
\max\set{S}&=\max\{T: {\bf H}|\; T\subseteq S, T\:\mbox{is a construction} \} \;\mbox{and}\\
\min\set{S}&=\min\{T: {\bf H}|\:T\subseteq S, T\:\mbox{is a construction}\},
\end{align*}
where the minimum and maximum are considered relatively to the flip order $\leq_{\operatorname{flip}}$  (cf. Definition \ref{flip-order}).
Consider the   relations: $S \lessdot_{\mathit{FWO}} T$ if and only if
\begin{itemize}
 \item $S$ is a facet of $T$, $\min\set{S}=\min\set{T}$ and $\max\set{S}<_{\operatorname{flip}} \max\set{T}$,
 \item or $T$ is a facet of $S$, $\max\set{S}=\max\set{T}$ and $\min\set{S}<_{\operatorname{flip}} \min\set{T}$.
\end{itemize}

The transitive and reflexive closure $\leq_{\mathit{FWO}}$ of $S \lessdot_{\mathit{FWO}} T$ is an order called the \emph{facial weak order} (FWO).
\end{definition}

\begin{remark} \label{min-equals-bot}
\begin{itemize}
\item Note that if $S <_{\mathit{FWO}} T$, then either $\min\set{S}\leq \min\set{T}$ and $\max\set{S}\leq \max\set{T}$, at least one of these inequalities being strict. Therefore, $S <_{\mathit{FWO}} S$ implies either $\max\set{S}< \max\set{S}$ or $\min\set{S}< \min\set{S}$, implying that this order is well-defined.
\item Following Lemma \ref{lemmeWeak}, $\max\set{S}=\top_S$ and $\min\set{S}=\bot_S$ as both $\top_S$ and $\bot_S$ are constructions.
\item The facial weak order is not included in the generalised flip order on the hypercube (see Figure \ref{WvsF}).
\end{itemize}
\end{remark}

\begin{figure}[H]
\begin{center}
\resizebox{4.5cm}{!}{\begin{tikzpicture}[scale=0.85,
 hasse/.style={
 draw,
 rounded corners,
 fill=#1,
 inner sep=2pt
 },
 tree/.style={font=\footnotesize},
 edge/.style={thick}
]

\node[hasse=MutedLavender!25] (c1) at (0,5) {
 \begin{tikzpicture}[tree]
 \node (a) at (0,0) {$1$};
 \node (b) at (-0.35,0.6) {$2$};
 \node (c) at (0.35,0.6) {$3$};
 \draw[edge] (b)--(a)--(c);
 \end{tikzpicture}
};

\node[hasse=SlateBlue!25] (a1) at (-1.85,3.1) {
 \begin{tikzpicture}[tree]
 \node (a) at (0,0) {$\{1,3\}$};
 \node (b) at (0,0.7) {$2$};
 \draw[edge] (a)--(b);
 \end{tikzpicture}
};

\node[hasse=SageGreen!25] (a2) at (0,1) {
 \begin{tikzpicture}[tree]
 \node {$\{1,2,3\}$};
 \end{tikzpicture}
};

\node[hasse=SageGreen!25] (a3) at (3.8,1) {
 \begin{tikzpicture}[tree]
 \node (a) at (0,0) {$2$};
 \node (b) at (-0.35,0.6) {$1$};
 \node (c) at (0.35,0.6) {$3$};
 \draw[edge] (b)--(a)--(c);
 \end{tikzpicture}
};

\node[hasse=SlateBlue!25] (b1) at (1.85,3.1) {
 \begin{tikzpicture}[tree]
 \node (a) at (0,0) {$\{1,2\}$};
 \node (b) at (0,0.7) {$3$};
 \draw[edge] (a)--(b);
 \end{tikzpicture}
};

\node[hasse=DirtyYellow!25] (b2) at (1.85,-1.1) {
 \begin{tikzpicture}[tree]
 \node (a) at (0,0) {$\{2,3\}$};
 \node (b) at (0,0.7) {$1$};
 \draw[edge] (a)--(b);
 \end{tikzpicture}
};

\node[hasse=DirtyYellow!25] (b3) at (-1.85,-1.1) {
 \begin{tikzpicture}[tree]
 \node (a) at (0,0) {$3$};
 \node (b) at (0,0.7) {$\{1,2\}$};
 \draw[edge] (b)--(a);
 \end{tikzpicture}
};

\node[hasse=EggShell!25] (c2) at (0,-3) {
 \begin{tikzpicture}[tree]
 \node (a) at (0,0) {$3$};
 \node (b) at (0,0.6) {$2$};
 \node (c) at (0,1.2) {$1$};
 \draw[edge] (a)--(b)--(c);
 \end{tikzpicture}
};

\node[hasse=SageGreen!25] (d1) at (-3.8,1) {
 \begin{tikzpicture}[tree]
 \node (a) at (0,0) {$3$};
 \node (b) at (0,0.6) {$1$};
 \node (c) at (0,1.2) {$2$};
 \draw[edge] (a)--(b)--(c);
 \end{tikzpicture}
};

\draw[edge]
 (c2)--(b3)--(a2)--(b2)--(c2);
 
\draw[edge]
(b3)--(d1)--(a1)--(c1)--(b1)--(a3)--(b2);
\draw[edge]
(a2)--(b1);
\end{tikzpicture}}
\hspace{1cm}
\resizebox{4.5cm}{!}{\begin{tikzpicture}[scale=0.85,
 hasse/.style={
 draw,
 rounded corners,
 fill=#1,
 inner sep=2pt
 },
 tree/.style={font=\footnotesize},
 edge/.style={thick}
]

\node[hasse=MutedLavender!25] (c1) at (0,5) {
 \begin{tikzpicture}[tree]
 \node (a) at (0,0) {$1$};
 \node (b) at (-0.35,0.6) {$2$};
 \node (c) at (0.35,0.6) {$3$};
 \draw[edge] (b)--(a)--(c);
 \end{tikzpicture}
};

\node[hasse=SlateBlue!25] (a1) at (-1.85,3.1) {
 \begin{tikzpicture}[tree]
 \node (a) at (0,0) {$\{1,3\}$};
 \node (b) at (0,0.7) {$2$};
 \draw[edge] (a)--(b);
 \end{tikzpicture}
};

\node[hasse=SageGreen!25] (a2) at (0,1) {
 \begin{tikzpicture}[tree]
 \node {$\{1,2,3\}$};
 \end{tikzpicture}
};

\node[hasse=SageGreen!25] (a3) at (3.8,1) {
 \begin{tikzpicture}[tree]
 \node (a) at (0,0) {$2$};
 \node (b) at (-0.35,0.6) {$1$};
 \node (c) at (0.35,0.6) {$3$};
 \draw[edge] (b)--(a)--(c);
 \end{tikzpicture}
};

\node[hasse=SlateBlue!25] (b1) at (1.85,3.1) {
 \begin{tikzpicture}[tree]
 \node (a) at (0,0) {$\{1,2\}$};
 \node (b) at (0,0.7) {$3$};
 \draw[edge] (a)--(b);
 \end{tikzpicture}
};

\node[hasse=DirtyYellow!25] (b2) at (1.85,-1.1) {
 \begin{tikzpicture}[tree]
 \node (a) at (0,0) {$\{2,3\}$};
 \node (b) at (0,0.7) {$1$};
 \draw[edge] (a)--(b);
 \end{tikzpicture}
};

\node[hasse=DirtyYellow!25] (b3) at (-1.85,-1.1) {
 \begin{tikzpicture}[tree]
 \node (a) at (0,0) {$3$};
 \node (b) at (0,0.7) {$\{1,2\}$};
 \draw[edge] (b)--(a);
 \end{tikzpicture}
};

\node[hasse=EggShell!25] (c2) at (0,-3) {
 \begin{tikzpicture}[tree]
 \node (a) at (0,0) {$3$};
 \node (b) at (0,0.6) {$2$};
 \node (c) at (0,1.2) {$1$};
 \draw[edge] (a)--(b)--(c);
 \end{tikzpicture}
};

\node[hasse=SageGreen!25] (d1) at (-3.8,1) {
 \begin{tikzpicture}[tree]
 \node (a) at (0,0) {$3$};
 \node (b) at (0,0.6) {$1$};
 \node (c) at (0,1.2) {$2$};
 \draw[edge] (a)--(b)--(c);
 \end{tikzpicture}
};

\draw[edge]
 (c2)--(b3)--(a2)--(b2)--(c2);
 
\draw[edge]
(b3)--(d1)--(a1)--(c1)--(b1)--(a3)--(b2);
\draw[edge]
(a2)--(b1);
\draw[edge,very thick,red](a2)--(a1);
\end{tikzpicture}}
\end{center}
\caption{On the left, the GFO and on the right the FWO on the constructs of the hypergraph $\mathbf{C}^{1<2<3}=\{\{1\},\{2\},\{3\},\{1,2\}, \{1,2,3\}\}$.
Note that $\set{1,2,3}\leq_{\mathit{FWO}}\set{1,3}(2) \;\mbox{while}\; \set{1,2,3}\not\leq_{\mathit{GFO}}\set{1,3}(2)$.} \label{WvsF}
\end{figure}

In the following proposition, we show that (our version of) the facial weak order is an extension of the generalised flip order. 
\begin{prop}\label{ProofWeakOrder}
If $S \leq_{\mathit{GFO}} T$, then $S \leq_{\mathit{FWO}} T$.
\end{prop}

\begin{proof}
It is sufficient to prove the implication for the covering relations.
We observe that, by Remark \ref{min-equals-bot} and Proposition \ref{GFO-FO}, we can safely replace $\min\set{S}$ and $\min\set{T}$ with $\bot_S$ and $\bot_T$, and $\leq_{\operatorname{flip}}$ by $\leq_{\mathit{GFO}}$ in Definition \ref{fwo-def}.
If $S \lessdot T$ is a $(X,Y)$-fusion, then we have that $S \subseteq T$, and
 $S$ and $T$ coincide below the node containing $X$. Moreover, as $Y<X$, $\max(X \cup Y) = \max(X)$. One deduces easily from these observations that $\min\set{S}=\min\set{T}$.

To deal with the second inequality, let us note that, as $S$ is a facet of $T$, $\max\set{S} \leq \max\set{T}$ by definition. Once again, $S$ and $T$ coincide below the node containing $X$ and $\min(Y \cup Y)=\min(Y) \neq \min(X)$, hence the inequality is strict.

The case where $S <_{\mathit{GFO}} T$ is a split is treated in the same way.
\end{proof}

\subsection{Generalised Tamari order}
In \cite{RoncoGTO}, Mar\'{\i}a Ronco introduces a relation on all faces of ordered graph associahedra, i.e., graph associahedra for which additionally a total order on the set of vertices is given, which she calls generalised Tamari order (GTO). She defines
a product on certain families of graph associahedra. We refer to \cite{PLBJ2} for a comparison of her conditions with the ones considered in our works.
Her definition of product is given by means of restrictions of tubings, like we do in Proposition \ref{restriction-product-characterisation}. She gives a characterisation of this product in terms of intervals like we do in Theorem \ref{interval-product}. 
In fact, her work served as inspiration for ours. We note the following differences with respect to our work, though.
\begin{itemize}
\item Our framework is not limited to graph associahedra and takes place in the more general setting of nestohedra. 
\item Our characterisation via intervals extends to a tridendriform structure decomposing the product. In doing so (cf. \eqref{eq:equalityProdB}), we were inspired by the work of Palacios and Ronco in which the authors worked out the case of associahedra and permutohedra in detail \cite{PalaciosRonco}.
\item More importantly, our generalised flip order differs from her general Tamari order, as we show in Figure \ref{Rdiff}. The discrepancy  appears for linear graphs whose graph-theoretic ordering of vertices is not compatible with the ambient linear order. We expect that Ronco's covering relations coincide with ours for associahedra equipped with their standard ordering, that is, when the chosen linear order agrees with the order induced by the underlying linear graph. But one cannot hope to induce a preferred total order on the vertices from a graph structure in general.
\end{itemize}
 
\begin{figure}[H]
\begin{center}
\resizebox{4.5cm}{!}{\begin{tikzpicture}[scale=0.85,
 hasse/.style={
 draw,
 rounded corners,
 fill=#1,
 inner sep=2pt
 },
 tree/.style={font=\footnotesize},
 edge/.style={thick}
]

\node[hasse=espresso!25] (c1) at (0,5) {
 \begin{tikzpicture}[tree]
 \node (a) at (0,0) {$1$};
 \node (b) at (0,0.6) {$2$};
 \node (c) at (0,1.2) {$3$};
 \draw[edge] (a)--(b)--(c);
 \end{tikzpicture}
};

\node[hasse=SageGreen!25] (a1) at (-3.8,1) {
 \begin{tikzpicture}[tree]
 \node (a) at (0,0) {$1$};
 \node (b) at (0,0.6) {$3$};
 \node (c) at (0,1.2) {$2$};
 \draw[edge] (a)--(b)--(c);
 \end{tikzpicture}
};

\node[hasse=SageGreen!25] (a2) at (0,1) {
 \begin{tikzpicture}[tree]
 \node {$\{1,2,3\}$};
 \end{tikzpicture}
};

\node[hasse=SlateBlue!25] (a3) at (3.8,1) {
 \begin{tikzpicture}[tree]
 \node (a) at (0,0) {$2$};
 \node (b) at (0,0.7) {$\{1,3\}$};
 \draw[edge] (b)--(a);
 \end{tikzpicture}
};

\node[hasse=DustyRed!25] (b1) at (1.65,3.5) {
 \begin{tikzpicture}[tree]
 \node (a) at (0,0) {$\{1,2\}$};
 \node (b) at (0,0.7) {$3$};
 \draw[edge] (a)--(b);
 \end{tikzpicture}
};

\node[hasse=SlateBlue!25] (e1) at (-1.65,3.5) {
 \begin{tikzpicture}[tree]
 \node (a) at (0,0) {$1$};
 \node (b) at (0,0.7) {$\{2,3\}$};
 \draw[edge] (a)--(b);
 \end{tikzpicture}
};

\node[hasse=DirtyYellow!25] (b2) at (1.65,-1.5) {
 \begin{tikzpicture}[tree]
 \node (a) at (0,0) {$\{2,3\}$};
 \node (b) at (0,0.7) {$1$};
 \draw[edge] (a)--(b);
 \end{tikzpicture}
};

\node[hasse=DirtyYellow!25] (e2) at (-1.65,-1.5) {
 \begin{tikzpicture}[tree]
 \node (a) at (0,0) {$\{1,3\}$};
 \node (b) at (0,0.7) {$2$};
 \draw[edge] (a)--(b);
 \end{tikzpicture}
};

\node[hasse=EggShell!25] (c2) at (0,-3) {
 \begin{tikzpicture}[tree]
 \node (a) at (0,0) {$3$};
 \node (b) at (-0.35,0.6) {$1$};
 \node (c) at (0.35,0.6) {$2$};
 \draw[edge] (b)--(a)--(c);
 \end{tikzpicture}
};

\node[hasse=SageGreen!25] (d1) at (2.65,-0.25) {
 \begin{tikzpicture}[tree]
 \node (a) at (0,0) {$2$};
 \node (b) at (0,0.6) {$3$};
 \node (c) at (0,1.2) {$1$};
 \draw[edge] (a)--(b)--(c);
 \end{tikzpicture}
};

\node[hasse=MutedLavender!25] (d2) at (2.65,2.25) {
 \begin{tikzpicture}[tree]
 \node (a) at (0,0) {$2$};
 \node (b) at (0,0.6) {$1$};
 \node (c) at (0,1.2) {$3$};
 \draw[edge] (a)--(b)--(c);
 \end{tikzpicture}
};

\draw[edge,very thick,red]
(b2)--(c2);
 \draw[edge]
(e1)--(a2);
\draw[edge]
(b2)--(d1)--(a3)--(d2)--(b1)--(c1)--(e1)--(a1)--(e2)--(c2);
\draw[edge]
(b1)--(a2)--(b2);
\end{tikzpicture}}
\hspace{1cm}
\resizebox{4.5cm}{!}{\begin{tikzpicture}[scale=0.85,
 hasse/.style={
 draw,
 rounded corners,
 fill=#1,
 inner sep=2pt
 },
 tree/.style={font=\footnotesize},
 edge/.style={thick}
]

\node[hasse=espresso!25] (c1) at (0,5) {
 \begin{tikzpicture}[tree]
 \node (a) at (0,0) {$1$};
 \node (b) at (0,0.6) {$2$};
 \node (c) at (0,1.2) {$3$};
 \draw[edge] (a)--(b)--(c);
 \end{tikzpicture}
};

\node[hasse=SageGreen!25] (a1) at (-3.8,1) {
 \begin{tikzpicture}[tree]
 \node (a) at (0,0) {$1$};
 \node (b) at (0,0.6) {$3$};
 \node (c) at (0,1.2) {$2$};
 \draw[edge] (a)--(b)--(c);
 \end{tikzpicture}
};

\node[hasse=SageGreen!25] (a2) at (0,1) {
 \begin{tikzpicture}[tree]
 \node {$\{1,2,3\}$};
 \end{tikzpicture}
};

\node[hasse=SlateBlue!25] (a3) at (3.8,1) {
 \begin{tikzpicture}[tree]
 \node (a) at (0,0) {$2$};
 \node (b) at (0,0.7) {$\{1,3\}$};
 \draw[edge] (b)--(a);
 \end{tikzpicture}
};

\node[hasse=DustyRed!25] (b1) at (1.65,3.5) {
 \begin{tikzpicture}[tree]
 \node (a) at (0,0) {$\{1,2\}$};
 \node (b) at (0,0.7) {$3$};
 \draw[edge] (a)--(b);
 \end{tikzpicture}
};

\node[hasse=SlateBlue!25] (e1) at (-1.65,3.5) {
 \begin{tikzpicture}[tree]
 \node (a) at (0,0) {$1$};
 \node (b) at (0,0.7) {$\{2,3\}$};
 \draw[edge] (a)--(b);
 \end{tikzpicture}
};

\node[hasse=DirtyYellow!25] (b2) at (1.65,-1.5) {
 \begin{tikzpicture}[tree]
 \node (a) at (0,0) {$\{2,3\}$};
 \node (b) at (0,0.7) {$1$};
 \draw[edge] (a)--(b);
 \end{tikzpicture}
};

\node[hasse=DirtyYellow!25] (e2) at (-1.65,-1.5) {
 \begin{tikzpicture}[tree]
 \node (a) at (0,0) {$\{1,3\}$};
 \node (b) at (0,0.7) {$2$};
 \draw[edge] (a)--(b);
 \end{tikzpicture}
};

\node[hasse=EggShell!25] (c2) at (0,-3) {
 \begin{tikzpicture}[tree]
 \node (a) at (0,0) {$3$};
 \node (b) at (-0.35,0.6) {$1$};
 \node (c) at (0.35,0.6) {$2$};
 \draw[edge] (b)--(a)--(c);
 \end{tikzpicture}
};

\node[hasse=SageGreen!25] (d1) at (2.65,-0.25) {
 \begin{tikzpicture}[tree]
 \node (a) at (0,0) {$2$};
 \node (b) at (0,0.6) {$3$};
 \node (c) at (0,1.2) {$1$};
 \draw[edge] (a)--(b)--(c);
 \end{tikzpicture}
};

\node[hasse=MutedLavender!25] (d2) at (2.65,2.25) {
 \begin{tikzpicture}[tree]
 \node (a) at (0,0) {$2$};
 \node (b) at (0,0.6) {$1$};
 \node (c) at (0,1.2) {$3$};
 \draw[edge] (a)--(b)--(c);
 \end{tikzpicture}
};


 \draw[edge]
(e1)--(a2);
\draw[edge]
(b2)--(d1)--(a3)--(d2)--(b1)--(c1)--(e1)--(a1)--(e2)--(c2);
\draw[edge]
(b1)--(a2)--(b2);
\end{tikzpicture}} 
\end{center} 
\caption{On the left, the GFO, and on the right, the GTO, on the constructs of the hypergraph $\hyper{K}^{1<3<2}=\set{\set{1},\set{2},\set{3},\set{1,3},\set{3,2}}$. We have $3(1,2)\lessdot \set{2,3}(1)$, but the pair
$(3(1,2),\set{2,3}(1))$ is not in Ronco's relation (for the same choice of total order on the vertices of $\hyper{H}$). To see this, we translate the constructs in the language of tubings, giving $\set{\set{1},\set{2}}$ and $\set{\set{1}}$, respectively (note that Ronco takes as convention not to include the full tube $\set{1,2,3}$ in the tubings). Then we should be in one of the two situations described in case (1) of Definition 4.2 in \cite{RoncoGTO}. But setting $T=\set{\set{1},\set{2}}$ and $t=\set{2}$, we see that clause (a) does not apply since $t$ is maximal in $T$, and that clause (b) does not apply either as $\min(t)=2\neq 1$.}\label{Rdiff}
\end{figure}

\subsection{Future work}

This work suggests several research tracks. We list some of them below.

\begin{enumerate}
\item Krob-Latapy-Novelli-Phan-Schwer \cite{KLNPS} prove that the generalised flip order on permutohedra is a lattice. Dermenjian-Hohlweg-Pilaud \cite{DHP-WFO} prove the same result for the facial boolean lattice on the faces of the cube and for the facial Cambrian lattice on the faces of the corresponding generalised associahedron. Moreover, Barnard-McConville \cite{BM} exhibit some flip orders which are not a lattice and some conjectural forbidden motives in the associated graphs. The examples of generalised flip order presented in this article are all lattices but the ones of Figure 8 and Figure 10. Note that the example of Figure 10 is not associated with one of Barnard-McConville's forbidden motives: the poset restricted to constructions is a lattice. Several questions arise from these examples:
\begin{itemize}
\item Are the generalised flip orders associated to right-filled hypergraphs also lattices, as it is the case for Barnard-McConville's flip order?
\item The example of Figure 10 shows that the flip order can be a lattice, without the generalised flip order being one. Is the converse true? Is there an example of hypergraph whose generalised flip order is a lattice but whose restriction to constructions is not a lattice? 
\item More generally, what conditions on the hypergraph ${\bf H}$ make the generalised flip order a lattice?
\item In this case, could we find a characterisation of the supremum?
\end{itemize}
\item In \cite{DHP-WFO}, the authors introduced  several equivalent definitions for the facial weak order. In Proposition \ref{ProofWeakOrder}, we have related one of their definitions to our generalised flip order by an inclusion that is not strict in general. 
This suggests two orthogonal research directions:
\begin{itemize}
\item Can we characterise the polytopes on which the two orders coincide?
\item  Can one of the definitions of Dermenjian-Hohlweg-Pilaud be adapted to obtain an order equivalent to the one presented here? 
\end{itemize} 
\item The number of constructs of a hypergraph {\bf H} is computed inductively: can we find  a closed formula depending only on the shape of {\bf H}?
\item The proof of the link between the flip order and the generalised flip order in \S \ref{2.2} is reminiscent of a shellability proof as we endow chains in the poset with a weight. It differs from an EL-labelling as split covering relations are not labelled. However, would it be possible to find a suitable labelling for split covering relations which would endow these posets with an EL-labelling? More generally, what can be said on the topological properties of the generalised flip order?
\item Another natural direction concerns the limitations of the interval description of the shuffle product established in this article. As Example \ref{hcf} suggests, such a description does not work in the setting of hypercubes. It would be interesting to determine where exactly our order-theoretic results fail outside of the setting given by strict clans.
\item Our hereditarily ordered hypothesis plays a pervasive role in this work (except for the proof of cycle-freeness of the GFO). This condition is undebatable when it comes to define our polydendriform products inductively, since we need the induced teams to be ordered, starting from an ordered team. But we may ask if, setting this definition aside, the equivalence between the other two non-inductive definitions (based on restrictions and intervals, respectively) would still hold without this condition, or imposing a milder one. We raise the same question for the coincidence of GFO -- Barnard-McConville's flip order on 0-dimensional faces.
 \end{enumerate}

 \section*{Appendix}
 This appendix contains three pictures illustrating the GFO. They have been instrumental for the genesis of its characterisation in terms of generalised inversions in \S \ref{inversion-subsection}.
\begin{figure}[H]
\centering
\resizebox{9cm}{!}
 {\begin{tikzpicture}[scale=0.85,  
  hasse/.style={
    draw,
    rounded corners,
    fill=#1,
    inner sep=1.75pt
  },
  tree/.style={font=\footnotesize},
]

\node[hasse=espresso!25] (r0) at (0,-15) {
  \begin{tikzpicture}[tree]
    \node (a) at (0,0) {$4$};
    \node (b) at (0,0.6) {$3$};
    \node (c) at (-0.35,1.2) {$1$};
 \node (d) at (0.35,1.2) {$2$};
    \draw[thick] (a)--(b)--(c); \draw[thick](b)--(d);
  \end{tikzpicture}
};

 \node[hasse=espresso!50!EggShell!40] (r1) at (-3.5,-12) {
  \begin{tikzpicture}[tree]
  \node (a) at (0,0) {$4$};
    \node (b) at (0,0.6) {$\{2,3\}$};
    \node (c) at (0,1.2) {$1$};
    \draw[thick] (a)--(b)--(c);  
  \end{tikzpicture}
};

\node[hasse=espresso!45!EggShell!35] (r2) at (-3,-9) {
  \begin{tikzpicture}[tree]
   \node (a) at (0,0) {$4$};
    \node (b) at (0,0.6) {$2$};
    \node (c) at (-0.35,1.2) {$1$};
 \node (d) at (0.35,1.2) {$3$};
    \draw[thick] (a)--(b)--(c); \draw[thick](b)--(d);
  \end{tikzpicture}
};

\node[hasse=espresso!40!EggShell!30] (r3) at (2,-6) {
  \begin{tikzpicture}[tree]
  \node (a) at (0,0) {$4$};
    \node (b) at (0,0.6) {$\{1,2\}$};
    \node (c) at (0,1.2) {$3$};
    \draw[thick] (a)--(b)--(c);  
  \end{tikzpicture}
};

\node[hasse=espresso!5!EggShell!25] (r7) at (5.5,4) {
  \begin{tikzpicture}[tree]
  \node (a) at (0,0) {$1$};
    \node (b) at (0,0.6) {$\{2,4\}$};
    \node (c) at (0,1.2) {$3$};
    \draw[thick] (a)--(b)--(c);  
  \end{tikzpicture}
};

\node[hasse=EggShell!25] (r8) at (0,7) {
  \begin{tikzpicture}[tree]
    \node (a) at (0,0) {$1$};
    \node (b) at (0,0.6) {$2$};
    \node (c) at (0,1.2) {$4$};
 \node (d) at (0,1.8) {$3$};
    \draw[thick] (a)--(b)--(c)--(d);
  \end{tikzpicture}
};

\node[hasse=EggShell!15] (r9) at (-0,10.5) {
  \begin{tikzpicture}[tree]
    \node (a) at (0,0) {$1$};
    \node (b) at (0,0.6) {$2$};
    \node (c) at (0,1.2) {$\{3,4\}$};
    \draw[thick] (a)--(b)--(c);
  \end{tikzpicture}
};

\node[hasse=EggShell!5] (r10) at (0,14) {
  \begin{tikzpicture}[tree]
    \node (a) at (0,0) {$1$};
    \node (b) at (0,0.6) {$2$};
    \node (c) at (0,1.2) {$3$};
 \node (d) at (0,1.8) {$4$};
    \draw[thick] (a)--(b)--(c)--(d);
  \end{tikzpicture}
};

\node[hasse=espresso!50!EggShell!40] (l1) at (0,-12) {
  \begin{tikzpicture}[tree]
    \node (a) at (0,0) {$\{3,4\}$};
    \node (b) at (-0.35,0.6) {$1$};
    \node (c) at (0.35,0.6) {$2$};
    \draw[thick] (b)--(a)--(c);  
  \end{tikzpicture}
};

\node[hasse=espresso!45!EggShell!35] (l2) at (6,-9) {
  \begin{tikzpicture}[tree]
  \node (a) at (0,0) {$\{1,3,4\}$};
    \node (b) at (0,0.6) {$2$};
      \draw[thick] (a)--(b);  
  \end{tikzpicture}
};

\node[hasse=espresso!40!EggShell!30] (l3) at (4.25,-6) {
  \begin{tikzpicture}[tree]
    \node (a) at (0,0) {$\{1,3\}$};
    \node (b) at (-0.35,0.6) {$2$};
    \node (c) at (0.35,0.6) {$4$};
    \draw[thick] (b)--(a)--(c);  
  \end{tikzpicture}
};


\node[hasse=espresso!50!EggShell!40] (m0) at (4.5,-12) {
  \begin{tikzpicture}[tree]
    \node (a) at (0,0) {$4$};
    \node (b) at (0,0.6) {$\{1,3\}$};
    \node (c) at (0,1.2) {$2$};
    \draw[thick] (a)--(b)--(c);  
  \end{tikzpicture}
};

\node[hasse=espresso!5!EggShell!25] (m1) at (9.1,3) {
  \begin{tikzpicture}[tree]
  \node (a) at (0,0) {$1$};
    \node (b) at (0,0.6) {$\{3,4\}$};
    \node (c) at (0,1.2) {$2$};
    \draw[thick] (a)--(b)--(c);  
  \end{tikzpicture}
};

\node[hasse=EggShell!25] (m2) at (7.5,7) {
  \begin{tikzpicture}[tree]
   \node (a) at (0,0) {$1$};
    \node (b) at (0,0.6) {$3$};
    \node (c) at (-0.35,1.2) {$2$};
 \node (d) at (0.35,1.2) {$4$};
    \draw[thick] (a)--(b)--(c); \draw[thick](b)--(d);
  \end{tikzpicture}
};

\node[hasse=EggShell!15] (m3) at (4,10.5) {
  \begin{tikzpicture}[tree]
    \node (a) at (0,0) {$1$};
    \node (b) at (0,0.6) {$\{2,3\}$};
    \node (c) at (0,1.2) {$4$};
    \draw[thick] (a)--(b)--(c);  
  \end{tikzpicture}
};

\node[hasse=espresso!45!EggShell!35] (m4) at (3,-9) {
  \begin{tikzpicture}[tree]
  \node (a) at (0,0) {$4$};
    \node (b) at (0,0.6) {$\{1,2,3\}$};
      \draw[thick] (a)--(b);  
  \end{tikzpicture}
};

\node[hasse=espresso!45!EggShell!35] (22) at (-6.5,-9) {
  \begin{tikzpicture}[tree]
  \node (a) at (0,0) {$\{2,3,4\}$};
    \node (b) at (0,0.6) {$1$};
      \draw[thick] (a)--(b);  
  \end{tikzpicture}
};

\node[hasse=espresso!45!EggShell!35] (24) at (0,-9) {
  \begin{tikzpicture}[tree]
    \node (a) at (0,0) {$3$};
    \node (b) at (-0.35,0.6) {$1$};
    \node (c) at (0.35,0.6) {$4$};
 \node (d) at (0,0.6) {$2$};
    \draw[thick] (b)--(a)--(c); \draw[thick] (a)--(d);
  \end{tikzpicture}
};

\node[hasse=espresso!40!EggShell!30] (32) at (-5.25,-6) {
  \begin{tikzpicture}[tree]
    \node (a) at (0,0) {$\{2,3\}$};
    \node (b) at (-0.35,0.6) {$1$};
    \node (c) at (0.35,0.6) {$4$};
    \draw[thick] (b)--(a)--(c);  
  \end{tikzpicture}
};

\node[hasse=espresso!35!EggShell!25] (41) at (-2,-3.5) {
  \begin{tikzpicture}[tree]
  \node (a) at (0,0) {$\{1,2,3\}$};
    \node (b) at (0,0.6) {$4$};
      \draw[thick] (a)--(b);  
  \end{tikzpicture}
};

\node[hasse=espresso!5!EggShell!25] (72) at (-6,4) {
  \begin{tikzpicture}[tree]
  \node (a) at (0,0) {$\{1,2\}$};
    \node (b) at (0,0.6) {$3$};
    \node (c) at (0,1.2) {$4$};
    \draw[thick] (a)--(b)--(c);  
  \end{tikzpicture}
};

\node[hasse=espresso!15!EggShell!25] (63) at (-8.5,1.5) {
  \begin{tikzpicture}[tree]
  \node (a) at (0,0) {$2$};
    \node (b) at (-0.35,0.6) {$1$};
\node (c) at (0.35,0.6) {$3$};
\node (d) at (0.35,1.2) {$4$};
      \draw[thick] (d)--(c)--(a)--(b);  
  \end{tikzpicture}
};

\node[hasse=espresso!40!EggShell!30] (33) at (0,-6) {
  \begin{tikzpicture}[tree]
    \node (a) at (0,0) {$\{2,4\}$};
    \node (b) at (-0.35,0.6) {$1$};
    \node (c) at (0.35,0.6) {$3$};
    \draw[thick] (b)--(a)--(c);  
  \end{tikzpicture}
};
\node[hasse=espresso!40!EggShell!30] (36) at (-2.25,-6) {
  \begin{tikzpicture}[tree]
  \node (a) at (0,0) {$\{1,2,3,4\}$};
  \end{tikzpicture}
};

\node[hasse=espresso!35!EggShell!25] (411) at (2,-3.5) {
  \begin{tikzpicture}[tree]
  \node (a) at (0,0) {$\{1,2,4\}$};
    \node (b) at (0,0.6) {$3$};
      \draw[thick] (a)--(b);  
  \end{tikzpicture}
};

\node[hasse=espresso!35!EggShell!25] (43) at (-5,-3.5) {
  \begin{tikzpicture}[tree]
  \node (a) at (0,0) {$2$};
    \node (b) at (-0.35,0.6) {$1$};
\node (c) at (0.35,0.6) {$4$};
\node (d) at (0.35,1.2) {$3$};
      \draw[thick] (d)--(c)--(a)--(b);  
  \end{tikzpicture}
};

\node[hasse=EggShell!25] (44) at (5.5,6.75) {
  \begin{tikzpicture}[tree]
  \node (a) at (0,0) {$1$};
    \node (b) at (0,0.6) {$\{2,3,4\}$};
      \draw[thick] (a)--(b);  
  \end{tikzpicture}
};

\node[hasse=espresso!25!EggShell!25] (54) at (-8.5,-1.5) {
  \begin{tikzpicture}[tree]
  \node (a) at (0,0) {$2$};
    \node (b) at (-0.35,0.6) {$1$};
    \node (c) at (0.35,0.6) {$\{3,4\}$};
    \draw[thick] (b)--(a)--(c);  
  \end{tikzpicture}
};

\node[hasse=espresso!25!EggShell!25] (53) at (-2,-1) {
  \begin{tikzpicture}[tree]
  \node (a) at (0,0) {$\{1,2\}$};
    \node (b) at (0,0.6) {$4$};
    \node (c) at (0,1.2) {$3$};
    \draw[thick] (a)--(b)--(c);  
  \end{tikzpicture}
};

\node[hasse=espresso!15!EggShell!25] (62) at (-6,1.5) {
  \begin{tikzpicture}[tree]
   \node (a) at (0,0) {$\{1,2\}$};
    \node (b) at (0,0.6) {$\{3,4\}$};
    \draw[thick] (a)--(b);
  \end{tikzpicture}
};

\node[hasse=espresso!35!EggShell!25!SlateBlue!50] (r4) at (6,-3.5) {
  \begin{tikzpicture}[tree]
   \node (a) at (0,0) {$4$};
    \node (b) at (0,0.6) {$1$};
    \node (c) at (-0.35,1.2) {$2$};
 \node (d) at (0.35,1.2) {$3$};
    \draw[thick] (a)--(b)--(c); \draw[thick](b)--(d);
  \end{tikzpicture}
};

\node[hasse=espresso!25!EggShell!25!SlateBlue!30] (r5) at (6.8,-1) {
  \begin{tikzpicture}[tree]
    \node (a) at (0,0) {$\{1,4\}$};
    \node (b) at (-0.35,0.6) {$2$};
    \node (c) at (0.35,0.6) {$3$};
    \draw[thick] (b)--(a)--(c);  
  \end{tikzpicture}
};
 
\node[hasse=espresso!15!EggShell!25!SlateBlue!10] (r6) at (7.25,1) {
  \begin{tikzpicture}[tree]
   \node (a) at (0,0) {$1$};
    \node (b) at (0,0.6) {$4$};
    \node (c) at (-0.35,1.2) {$2$};
 \node (d) at (0.35,1.2) {$3$};
    \draw[thick] (a)--(b)--(c); \draw[thick](b)--(d);
  \end{tikzpicture}
};
 
\draw[very thick,Graphite!70] (r0)--(r1)--(r2)--(r3);
\draw[very thick,DustyRed!70] (r3)--(r4);
\draw[very thick,Graphite!70] (r4)--(r5)--(r6)--(r7)--(r8)--(r9)--(r10);
\draw[very thick,Graphite!70] (r0)--(l1)--(l2)--(l3)--(m2)--(m3)--(r10); 
\draw[very thick,Graphite!70] (m1)--(m2)--(m3)--(r10);
\draw[very thick,Graphite!70] (22)--(r1)--(m4)--(r3);
\draw[very thick,Graphite!70] (r0)--(m0)--(l2)--(m1);
\draw[very thick,Graphite!70](m0)--(r4);
\draw[very thick,Graphite!70](22)--(l1)--(24)--(32)--(22)--(36)--(m4);
\draw[very thick,Graphite!70](24)--(l3);
\draw[very thick,Graphite!70](72)--(63)--(32)--(41)--(m3);
\draw[very thick,Graphite!70](41)--(72)--(r10);
\draw[very thick,Graphite!70] (r2)--(33)--(411)--(r7);
\draw[very thick,Graphite!70] (43)--(33);
\draw[very thick,Graphite!70] (m1)--(44)--(36)--(41);
\draw[very thick,Graphite!70] (r9)--(44)--(m3);
\draw[very thick,Graphite!70] (63)--(54)--(62)--(53)--(43)--(54)--(22);
\draw[very thick,Graphite!70] (r9)--(62)--(72);
\draw[very thick,Graphite!70](r3)--(411)--(53);
\draw[very thick,Graphite!70] (36)--(62);
\draw[very thick,Graphite!70] (r6)--(m1);
\draw[very thick,Graphite!70] (r8)--(53);
 
\end{tikzpicture}}
\caption{The Hasse diagram of the GFO for the hypergraph $\hyper{H}^\maltese=\{\{1\},\{2\},\{3\},\{4\},\{3,4\}, \{1,2,3\},\{2,3,4\}\}$, which is hereditarily ordered, but not right-filled. A witness of the failure of Theorem \ref{gfo-inv} for ${\bf H}$ is the covering relation highlighted in red, which abolishes the good pair $(2,3)$ of the bottom construct.}
\label{hassediag3}
\end{figure}
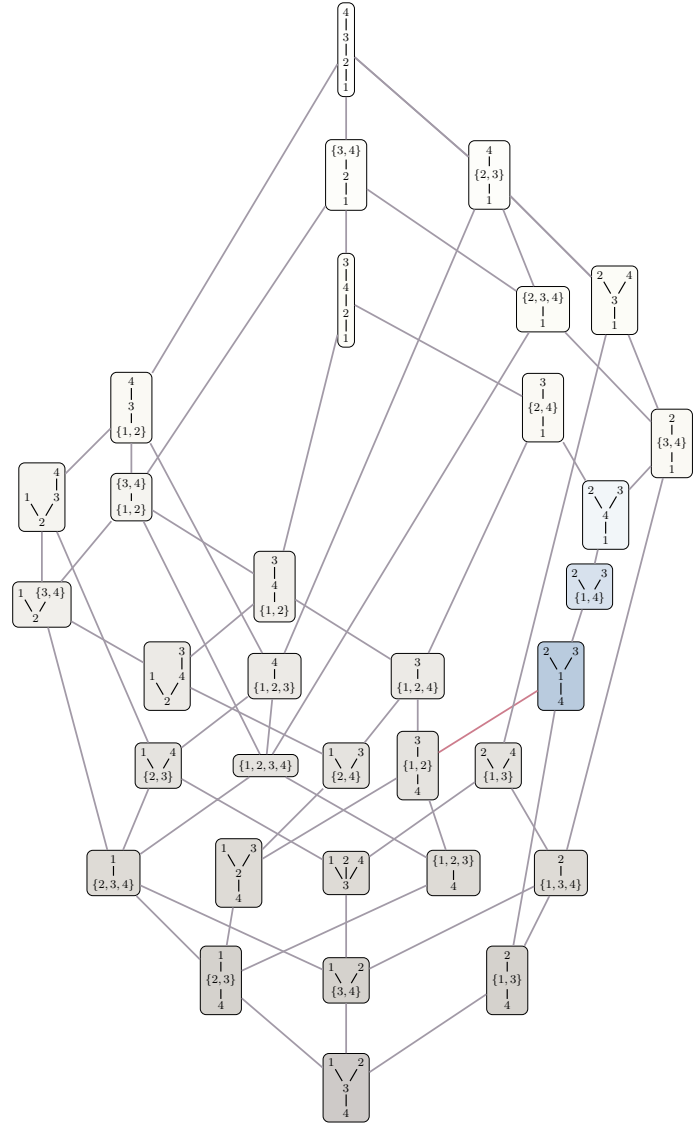
\begin{figure}[H]
\centering
\input{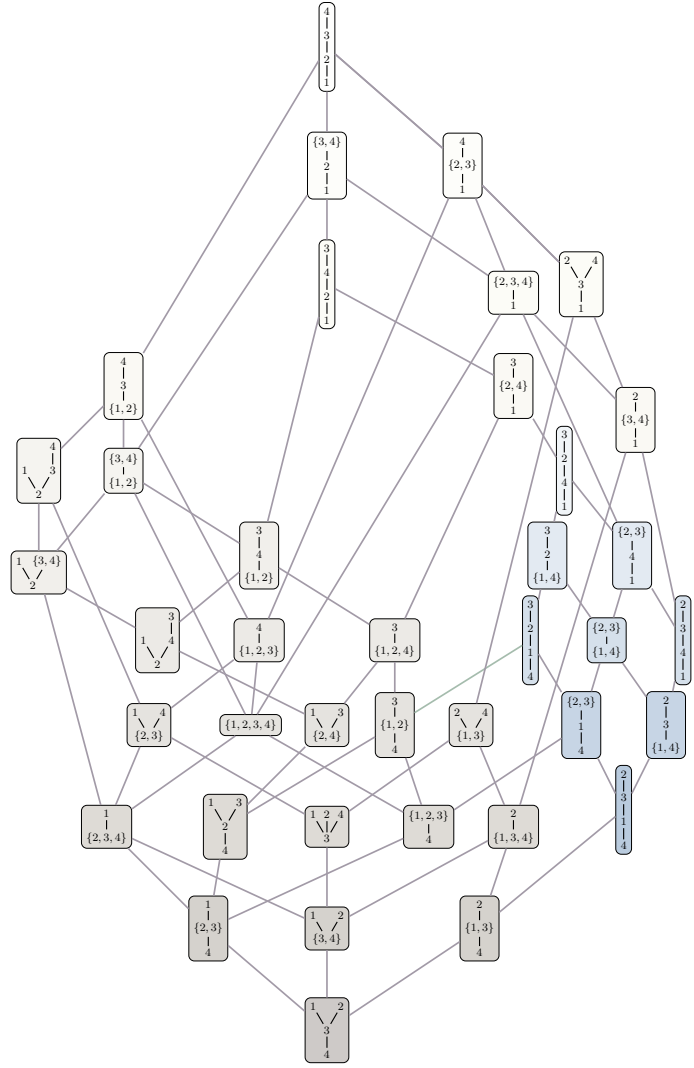}
\caption{The Hasse diagram of the GFO for the right-filled (and hence hereditarily ordered) hypergraph $\hyper{H}^{\text{\ding{170}}}= \hyper{H}^\maltese\cup \{2,3\}=\{\{1\},\{2\},\{3\},\{4\},\{2,3\},\{3,4\},\{1,2,3\},\{2,3,4\}\}$. The blue subposet of the GFO for $\hyper{H}\cup \{2,3\}$ is the replacement of the blue subposet of the GFO for ${\bf H}$. Since $\{2,3\}$ is now connected, the good pair $(2,3)$ of the construct $4(\{1,2\}(3))$ does not get abolished by the $(1,2)$-splitting of the node $\{1,2\}$; the corresponding covering relation is now highlighted in green.}
\label{hassediag5}
\end{figure}
\begin{figure}[H]
\input{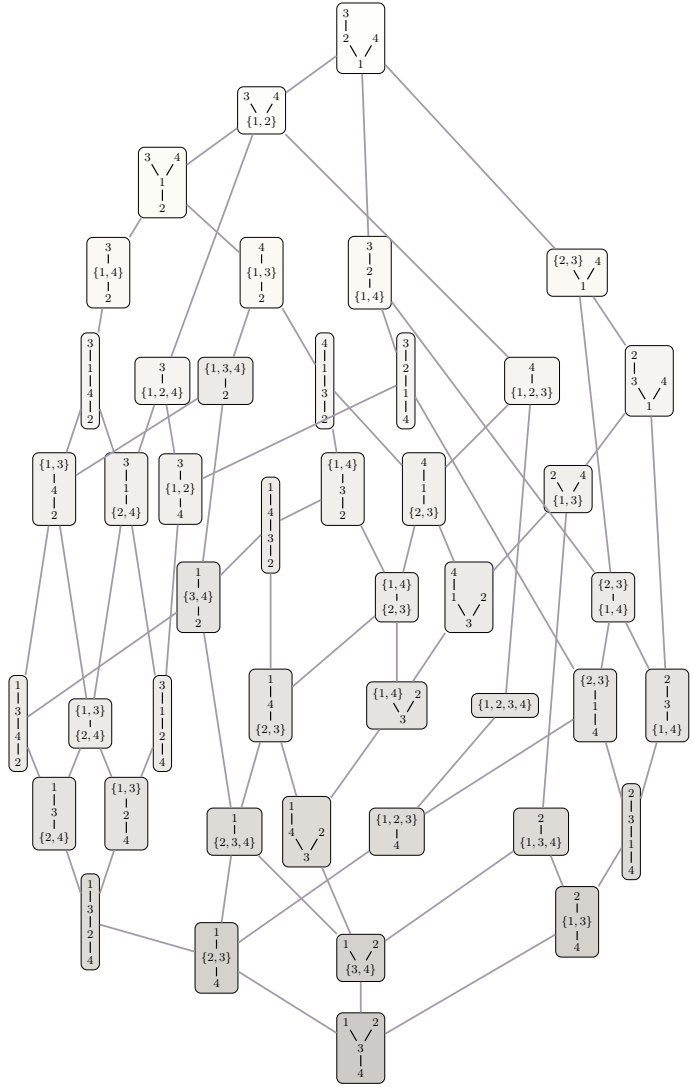}
\caption{The Hasse diagram of the GFO for $\hyper{H}^\dag=\{\{1\},\{2\},\{3\},\{4\},\{2,3\}, \{3,1\}, \{1,4\}\}$ of Example \ref{Exple27}, which is not hereditarily ordered}
\label{hassediag4}
\end{figure}
\printbibliography
 
\end{document}